\documentclass[11pt,twoside,a4paper]{article}
\usepackage{makeidx}
\usepackage{graphicx}
\usepackage{amscd}
\usepackage{amsmath}
\usepackage{amssymb}
\usepackage{latexsym}
\usepackage[latin1]{inputenc}
\usepackage[T1]{fontenc}
\usepackage[english]{babel}
\usepackage[T1]{fontenc}
\input{xy}
\xyoption{all}

\newcommand{\tun}{\begin{picture}(5,0)(-2,-1)
\put(0,0){\circle*{2}}
\end{picture}}

\newcommand{\tdeux}{\begin{picture}(7,7)(0,-1)
\put(3,0){\circle*{2}}
\put(3,0){\line(0,1){5}}
\put(3,5){\circle*{2}}
\end{picture}}

\newcommand{\ttroisun}{\begin{picture}(15,8)(-5,-1)
\put(3,0){\circle*{2}}
\put(-0.65,0){$\vee$}
\put(6,7){\circle*{2}}
\put(0,7){\circle*{2}}
\end{picture}}
\newcommand{\ttroisdeux}{\begin{picture}(5,12)(-2,-1)
\put(0,0){\circle*{2}}
\put(0,0){\line(0,1){5}}
\put(0,5){\circle*{2}}
\put(0,5){\line(0,1){5}}
\put(0,10){\circle*{2}}
\end{picture}}

\newcommand{\tquatreun}{\begin{picture}(15,12)(-5,-1)
\put(3,0){\circle*{2}}
\put(-0.65,0){$\vee$}
\put(6,7){\circle*{2}}
\put(0,7){\circle*{2}}
\put(3,7){\circle*{2}}
\put(3,0){\line(0,1){7}}
\end{picture}}
\newcommand{\tquatredeux}{\begin{picture}(15,18)(-5,-1)
\put(3,0){\circle*{2}}
\put(-0.65,0){$\vee$}
\put(6,7){\circle*{2}}
\put(0,7){\circle*{2}}
\put(0,14){\circle*{2}}
\put(0,7){\line(0,1){7}}
\end{picture}}
\newcommand{\tquatretrois}{\begin{picture}(15,18)(-5,-1)
\put(3,0){\circle*{2}}
\put(-0.65,0){$\vee$}
\put(6,7){\circle*{2}}
\put(0,7){\circle*{2}}
\put(6,14){\circle*{2}}
\put(6,7){\line(0,1){7}}
\end{picture}}
\newcommand{\tquatrequatre}{\begin{picture}(15,18)(-5,-1)
\put(3,5){\circle*{2}}
\put(-0.65,5){$\vee$}
\put(6,12){\circle*{2}}
\put(0,12){\circle*{2}}
\put(3,0){\circle*{2}}
\put(3,0){\line(0,1){5}}
\end{picture}}
\newcommand{\tquatrecinq}{\begin{picture}(9,19)(-2,-1)
\put(0,0){\circle*{2}}
\put(0,0){\line(0,1){5}}
\put(0,5){\circle*{2}}
\put(0,5){\line(0,1){5}}
\put(0,10){\circle*{2}}
\put(0,10){\line(0,1){5}}
\put(0,15){\circle*{2}}
\end{picture}}

\newcommand{\tcinqun}{\begin{picture}(20,8)(-5,-1)
\put(3,0){\circle*{2}}
\put(-0.5,0){$\vee$}
\put(6,7){\circle*{2}}
\put(0,7){\circle*{2}}
\put(3,0){\line(2,1){10}}
\put(3,0){\line(-2,1){10}}
\put(-7,5){\circle*{2}}
\put(13,5){\circle*{2}}
\end{picture}}
\newcommand{\tcinqdeux}{\begin{picture}(15,14)(-5,-1)
\put(3,0){\circle*{2}}
\put(-0.65,0){$\vee$}
\put(6,7){\circle*{2}}
\put(0,7){\circle*{2}}
\put(3,7){\circle*{2}}
\put(3,0){\line(0,1){7}}
\put(0,7){\line(0,1){7}}
\put(0,14){\circle*{2}}
\end{picture}}
\newcommand{\tcinqtrois}{\begin{picture}(15,15)(-5,-1)
\put(3,0){\circle*{2}}
\put(-0.65,0){$\vee$}
\put(6,7){\circle*{2}}
\put(0,7){\circle*{2}}
\put(3,7){\circle*{2}}
\put(3,0){\line(0,1){7}}
\put(3,7){\line(0,1){7}}
\put(3,14){\circle*{2}}
\end{picture}}
\newcommand{\tcinqquatre}{\begin{picture}(15,14)(-5,-1)
\put(3,0){\circle*{2}}
\put(-0.65,0){$\vee$}
\put(6,7){\circle*{2}}
\put(0,7){\circle*{2}}
\put(3,7){\circle*{2}}
\put(3,0){\line(0,1){7}}
\put(6,7){\line(0,1){7}}
\put(6,14){\circle*{2}}
\end{picture}}
\newcommand{\tcinqcinq}{\begin{picture}(15,19)(-5,-1)
\put(3,0){\circle*{2}}
\put(-0.65,0){$\vee$}
\put(6,7){\circle*{2}}
\put(0,7){\circle*{2}}
\put(6,14){\circle*{2}}
\put(6,7){\line(0,1){7}}
\put(0,14){\circle*{2}}
\put(0,7){\line(0,1){7}}
\end{picture}}
\newcommand{\tcinqsix}{\begin{picture}(15,20)(-7,-1)
\put(3,0){\circle*{2}}
\put(-0.65,0){$\vee$}
\put(6,7){\circle*{2}}
\put(0,7){\circle*{2}}
\put(-3.65,7){$\vee$}
\put(3,14){\circle*{2}}
\put(-3,14){\circle*{2}}
\end{picture}}
\newcommand{\tcinqsept}{\begin{picture}(15,8)(-5,-1)
\put(3,0){\circle*{2}}
\put(-0.65,0){$\vee$}
\put(6,7){\circle*{2}}
\put(0,7){\circle*{2}}
\put(2.35,7){$\vee$}
\put(3,14){\circle*{2}}
\put(9,14){\circle*{2}}
\end{picture}}
\newcommand{\tcinqhuit}{\begin{picture}(15,26)(-5,-1)
\put(3,0){\circle*{2}}
\put(-0.65,0){$\vee$}
\put(6,7){\circle*{2}}
\put(0,7){\circle*{2}}
\put(0,14){\circle*{2}}
\put(0,7){\line(0,1){7}}
\put(0,21){\circle*{2}}
\put(0,14){\line(0,1){7}}
\end{picture}}
\newcommand{\tcinqneuf}{\begin{picture}(15,26)(-5,-1)
\put(3,0){\circle*{2}}
\put(-0.65,0){$\vee$}
\put(6,7){\circle*{2}}
\put(0,7){\circle*{2}}
\put(6,14){\circle*{2}}
\put(6,7){\line(0,1){7}}
\put(6,21){\circle*{2}}
\put(6,14){\line(0,1){7}}
\end{picture}}
\newcommand{\tcinqdix}{\begin{picture}(15,19)(-5,-1)
\put(3,5){\circle*{2}}
\put(-0.5,5){$\vee$}
\put(6,12){\circle*{2}}
\put(0,12){\circle*{2}}
\put(3,0){\circle*{2}}
\put(3,0){\line(0,1){12}}
\put(3,12){\circle*{2}}
\end{picture}}
\newcommand{\tcinqonze}{\begin{picture}(15,26)(-5,-1)
\put(3,5){\circle*{2}}
\put(-0.65,5){$\vee$}
\put(6,12){\circle*{2}}
\put(0,12){\circle*{2}}
\put(3,0){\circle*{2}}
\put(3,0){\line(0,1){5}}
\put(0,12){\line(0,1){7}}
\put(0,19){\circle*{2}}
\end{picture}}
\newcommand{\tcinqdouze}{\begin{picture}(15,26)(-5,-1)
\put(3,5){\circle*{2}}
\put(-0.65,5){$\vee$}
\put(6,12){\circle*{2}}
\put(0,12){\circle*{2}}
\put(3,0){\circle*{2}}
\put(3,0){\line(0,1){5}}
\put(6,12){\line(0,1){7}}
\put(6,19){\circle*{2}}
\end{picture}}
\newcommand{\tcinqtreize}{\begin{picture}(5,26)(-2,-1)
\put(0,0){\circle*{2}}
\put(0,0){\line(0,1){7}}
\put(0,7){\circle*{2}}
\put(0,7){\line(0,1){7}}
\put(0,14){\circle*{2}}
\put(-3.65,14){$\vee$}
\put(-3,21){\circle*{2}}
\put(3,21){\circle*{2}}
\end{picture}}
\newcommand{\tcinqquatorze}{\begin{picture}(9,26)(-5,-1)
\put(0,0){\circle*{2}}
\put(0,0){\line(0,1){5}}
\put(0,5){\circle*{2}}
\put(0,5){\line(0,1){5}}
\put(0,10){\circle*{2}}
\put(0,10){\line(0,1){5}}
\put(0,15){\circle*{2}}
\put(0,15){\line(0,1){5}}
\put(0,20){\circle*{2}}
\end{picture}}

\newcommand{\tdun}[1]{\begin{picture}(10,5)(-2,-1)
\put(0,0){\circle*{2}}
\put(3,-2){\tiny #1}
\end{picture}}

\newcommand{\tddeux}[2]{\begin{picture}(12,5)(0,-1)
\put(3,0){\circle*{2}}
\put(3,0){\line(0,1){5}}
\put(3,5){\circle*{2}}
\put(6,-2){\tiny #1}
\put(6,3){\tiny #2}
\end{picture}}

\newcommand{\tdtroisun}[3]{\begin{picture}(20,12)(-5,-1)
\put(3,0){\circle*{2}}
\put(-0.65,0){$\vee$}
\put(6,7){\circle*{2}}
\put(0,7){\circle*{2}}
\put(5,-2){\tiny #1}
\put(9,5){\tiny #2}
\put(-5,5){\tiny #3}
\end{picture}}

\newcommand{\tdtroisdeux}[3]{\begin{picture}(12,14)(-2,-1)
\put(0,0){\circle*{2}}
\put(0,0){\line(0,1){5}}
\put(0,5){\circle*{2}}
\put(0,5){\line(0,1){5}}
\put(0,10){\circle*{2}}
\put(3,-2){\tiny #1}
\put(3,3){\tiny #2}
\put(3,9){\tiny #3}
\end{picture}}

\newcommand{\tdquatreun}[4]{\begin{picture}(20,12)(-5,-1)
\put(3,0){\circle*{2}}
\put(-0.6,0){$\vee$}
\put(6,7){\circle*{2}}
\put(0,7){\circle*{2}}
\put(3,7){\circle*{2}}
\put(3,0){\line(0,1){7}}
\put(5,-2){\tiny #1}
\put(8.5,5){\tiny #2}
\put(1,10){\tiny #3}
\put(-5,5){\tiny #4}
\end{picture}}

\newcommand{\tdquatredeux}[4]{\begin{picture}(20,20)(-5,-1)
\put(3,0){\circle*{2}}
\put(-.65,0){$\vee$}
\put(6,7){\circle*{2}}
\put(0,7){\circle*{2}}
\put(0,14){\circle*{2}}
\put(0,7){\line(0,1){7}}
\put(5,-2){\tiny #1}
\put(9,5){\tiny #2}
\put(-5,5){\tiny #3}
\put(-5,12){\tiny #4}
\end{picture}}

\newcommand{\tdquatretrois}[4]{\begin{picture}(20,20)(-5,-1)
\put(3,0){\circle*{2}}
\put(-.65,0){$\vee$}
\put(6,7){\circle*{2}}
\put(0,7){\circle*{2}}
\put(6,14){\circle*{2}}
\put(6,7){\line(0,1){7}}
\put(5,-2){\tiny #1}
\put(9,5){\tiny #2}
\put(-5,5){\tiny #4}
\put(9,12){\tiny #3}
\end{picture}}

\newcommand{\tdquatrequatre}[4]{\begin{picture}(20,14)(-5,-1)
\put(3,5){\circle*{2}}
\put(-.65,5){$\vee$}
\put(6,12){\circle*{2}}
\put(0,12){\circle*{2}}
\put(3,0){\circle*{2}}
\put(3,0){\line(0,1){5}}
\put(6,-3){\tiny #1}
\put(6,4){\tiny #2}
\put(9,12){\tiny #3}
\put(-5,12){\tiny #4}
\end{picture}}

\newcommand{\tdquatrecinq}[4]{\begin{picture}(12,19)(-2,-1)
\put(0,0){\circle*{2}}
\put(0,0){\line(0,1){5}}
\put(0,5){\circle*{2}}
\put(0,5){\line(0,1){5}}
\put(0,10){\circle*{2}}
\put(0,10){\line(0,1){5}}
\put(0,15){\circle*{2}}
\put(3,-2){\tiny #1}
\put(3,3){\tiny #2}
\put(3,9){\tiny #3}
\put(3,14){\tiny #4}
\end{picture}}

\newcommand{\tquatredeuxa}{\begin{picture}(15,18)(-5,-1)
\put(3,0){\circle*{2}}
\put(-0.2,0.2){$\vee$}
\put(6,7){\circle*{2}}
\put(0,7){\circle*{2}}
\put(0,14){\circle*{2}}
\put(0,7){\line(0,1){7}}
\put(-4,4){\line(1,0){7}}
\end{picture}}

\newcommand{\tquatredeuxb}{\begin{picture}(15,18)(-5,-1)
\put(3,0){\circle*{2}}
\put(-0.2,0.2){$\vee$}
\put(6,7){\circle*{2}}
\put(0,7){\circle*{2}}
\put(0,14){\circle*{2}}
\put(0,7){\line(0,1){7}}
\put(-4,11){\line(1,0){7}}
\end{picture}}

\newcommand{\tquatredeuxc}{\begin{picture}(15,18)(-5,-1)
\put(3,0){\circle*{2}}
\put(-0.2,0.2){$\vee$}
\put(6,7){\circle*{2}}
\put(0,7){\circle*{2}}
\put(0,14){\circle*{2}}
\put(0,7){\line(0,1){7}}
\put(3,4){\line(1,0){7}}
\end{picture}}

\newcommand{\tquatredeuxd}{\begin{picture}(15,18)(-5,-1)
\put(3,0){\circle*{2}}
\put(-0.2,0.2){$\vee$}
\put(6,7){\circle*{2}}
\put(0,7){\circle*{2}}
\put(0,14){\circle*{2}}
\put(0,7){\line(0,1){7}}
\put(-4,4){\line(1,0){7}}
\put(-4,11){\line(1,0){7}}
\end{picture}}

\newcommand{\tquatredeuxe}{\begin{picture}(15,18)(-5,-1)
\put(3,0){\circle*{2}}
\put(-0.2,0.2){$\vee$}
\put(6,7){\circle*{2}}
\put(0,7){\circle*{2}}
\put(0,14){\circle*{2}}
\put(0,7){\line(0,1){7}}
\put(-4,4){\line(1,0){7}}
\put(3,4){\line(1,0){7}}
\end{picture}}

\newcommand{\tquatredeuxf}{\begin{picture}(15,18)(-5,-1)
\put(3,0){\circle*{2}}
\put(-0.2,0.2){$\vee$}
\put(6,7){\circle*{2}}
\put(0,7){\circle*{2}}
\put(0,14){\circle*{2}}
\put(0,7){\line(0,1){7}}
\put(-4,11){\line(1,0){7}}
\put(3,4){\line(1,0){7}}
\end{picture}}

\newcommand{\tquatredeuxg}{\begin{picture}(15,18)(-5,-1)
\put(3,0){\circle*{2}}
\put(-0.2,0.2){$\vee$}
\put(6,7){\circle*{2}}
\put(0,7){\circle*{2}}
\put(0,14){\circle*{2}}
\put(0,7){\line(0,1){7}}
\put(-4,4){\line(1,0){7}}
\put(-4,11){\line(1,0){7}}
\put(3,4){\line(1,0){7}}
\end{picture}}

\newcommand{\tdquatredeuxa}[4]{\begin{picture}(15,18)(-5,-1)
\put(3,0){\circle*{2}}
\put(-0.2,0.2){$\vee$}
\put(6,7){\circle*{2}}
\put(0,7){\circle*{2}}
\put(0,14){\circle*{2}}
\put(0,7){\line(0,1){7}}
\put(-4,4){\line(1,0){7}}
\put(5,-2){\tiny #1}
\put(9,5){\tiny #2}
\put(-5,5){\tiny #3}
\put(-5,12){\tiny #4}
\end{picture}}

\newcommand{\tdquatredeuxb}[4]{\begin{picture}(15,18)(-5,-1)
\put(3,0){\circle*{2}}
\put(-0.2,0.2){$\vee$}
\put(6,7){\circle*{2}}
\put(0,7){\circle*{2}}
\put(0,14){\circle*{2}}
\put(0,7){\line(0,1){7}}
\put(-4,11){\line(1,0){7}}
\put(5,-2){\tiny #1}
\put(9,5){\tiny #2}
\put(-5,5){\tiny #3}
\put(-5,12){\tiny #4}
\end{picture}}

\newcommand{\tdquatredeuxc}[4]{\begin{picture}(15,18)(-5,-1)
\put(3,0){\circle*{2}}
\put(-0.2,0.2){$\vee$}
\put(6,7){\circle*{2}}
\put(0,7){\circle*{2}}
\put(0,14){\circle*{2}}
\put(0,7){\line(0,1){7}}
\put(3,4){\line(1,0){7}}
\put(5,-2){\tiny #1}
\put(9,5){\tiny #2}
\put(-5,5){\tiny #3}
\put(-5,12){\tiny #4}
\end{picture}}

\newcommand{\tdquatredeuxd}[4]{\begin{picture}(15,18)(-5,-1)
\put(3,0){\circle*{2}}
\put(-0.2,0.2){$\vee$}
\put(6,7){\circle*{2}}
\put(0,7){\circle*{2}}
\put(0,14){\circle*{2}}
\put(0,7){\line(0,1){7}}
\put(-4,4){\line(1,0){7}}
\put(-4,11){\line(1,0){7}}
\put(5,-2){\tiny #1}
\put(9,5){\tiny #2}
\put(-5,5){\tiny #3}
\put(-5,12){\tiny #4}
\end{picture}}

\newcommand{\tdquatredeuxe}[4]{\begin{picture}(15,18)(-5,-1)
\put(3,0){\circle*{2}}
\put(-0.2,0.2){$\vee$}
\put(6,7){\circle*{2}}
\put(0,7){\circle*{2}}
\put(0,14){\circle*{2}}
\put(0,7){\line(0,1){7}}
\put(-4,4){\line(1,0){7}}
\put(3,4){\line(1,0){7}}
\put(5,-2){\tiny #1}
\put(9,5){\tiny #2}
\put(-5,5){\tiny #3}
\put(-5,12){\tiny #4}
\end{picture}}

\newcommand{\tdquatredeuxf}[4]{\begin{picture}(15,18)(-5,-1)
\put(3,0){\circle*{2}}
\put(-0.2,0.2){$\vee$}
\put(6,7){\circle*{2}}
\put(0,7){\circle*{2}}
\put(0,14){\circle*{2}}
\put(0,7){\line(0,1){7}}
\put(-4,11){\line(1,0){7}}
\put(3,4){\line(1,0){7}}
\put(5,-2){\tiny #1}
\put(9,5){\tiny #2}
\put(-5,5){\tiny #3}
\put(-5,12){\tiny #4}
\end{picture}}

\newcommand{\tdquatredeuxg}[4]{\begin{picture}(15,18)(-5,-1)
\put(3,0){\circle*{2}}
\put(-0.2,0.2){$\vee$}
\put(6,7){\circle*{2}}
\put(0,7){\circle*{2}}
\put(0,14){\circle*{2}}
\put(0,7){\line(0,1){7}}
\put(-4,4){\line(1,0){7}}
\put(-4,11){\line(1,0){7}}
\put(3,4){\line(1,0){7}}
\put(5,-2){\tiny #1}
\put(9,5){\tiny #2}
\put(-5,5){\tiny #3}
\put(-5,12){\tiny #4}
\end{picture}}

\newcommand{\addeux}[1]{\begin{picture}(12,5)(0,-1)
\put(3,0){\circle*{2}}
\put(3,0){\line(0,1){5}}
\put(3,5){\circle*{2}}
\put(4,1){\tiny #1}
\end{picture}}

\newcommand{\adtroisun}[2]{\begin{picture}(20,12)(-5,-1)
\put(3,0){\circle*{2}}
\put(-0.65,0){$\vee$}
\put(6,7){\circle*{2}}
\put(0,7){\circle*{2}}
\put(6,1){\tiny #1}
\put(-3,1){\tiny #2}
\end{picture}}

\newcommand{\adtroisdeux}[2]{\begin{picture}(12,14)(-2,-1)
\put(0,0){\circle*{2}}
\put(0,0){\line(0,1){5}}
\put(0,5){\circle*{2}}
\put(0,5){\line(0,1){5}}
\put(0,10){\circle*{2}}
\put(2,1){\tiny #1}
\put(2,6){\tiny #2}
\end{picture}}

\newcommand{\adquatreun}[3]{\begin{picture}(25,12)(-5,-1)
\put(5,0){\circle*{2}}
\put(15,10){\circle*{2}}
\put(-5,10){\circle*{2}}
\put(5,10){\circle*{2}}
\put(5,0){\line(0,1){10}}
\put(5,0){\line(-1,1){10}}
\put(5,0){\line(1,1){10}}
\put(-5,1){\tiny #1}
\put(1,4){\tiny #2}
\put(11,1){\tiny #3}
\end{picture}}

\newcommand{\adquatredeux}[3]{\begin{picture}(20,20)(-5,-1)
\put(3,0){\circle*{2}}
\put(-.65,0){$\vee$}
\put(6,7){\circle*{2}}
\put(0,7){\circle*{2}}
\put(0,14){\circle*{2}}
\put(0,7){\line(0,1){7}}
\put(6,1){\tiny #1}
\put(-3,1){\tiny #2}
\put(-5,9){\tiny #3}
\end{picture}}

\newcommand{\adquatretrois}[3]{\begin{picture}(20,20)(-5,-1)
\put(3,0){\circle*{2}}
\put(-.65,0){$\vee$}
\put(6,7){\circle*{2}}
\put(0,7){\circle*{2}}
\put(6,14){\circle*{2}}
\put(6,7){\line(0,1){7}}
\put(6,1){\tiny #1}
\put(-3,1){\tiny #2}
\put(7,9){\tiny #3}
\end{picture}}

\newcommand{\adquatrequatre}[3]{\begin{picture}(20,14)(-5,-1)
\put(3,5){\circle*{2}}
\put(-.65,5){$\vee$}
\put(6,12){\circle*{2}}
\put(0,12){\circle*{2}}
\put(3,0){\circle*{2}}
\put(3,0){\line(0,1){5}}
\put(4,1){\tiny #1}
\put(-3,6){\tiny #2}
\put(5,6){\tiny #3}
\end{picture}}

\newcommand{\adquatrecinq}[3]{\begin{picture}(12,19)(-2,-1)
\put(0,0){\circle*{2}}
\put(0,0){\line(0,1){5}}
\put(0,5){\circle*{2}}
\put(0,5){\line(0,1){5}}
\put(0,10){\circle*{2}}
\put(0,10){\line(0,1){5}}
\put(0,15){\circle*{2}}
\put(2,1){\tiny #1}
\put(2,6){\tiny #2}
\put(2,11){\tiny #3}
\end{picture}}

\newcommand{\adcinqsix}[4]{\begin{picture}(15,8)(-5,-1)
\put(3,0){\circle*{2}}
\put(-0.65,0){$\vee$}
\put(6,7){\circle*{2}}
\put(0,7){\circle*{2}}
\put(-3.65,7){$\vee$}
\put(-3,14){\circle*{2}}
\put(3,14){\circle*{2}}
\put(5,1){\tiny #2}
\put(-6,9){\tiny #3}
\put(-3,1){\tiny #1}
\put(2,9){\tiny #4}
\end{picture}}

\newcommand{\adcinqsept}[4]{\begin{picture}(15,8)(-5,-1)
\put(3,0){\circle*{2}}
\put(-0.65,0){$\vee$}
\put(6,7){\circle*{2}}
\put(0,7){\circle*{2}}
\put(2.35,7){$\vee$}
\put(3,14){\circle*{2}}
\put(9,14){\circle*{2}}
\put(-3,1){\tiny #1}
\put(5,1){\tiny #2}
\put(9,9){\tiny #3}
\put(0,9){\tiny #4}
\end{picture}}

\newcommand{\tdelta}{\tilde{\Delta}}
\newcommand{\mmodels}{\mid \hspace{-.2mm} \models}
\newcommand{\eps}{\epsilon}
\newcommand{\Sh}{\mathbf{Sh}}
\newcommand{\Csh}{\mathbf{Csh}}
\newcommand{\FQSym}{\mathbf{FQSym}}
\newcommand{\WQSyms}{\mathbf{WQSym^{\ast}}}
\newcommand{\WQSym}{\mathbf{WQSym}}
\newcommand{\hh}{\mathbf{H}}
\newcommand{\cc}{\mathbf{C}}
\newcommand{\D}{\mathcal{D}}

\newcommand{\FF}{\mathbb{F}}

\title{Preordered forests, packed words and contraction algebras}

\date{}

\author{Anthony Mansuy \\ \\
{\small{\it Laboratoire de Mathématiques, Université de Reims}}\\
\small{{\it Moulin de la Housse - BP 1039 - 51687 REIMS Cedex 2, France}}\\
\small{e-mail : anthony.mansuy@univ-reims.fr}}

\newtheorem{defi}{\indent Definition}
\newtheorem{lemma}[defi]{\indent Lemma}

\newtheorem{theo}[defi]{\indent Theorem}
\newtheorem{prop}[defi]{\indent Proposition}

\newenvironment{proof}{{\bf Proof.}}{\hfill $\Box$}
\usepackage{fancyhdr}
\def\shuff#1#2{\mathbin{
      \hbox{\vbox{
        \hbox{\vrule
              \hskip#2
              \vrule height#1 width 0pt
               }%
        \hrule}%
             \vbox{
        \hbox{\vrule
              \hskip#2
              \vrule height#1 width 0pt
               \vrule }%
        \hrule}%
}}}

\def\shuffl{{\mathchoice{\shuff{7pt}{3.5pt}}%
                        {\shuff{6pt}{3pt}}%
                        {\shuff{4pt}{2pt}}%
                        {\shuff{3pt}{1.5pt}}}}%
\def\shuffle{\, \shuffl \,}

\begin{document}
\maketitle

\textbf{Abstract.} We introduce the notions of preordered and heap-preordered forests, generalizing the construction of ordered and heap-ordered forests. We prove that the algebras of preordered and heap-preordered forests are Hopf for the cut coproduct, and we construct a Hopf morphism to the Hopf algebra of packed words. Moreover, we define another coproduct on the preordered forests given by the contraction of edges. Finally, we give a combinatorial description of morphims defined on Hopf algebras of forests with values in the Hopf algebras of shuffes or quasi-shuffles.
\\

\textbf{Résumé.} Nous introduisons les notions de forêts préordonnées et préordonnées en tas, généralisant les constructions des forêts ordonnées et ordonnées en tas. On démontre que les algèbres des forêts préordonnées et préordonnées en tas sont des algèbres de Hopf pour le coproduit de coupes et on construit un morphisme d'algèbres de Hopf dans l'algèbre des mots tassés. D'autre part, nous définissons un autre coproduit sur les forêts préordonnées donné par la contraction d'arêtes. Enfin, nous donnons une description combinatoire de morphismes définis sur des algèbres de Hopf de forêts et à valeurs dans les algèbres de Hopf de battages et de battages contractants.
\\



\tableofcontents

\section*{Introduction}

The Connes-Kreimer Hopf algebra of rooted forests $\mathbf{H}_{CK}$ is introduced and studied in \cite{Connes98,Moerdijk01}. This commutative, noncocommutative Hopf algebra is used to study a problem of Renormalisation
in Quantum Field Theory, as explained in \cite{Connes00,Connes01}. The coproduct is given by admissible cuts. We denote by $\mathbf{H}_{CK}^{\mathcal{D}}$ the Hopf algebra of rooted trees, where the vertices are decorated by decorations belonging to the set $ \mathcal{D}$. A noncommutative version, the Hopf algebra $\mathbf{H}_{NCK}$ of planar rooted forests, is introduced in \cite{FoissyI02,Holtkamp03}. When the vertices are given a total order, we obtain the Hopf algebra of ordered forests $\mathbf{H}_{o}$ and, adding an increasing condition, we obtain the Hopf subalgebra of heap-ordered forests $\mathbf{H}_{ho}$ (see \cite{Foissy10,Grossman90}).

Moreover, the Hopf algebra $\FQSym$ of free quasi-symmetric functions is introduced  in \cite{Duchamp00,Malvenuto05}. L. Foissy and J. Unterberger prove in \cite{Foissy10} that there exists a Hopf algebra morphism from $\mathbf{H}_{o}$ to $\FQSym$ and that its restriction to $\mathbf{H}_{ho}$ is an isomorphism of Hopf algebras.

In this text, we introduce the notion of preordered forests. A preorder is a binary reflexive and transitive relation. A preordered forest is a rooted forest with a total preorder on its vertices. We prove that the algebra of preordered forests $\mathbf{H}_{po}$ is a Hopf algebra for the cut coproduct. With an increasing condition, we define the algebra of heap-preordered trees $\mathbf{H}_{hpo}$ and we prove that $\mathbf{H}_{hpo}$ is a Hopf subalgebra of $\mathbf{H}_{po}$.

In \cite{Novelli06}, J.-C. Novelli and J.-Y. Thibon construct a generalization of $\FQSym$: the Hopf algebra $\WQSyms$ of free packed words. We prove a result similar to that of L. Foissy and J. Unterberger, by substituting the ordered forests by the preordered forests and the quasi-symmetric functions by the packed words. More precisely, we prove that there exists a Hopf algebra morphism from $\mathbf{H}_{po}$ to $\WQSyms$. In addition, we prove that its restriction to $\mathbf{H}_{hpo}$ is an injection of Hopf algebras.

Afterwards, we study a another coproduct called in this paper the contraction coproduct. In \cite{Calaque11}, D. Calaque, K. Ebrahimi-Fard and D. Manchon define this coproduct in a commutative case, on a quotient $\mathbf{C}_{CK}$ of $\mathbf{H}_{CK}$ (see also \cite{Manchon08}). We give a decorated version $\mathbf{C}_{CK}^{\mathcal{D}}$ of $\mathbf{C}_{CK}$. We define two operations $\curlyvee$ and $\rhd$ on the vector space $\mathbf{T}^{\mathcal{D}}_{CK}$ spanned by the trees of $\mathbf{C}_{CK}^{\mathcal{D}}$. We prove that $(\mathbf{T}^{\mathcal{D}}_{CK}, \curlyvee , \rhd)$ is a commutative prelie algebra, that is to say that $(A,\curlyvee)$ is a commutative algebra, $(A, \rhd)$ is a prelie algebra and with the following relation: for all $x,y,z \in \mathbf{T}^{\mathcal{D}}_{CK}$,
$$ x \rhd (y \curlyvee z) = (x \rhd y) \curlyvee z + (x \rhd z) \curlyvee y .$$
We prove that $(\mathbf{T}^{\mathcal{D}}_{CK}, \curlyvee , \rhd)$ is generated as commutative prelie algebra by the trees $\addeux{$d$}$, $d \in \mathcal{D}$.

We construct a noncommutative version of $\mathbf{C}_{CK}$. For this, we consider quotients of $\mathbf{H}_{NCK}$, $\mathbf{H}_{ho}$, $\mathbf{H}_{o}$, $\mathbf{H}_{hpo}$, $ \mathbf{H}_{po}$, denoted respectively $\mathbf{C}_{NCK}$, $\mathbf{C}_{ho}$, $\mathbf{C}_{o}$, $\mathbf{C}_{hpo}$, $ \mathbf{C}_{po}$, and we define on these quotients a contraction coproduct. We prove that $\mathbf{C}_{ho}, \mathbf{C}_{o}, \mathbf{C}_{hpo}, \mathbf{C}_{po}$ are Hopf algebras and that $\mathbf{C}_{NCK}$ is a left comodule of the Hopf algebra $\mathbf{C}_{ho}$.

Finally, we study the Hopf algebra morphisms from $\mathbf{H}_{CK}^{\mathcal{D}}$ or $\mathbf{C}_{CK}^{\mathcal{D}}$ to the Hopf algebra $\Sh^{\mathcal{D}}$ of shuffles or the Hopf algebra $\Csh^{\mathcal{D}}$ of quasi-shuffles (see \cite{Hoffman00}). We give a combinatorial description of these morphisms in each case. In particular, we note that, in the description of the morphism from $\mathbf{H}_{CK}^{\mathcal{D}}$ to $\Sh^{\mathcal{D}}$ or $\Csh^{\mathcal{D}}$, the contraction coproduct and the preordered forests appear naturally.
\\

This text is organized as follows: the first section is devoted to recalls about the Hopf algebras, for the cut coproduct, of rooted forests, planar forests and ordered and heap-ordered forests. We give recalls on the Hopf algebras of words in the second section. We define the Hopf algebra of permutations and packed words and we deduce the construction of $\Sh^{\mathcal{D}}$ and $\Csh^{\mathcal{D}}$. In section three, we define the algebras $\mathbf{H}_{po}$ and $\mathbf{H}_{hpo}$ of preordered and heap-preordered forests and we prove that these are Hopf algebras. The contraction coproduct, is introduced in the section four. We describe a commutative case and we study an insertion operation. We give a noncommutative version using ordered and preordered forests. The last section deals with Hopf algebra morphisms from $\mathbf{H}_{CK}^{\mathcal{D}}$ or $\mathbf{C}_{CK}^{\mathcal{D}}$ to $\Sh^{\mathcal{D}}$ or $\Csh^{\mathcal{D}}$. We give a combinatorial description of these morphisms in each case.
\\

{\bf Acknowledgment.} {I would like to thank my advisor Loïc Foissy for stimulating discussions and his support during my research.}
\\

{\bf Notations.} {
\begin{enumerate}
\item We shall denote by $ \mathbb{K} $ a commutative field of characteristic zero. Every vector space, algebra, coalgebra, etc, will be taken over $ \mathbb{K} $. Given a set $ X $, we denote by $ \mathbb{K} \left( X \right) $ the vector space spanned by $ X $.
\item Let $n$ be an integer. We denote by $\Sigma_{n}$ the symmetric group of order $n$ ($\Sigma_{0} = \{ 1 \}$) and $\Sigma$ the disjoint union of $\Sigma_{n}$ for all $n \geq 0$.
\item Let $ (A, \Delta , \varepsilon) $ be a counitary coalgebra. Let $ 1 \in A $, non zero, such that $ \Delta (1) = 1 \otimes 1 $. We then define the noncounitary coproduct:
\begin{eqnarray*}
\tdelta : \left\lbrace 
\begin{array}{rcl}
Ker(\varepsilon) & \rightarrow & Ker(\varepsilon) \otimes Ker(\varepsilon) ,\\
a & \mapsto & \Delta (a ) - a \otimes 1 - 1 \otimes a .
\end{array} \right. 
\end{eqnarray*}
\end{enumerate}
}

\section{Recalls on the Hopf algebras of forests}

\subsection{The Connes-Kreimer Hopf algebra of rooted trees}

We briefly recall the construction of the Connes-Kreimer Hopf algebra of rooted trees \cite{Connes98}. A {\it rooted tree} is a finite graph, connected, without loops, with a distinguished vertex called the {\it root} \cite{Stanley97}. We denote by $1$ the empty rooted tree. If $T$ is a rooted tree, we denote by $R_{T}$ the root of $T$. A {\it rooted forest} is a finite graph $ F $ such that any connected component of $ F $ is a rooted tree. The \textit{length} of a forest $ F $, denoted $ l(F) $, is the number of connected components of $ F $. The set of vertices of the rooted forest $ F $ is denoted by $V(F)$. The {\it vertices degree} of a forest $ F $, denoted $ \left| F \right|_{v} $, is the number of its vertices. The set of edges of the rooted forest $ F $ is denoted by $ E(F) $. The {\it edges degree} of a forest $ F $, denoted $ \left| F \right|_{e} $, is the number of its edges.\\

{\bf Remark.} {Let $F$ be a rooted forest. Then $ \left| F \right|_{v} = \left| F \right|_{e} + l(F) $.}
\\

{\bf Examples.} {\begin{enumerate}
\item Rooted trees of vertices degree $ \leq 5 $:
$$ 1,\tun,\tdeux,\ttroisun,\ttroisdeux,\tquatreun,\tquatredeux,\tquatrequatre,\tquatrecinq,\tcinqun,\tcinqdeux,\tcinqcinq,\tcinqsix,\tcinqhuit,\tcinqdix,\tcinqonze,\tcinqtreize,\tcinqquatorze . $$
\item Rooted forests of vertices degree $ \leq 4 $:
$$ 1,\tun,\tun\tun,\tdeux,\tun\tun\tun,\tdeux\tun,\ttroisun,\ttroisdeux,\tun\tun\tun\tun,\tdeux\tun\tun,\tdeux\tdeux,\ttroisun\tun,\ttroisdeux\tun,\tquatreun,\tquatredeux,\tquatrequatre,\tquatrecinq . $$
\end{enumerate}}

Let $\mathcal{D}$ be a nonempty set. A rooted forest with its vertices decorated by $\mathcal{D}$ is a couple $(F,d)$ with $F$ a rooted forest and $d : V(F) \rightarrow \mathcal{D}$ a map.
\\

{\bf Examples.} {Rooted trees with their vertices decorated by $\mathcal{D}$ of vertices degree smaller than 4:
$$\tdun{$a$}, a \in \mathcal{D}, \hspace{1cm} \tddeux{$a$}{$b$}, (a,b) \in \mathcal{D}^{2}, \hspace{1cm} \tdtroisun{$a$}{$c$}{$b$}, \tdtroisdeux{$a$}{$b$}{$c$}, (a,b,c) \in \mathcal{D}^{3}$$
$$ \tdquatreun{$a$}{$d$}{$c$}{$b$}, \tdquatredeux{$a$}{$d$}{$b$}{$c$}, \tdquatretrois{$a$}{$c$}{$d$}{$b$}, \tdquatrequatre{$a$}{$b$}{$d$}{$c$}, \tdquatrecinq{$a$}{$b$}{$c$}{$d$} , (a,b,c,d) \in \mathcal{D}^{4}.$$
}

Let $\mathbb{F}_{\hh_{CK}}$ be the set of rooted forests and $\mathbb{F}_{\hh_{CK}}^{\mathcal{D}}$ the set of rooted forests with their vertices decorated by $\mathcal{D}$. We will denote by $\hh_{CK}$ the $\mathbb{K}$-vector space generated by $\mathbb{F}_{\hh_{CK}}$ and by $\hh_{CK}^{\mathcal{D}}$ the $\mathbb{K}$-vector space generated by $\mathbb{F}_{\hh_{CK}}^{\mathcal{D}}$. The set of nonempty rooted trees will be denoted $\mathbb{T}_{\hh_{CK}}$ and the set of nonempty rooted trees with their vertices decorated by $\mathcal{D}$ will be denoted $\mathbb{T}_{\hh_{CK}}^{\mathcal{D}}$. $\hh_{CK}$ and $\hh_{CK}^{\mathcal{D}}$ are algebras: the product is given by the concatenation of rooted forests.
\\

Let $ F $ be a nonempty rooted forest. A {\it subtree} $ T $ of $ F $ is a nonempty connected subgraph of $ F $. A {\it subforest} $ T_{1} \hdots T_{k} $ of $ F $ is a product of disjoint subtrees $ T_{1}, \hdots , T_{k} $ of $ F $. We can give the same definition in the decorated case.
\\

{\bf Examples.} {Consider the tree $ T = \tquatredeux $. Then:
\begin{itemize}
\item The subtrees of $ T $ are $ \tun$ (which appears $4$ times), $\tdeux$ (which appears $3$ times), $\ttroisun$, $\ttroisdeux$ and $\tquatredeux $ (which appear once).
\item The subforests of $T$ are $\tun \tun , \tdeux \tun$ (which appear $6$ times), $ \tun , \tun \tun \tun $ (which appear $4$ times), $\tdeux, \tdeux \tun \tun $ (which appear $3$ times) and $\ttroisun , \ttroisdeux , \tun \tun \tun \tun , \ttroisun \tun , \ttroisdeux \tun , \tquatredeux $ (which appear once).
\end{itemize}
}

Let $ F $ be a rooted forest. The edges of $ F $ are oriented downwards (from the leaves to the roots). If $v,w \in V(F)$, we shall note $v \rightarrow w$ if there is an edge in $ F $ from $v$ to $w$ and $v \twoheadrightarrow w$ if there is an oriented path from $v$ to $w$ in $ F $. By convention, $v \twoheadrightarrow v$ for any $v \in V(F)$.

Let $\boldsymbol{v}$ be a subset of $V(F)$. We shall say that $\boldsymbol{v}$ is an admissible cut of $F$, and we shall write $\boldsymbol{v} \models V(F)$, if $\boldsymbol{v}$ is totally disconnected, that is to say that $v \twoheadrightarrow w \hspace{-.7cm} / \hspace{.7cm}$ for any couple $(v,w)$ of two different elements of $\boldsymbol{v}$. If $\boldsymbol{v} \models V(F)$, we denote by $Lea_{\boldsymbol{v}} (F) $ the rooted subforest of $ F $ obtained by keeping only the vertices above $\boldsymbol{v}$, that is to say $\{ w \in V(F), \: \exists v \in \boldsymbol{v}, \:w \twoheadrightarrow v \}$, and the edges between these vertices. Note that $\boldsymbol{v} \subseteq Lea_{\boldsymbol{v}}(F) $. We denote by $Roo_{\boldsymbol{v}}(F)$ the rooted subforest obtained by keeping the other vertices and the edges between these vertices.

In particular, if $\boldsymbol{v}=\emptyset$, then $Lea_{\boldsymbol{v}} (F) =1$ and $Roo_{\boldsymbol{v}} (F) = F $: this is the {\it empty cut} of $ F $. If $\boldsymbol{v}$ contains all the roots of $ F $, then it contains only the roots of $ F $, $Lea_{\boldsymbol{v}} (F) = F$ and $Roo_{\boldsymbol{v}}(F) = 1$: this is the {\it total cut} of $ F $. We shall write $\boldsymbol{v} \mmodels V( F )$ if $\boldsymbol{v}$ is a nontotal, nonempty admissible cut of $ F $.\\

Connes and Kreimer proved in \cite{Connes98} that $ \hh_{CK} $ is a Hopf algebra. The coproduct is the cut coproduct defined for any rooted forest $ F $ by:
$$\Delta_{\hh_{CK}}(F)=\sum_{\boldsymbol{v} \models V(F)} Lea_{\boldsymbol{v}} (F) \otimes Roo_{\boldsymbol{v}} (F)
=F \otimes 1+1\otimes F+\sum_{\boldsymbol{v} \mmodels V(F)} Lea_{\boldsymbol{v}}(F) \otimes Roo_{\boldsymbol{v}}(F) .$$
For example:
$$\Delta_{\hh_{CK}} (\tquatredeux)=\tquatredeux \otimes 1+1\otimes \tquatredeux+
\tun \otimes \ttroisun+\tdeux \otimes \tdeux+\tun \otimes \ttroisdeux+\tun\tun \otimes \tdeux+\tdeux \tun \otimes \tun.$$
In the same way, we can define a cut coproduct on $\hh_{CK}^{\mathcal{D}}$. With this coproduct, $\hh_{CK}^{\mathcal{D}}$ is also a Hopf algebra. For example:
$$ \Delta_{\hh_{CK}^{\mathcal{D}}} ( \tdquatredeux{$a$}{$d$}{$b$}{$c$} ) = \tdquatredeux{$a$}{$d$}{$b$}{$c$} \otimes 1 + 1 \otimes \tdquatredeux{$a$}{$d$}{$b$}{$c$} + \tdun{$c$} \otimes \tdtroisun{$a$}{$d$}{$b$} + \tddeux{$b$}{$c$} \otimes \tddeux{$a$}{$d$} + \tdun{$d$} \otimes \tdtroisdeux{$a$}{$b$}{$c$} + \tdun{$c$} \tdun{$d$} \otimes \tddeux{$a$}{$b$} + \tddeux{$b$}{$c$} \tdun{$d$} \otimes \tdun{$a$} .$$

$\hh_{CK}$ is graded by the number of vertices. We give some values of the number $f_{n}^{\hh_{CK}}$ of rooted forests of vertices degree $n$:

$$\begin{array}{c|c|c|c|c|c|c|c|c|c|c|c}
n&0&1&2&3&4&5&6&7&8&9&10\\
\hline f_{n}^{\hh_{CK}} &1&1&2&4&9&20&48&115&286&719&1842
\end{array}$$

These is the sequence A000081 in \cite{Sloane}.

\subsection{Hopf algebras of planar trees}

We now recall the construction of the noncommutative generalization of the Connes-Kreimer Hopf algebra \cite{FoissyI02,Holtkamp03}.\\

A {\it planar forest} is a rooted forest $ F $ such that the set of the roots of $ F $ is totally ordered and, for any vertex $v \in V(F)$, the set $\{w \in V(F)\:\mid \:w \rightarrow v\}$ is totally ordered. Planar forests are represented such that the total orders on the set of roots and the sets $\{w \in V(F)\:\mid \:w \rightarrow v\}$ for any $v \in V(F)$ is given from left to right. We denote by $ \mathbb{T}_{\hh_{NCK}} $ the set of nonempty planar trees and $\mathbb{F}_{\hh_{NCK}}$ the set of planar forests.\\

{\bf Examples.} {\begin{enumerate}
\item Planar rooted trees of vertices degree $\leq 5$:
$$\tun,\tdeux,\ttroisun,\ttroisdeux,\tquatreun, \tquatredeux,\tquatretrois,\tquatrequatre,\tquatrecinq,\tcinqun,\tcinqdeux,\tcinqtrois,\tcinqquatre,\tcinqcinq,
\tcinqsix,\tcinqsept,\tcinqhuit,\tcinqneuf,\tcinqdix,\tcinqonze,\tcinqdouze,\tcinqtreize,\tcinqquatorze.$$
\item Planar rooted forests of vertices degree $\leq 4$:
$$1,\tun,\tun\tun,\tdeux,\tun\tun\tun,\tdeux\tun,\tun \tdeux,\ttroisun,\ttroisdeux,\tun\tun\tun\tun,\tdeux\tun\tun,\tun \tdeux \tun, \tun \tun \tdeux,
\ttroisun\tun,\tun \ttroisun,\ttroisdeux\tun,\tun \ttroisdeux,\tdeux\tdeux,\tquatreun,\tquatredeux,\tquatretrois,\tquatrequatre,\tquatrecinq.$$
\end{enumerate}}

If $\boldsymbol{v} \models V(F)$, then $Lea_{\boldsymbol{v}}(F)$ and $Roo_{\boldsymbol{v}}(F)$ are naturally planar forests. It is proved in \cite{FoissyI02} that the space $ \hh_{NCK} $ generated by planar forests is a Hopf algebra. Its product is given by the concatenation of planar forests and its coproduct is defined for any rooted forest $ F $ by:
$$\Delta_{\hh_{NCK}}(F)=\sum_{\boldsymbol{v} \models V(F)} Lea_{\boldsymbol{v}}(F) \otimes Roo_{\boldsymbol{v}}(F)
=F \otimes 1+1\otimes F+\sum_{\boldsymbol{v} \mmodels V(F)} Lea_{\boldsymbol{v}}(F) \otimes Roo_{\boldsymbol{v}}(F).$$
For example:
\begin{eqnarray*}
\Delta_{\hh_{NCK}} (\tquatredeux )&=&\tquatredeux \otimes 1+1\otimes \tquatredeux+
\tun \otimes \ttroisun+\tdeux \otimes \tdeux+\tun \otimes \ttroisdeux+\tun\tun \otimes \tdeux+\tdeux \tun \otimes \tun ,\\
\Delta_{\hh_{NCK}}(\tquatretrois) &=&\tquatretrois \otimes 1+1\otimes \tquatretrois+
\tun \otimes \ttroisun+\tdeux \otimes \tdeux+\tun \otimes \ttroisdeux+\tun\tun \otimes \tdeux+\tun\tdeux \otimes \tun.
\end{eqnarray*}

As in the nonplanar case, there is a decorated version $\hh_{NCK}^{\mathcal{D}}$ of $\hh_{NCK}$. Moreover, $ \hh_{NCK} $ is a Hopf algebra graded by the number of vertices. The number $f_{n}^{\hh_{NCK}}$ of planar forests of vertices degree $n$ (equal to the number of planar trees of vertices degree $ n+1 $) is the $n$-Catalan number $\frac{(2n)!}{n!(n+1)!}$, see sequence A000108 of \cite{Sloane}. We have:
\begin{eqnarray} \label{seriesNCK}
T_{\hh_{NCK}}(x) = \dfrac{1-\sqrt{1-4x}}{2}, \hspace{1cm} F_{\hh_{NCK}}(x) = \dfrac{1-\sqrt{1-4x}}{2x}.
\end{eqnarray}

This gives:

$$\begin{array}{c|c|c|c|c|c|c|c|c|c|c|c}
n&0&1&2&3&4&5&6&7&8&9&10\\
\hline f_{n}^{\hh_{NCK}} &1&1&2&5&14&42&132&429&1430&4862&16796
\end{array}$$

\subsection{Ordered and heap-ordered forests}

\begin{defi}
An ordered forest $F$ is a rooted forest $F$ with a total order on $V(F)$. The set of ordered forests is denoted by $ \mathbb{F}_{\hh_{o}} $ and the $ \mathbb{K} $-vector space generated by $ \mathbb{F}_{\hh_{o}} $ is denoted by $ \hh_{o} $.
\end{defi}

{\bf Remarks and notations.} {If $F$ is an ordered forest, then there exits a unique increasing bijection $\sigma : V(F) \rightarrow \{1 ,\hdots, \left| F \right|_{v}\}$ for the total order on $V(F)$.

Reciprocally, if $F$ is a rooted forest and if $\sigma : V(F) \rightarrow \{1 ,\hdots, \left| F \right|_{v}\}$ is a bijection, then $\sigma$ defines a total order on $V(F)$ and $F$ is an ordered forest.

Depending on the case, we shall denote an ordered forest by $F$ or $(F,\sigma)$.}
\\

{\bf Examples.} {Ordered forests of vertices degree $ \leq 3 $:
\begin{eqnarray*}
1 , \tdun{$1$} , \tdun{$1$}\tdun{$2$},\tddeux{$1$}{$2$},\tddeux{$2$}{$1$} , \tdun{$1$}\tdun{$2$}\tdun{$3$}, \tdun{$1$}\tddeux{$2$}{$3$},\tdun{$1$}\tddeux{$3$}{$2$},\tddeux{$1$}{$3$}\tdun{$2$},\tdun{$2$}\tddeux{$3$}{$1$},\tddeux{$1$}{$2$}\tdun{$3$},\tddeux{$2$}{$1$}\tdun{$3$},
\tdtroisun{$1$}{$3$}{$2$},\tdtroisun{$2$}{$3$}{$1$},\tdtroisun{$3$}{$2$}{$1$}, \tdtroisdeux{$1$}{$2$}{$3$},\tdtroisdeux{$1$}{$3$}{$2$},\tdtroisdeux{$2$}{$1$}{$3$},\tdtroisdeux{$2$}{$3$}{$1$},\tdtroisdeux{$3$}{$1$}{$2$},\tdtroisdeux{$3$}{$2$}{$1$} .
\end{eqnarray*}
}

Let $ (F, \sigma^{F}) $ and $ (G, \sigma^{G}) $ be two ordered forests. Then the rooted forest $FG$ is also an ordered forest $ (F G , \sigma^{FG}) $  where
\begin{eqnarray} \label{prodforetord}
\sigma^{FG} = \sigma^{F} \otimes \sigma^{G} :
\left\lbrace \begin{array}{rcl}
V(F) \bigcup V(G) & \rightarrow & \{1 , \hdots , \left| F \right|_{v} + \left| G \right|_{v} \} \\
a \in V(F) & \mapsto & \sigma^{F}(a) \\
a \in V(G) & \mapsto & \sigma^{G}(a) + \left| F \right|_{v} .
\end{array}
\right.
\end{eqnarray}

In other terms, we keep the initial order on the vertices of $F$ and $G$ and we assume that the vertices of $F$ are smaller than the vertices of $G$. This defines a noncommutative product on the set of ordered forests. 
For example, the product of $\tdun{1}$ and $\tddeux{1}{2}$ gives $\tdun{1}\tddeux{2}{3} = \tddeux{2}{3} \tdun{1}$, whereas the product of $\tddeux{1}{2}$ and $\tdun{1}$
gives $\tddeux{1}{2}\tdun{3}=\tdun{3}\tddeux{1}{2}$. This product is linearly extended to $\hh_{o}$, which in this way becomes an algebra.\\

$\hh_{o}$ is graded by the number of vertices and there is $(n+1)^{n-1}$ ordered forests in vertices degree $n$, see sequence A000272 of \cite{Sloane}.\\

If $ F $ is an ordered forest, then any subforest $G$ of $ F $ is also ordered: the total order on $V(G)$ is the restriction of the total order of $V(F)$. So we can define a coproduct $ \Delta_{\hh_{o}} : \hh_o \rightarrow \hh_o \otimes \hh_o $ on $ \hh_o $ in the following way: for all $ F \in \mathbb{F}_{\hh_{o}} $,
$$ \Delta_{\hh_{o}}(F)=\sum_{\boldsymbol{v} \models V(F)} Lea_{\boldsymbol{v}} (F) \otimes Roo_{\boldsymbol{v}} (F) .$$

For example,
$$\Delta_{\hh_{o}} (\tdquatredeux{2}{3}{4}{1} )=\tdquatredeux{2}{3}{4}{1} \otimes 1+1\otimes \tdquatredeux{2}{3}{4}{1}
+\tdun{1} \otimes \tdtroisun{1}{3}{2}+\tddeux{2}{1} \otimes \tddeux{1}{2}+\tdun{1} \otimes \tdtroisdeux{2}{3}{1}
+\tdun{1}\tdun{2} \otimes \tddeux{1}{2}+\tddeux{3}{1}\tdun{2} \otimes \tdun{1}.$$
With this coproduct, $ \hh_o $ is a Hopf algebra.

\begin{defi} \cite{Grossman90}
A heap-ordered forest is an ordered forest $ F $ such that if $a, b \in V(F)$, $a \neq b$ and $a \twoheadrightarrow b$, then $a$ is greater than $b$ for the total order on $V(F)$. The set of heap-ordered forests is denoted by $\mathbb{F}_{\hh_{ho}}$.
\end{defi}

{\bf Examples.} {Heap-ordered forests of vertices degree $ \leq 3 $:
$$ 1 , \tdun{$1$} , \tdun{$1$}\tdun{$2$},\tddeux{$1$}{$2$} , \tdun{$1$}\tdun{$2$}\tdun{$3$},\tdun{$1$}\tddeux{$2$}{$3$},\tdun{$2$}\tddeux{$1$}{$3$},\tdun{$3$}\tddeux{$1$}{$2$},\tdtroisun{$1$}{$3$}{$2$},\tdtroisdeux{$1$}{$2$}{$3$} .$$}

\begin{defi} \label{ordrelineaire}
A linear order on a nonempty rooted forest $ F $ is a bijective map $ \sigma : V(F) \rightarrow \{1, \hdots , \left| F \right|_{v} \} $ such that if $ a,b \in V(F) $ and $ a \twoheadrightarrow b $, then $ \sigma (a) \geq \sigma (b) $. We denote by $ \mathcal{O}(F) $ the set of linear orders on the nonempty rooted forest $ F $.
\end{defi}

{\bf Remarks.} {If $(F,\sigma)$ is a heap-ordered forest, then the increasing bijection $ \sigma : V(F) \rightarrow \{1, \hdots , \left| F \right|_{v} \} $ is a linear order on $F$. Reciprocally, if $F$ is a rooted forest and $\sigma \in \mathcal{O}(F)$, then $\sigma$ defines a total order on $V(F)$ such that $(F,\sigma)$ is a heap-ordered forest.}
\\

If $F$ and $G$ are two heap-ordered forests, then $FG$ is an ordered forest with (\ref{prodforetord}) and also a heap-ordered forest. Moreover, any subforest $G$ of a heap-ordered forest $ F $ is also a heap-ordered forest by restriction on $V(G)$ of the total order of $V(F)$. So the subspace $\hh_{ho}$ of $\hh_o$ generated by the heap-ordered forests is a graded Hopf subalgebra of $\hh_o$.\\

The number of heap-ordered forests of vertices degree $n$ is $n!$, see sequence A000142 of \cite{Sloane}.\\

{\bf Remarks.} {\begin{enumerate}
\item A planar forest can be considered as an ordered forest by ordering its vertices in the "north-west" direction (this is the order defined in \cite{FoissyI02} or given by the Depth First Search algorithm). This defines an algebra morphism $\phi : \hh_{NCK} \rightarrow \hh_{o}$. For example:
\begin{eqnarray} \label{exampleplanord}
\begin{array}{rcl|rcl|rcl}
\ttroisun & \overset{\phi}{\longrightarrow} & \tdtroisun{$1$}{$3$}{$2$} & \ttroisdeux &\overset{\phi}{\longrightarrow}& \tdtroisdeux{$1$}{$2$}{$3$} & \tquatreun \tdeux & \overset{\phi}{\longrightarrow}& \tdquatreun{$1$}{$4$}{$3$}{$2$} \tddeux{$5$}{$6$} \\
\ttroisun \tquatredeux & \overset{\phi}{\longrightarrow} & \tdtroisun{$1$}{$3$}{$2$} \tdquatredeux{$4$}{$7$}{$5$}{$6$} & \tquatretrois \ttroisdeux &\overset{\phi}{\longrightarrow}& \tdquatretrois{$1$}{$3$}{$4$}{$2$} \tdtroisdeux{$5$}{$6$}{$7$} &
\tdeux \tquatrequatre & \overset{\phi}{\longrightarrow} & \tddeux{$1$}{$2$} \tdquatrequatre{$3$}{$4$}{$6$}{$5$}
\end{array}
\end{eqnarray}
\item Reciprocally, an ordered forest is also planar, by restriction of the total order to the subsets of vertices formed by the roots or $\{w \in V(\FF)\:\mid \:w \rightarrow v\}$. This defines an algebra morphism $\psi : \hh_{o} \rightarrow \hh_{NCK}$. For example:
\begin{eqnarray}
\begin{array}{rcl|rcl|rcl}
\tdun{$1$} \tddeux{$3$}{$2$}& \overset{\psi}{\longrightarrow} & \tun \tdeux & \tdun{$2$} \tddeux{$1$}{$3$} & \overset{\psi}{\longrightarrow} & \tdeux \tun & \tdquatredeux{$5$}{$1$}{$3$}{$4$} \tddeux{$2$}{$6$} & \overset{\psi}{\longrightarrow} \tdeux \tquatretrois \\
\tdquatreun{$1$}{$5$}{$3$}{$2$} \tdtroisdeux{$4$}{$7$}{$6$} & \overset{\psi}{\longrightarrow} & \tquatreun \ttroisdeux & \tdquatretrois{$4$}{$7$}{$1$}{$5$} \tdtroisun{$2$}{$6$}{$3$} & \overset{\psi}{\longrightarrow} & \ttroisun \tquatretrois & \tdtroisun{$5$}{$3$}{$1$} \tdquatrequatre{$2$}{$6$}{$7$}{$4$}  & \overset{\psi}{\longrightarrow} \tquatrequatre \ttroisun
\end{array}
\end{eqnarray}
\end{enumerate}}

Note that $\psi \circ \phi = Id_{\hh_{NCK}}$ therefore $\psi$ is surjective and $\phi$ is injective. $\psi$ and $\phi$ are not bijective (by considering the dimensions).

Moreover $\phi$ is a Hopf algebra morphism and its image is included in $\hh_{ho}$. $\psi$ is not a Hopf algebra morphism: in the expression of $(\psi \otimes \psi) \circ \Delta_{\hh_{o}}(\tdun{$3$} \tdtroisdeux{$4$}{$1$}{$2$})$ we have the tensor $\tdeux \tun \otimes \tun $ and in the expression of $\Delta_{\hh_{NCK}} \circ \psi(\tdun{$3$} \tdtroisdeux{$4$}{$1$}{$2$})$ we have the different tensor $\tun \tdeux \otimes \tun $. By cons, the restriction of $\psi$ of $\hh_{ho}$ is a Hopf algebra morphism.\\

In the following, we consider $ \hh_{NCK} $ as a Hopf subalgebra of $ \hh_{ho} $ and $\hh_{o}$.

\section{Recalls on the Hopf algebras of words}

\subsection{Hopf algebra of permutations and shuffles}

\indent

{\bf Notations.} {\begin{enumerate}
\item Let $k,l$ be integers. A {\it $(k,l)$-shuffle} is a permutation $\zeta$ of $\{1,\ldots,k+l\}$, such that 
$\zeta(1)<\ldots< \zeta(k)$ and $\zeta(k+1)<\ldots< \zeta(k+l)$. The set of $(k,l)$-shuffles will be denoted by $Sh(k,l)$.
\item We represent a permutation $\sigma \in \Sigma_n$ by the word $(\sigma(1)\ldots \sigma(n))$. For example,
$Sh(2,1)=\{(123),\:(132),\:(231)\}$.
\end{enumerate}}

{\bf Remark.} {For any integer $k,l$, any permutation $\sigma \in \Sigma_{k+l}$ can be uniquely written as $\epsilon \circ (\sigma_1\otimes \sigma_2)$, where 
$\sigma_1 \in \Sigma_k$, $\sigma_2 \in \Sigma_l$, and $\epsilon \in Sh(k,l)$. Similarly, considering the inverses,
any permutation $\tau \in \Sigma_{k+l}$ can be uniquely written as $(\tau_1\otimes \tau_2) \circ \zeta^{-1}$, where 
$\tau_1 \in \Sigma_k$, $\tau_2 \in \Sigma_l$, and $\zeta \in Sh(k,l)$. Note that, whereas $\eps$ renames the numbers of each lists $(\sigma(1),\ldots,\sigma(k)),(\sigma(k+1),\ldots,\sigma(k+l))$ without
changing their orderings,
$\zeta^{-1}$ shuffles the lists $(\tau(1),\ldots,\tau(k)),(\tau(k+1),\ldots,\tau(k+l))$. For instance, if $k=4$, $l=3$ and $\sigma = (5172436) $ then
\begin{itemize}
\item  $\sigma = \epsilon \circ (\sigma_1\otimes \sigma_2)$, with $\epsilon = (1257346) \in Sh(4,3)$, $\sigma_{1} = (3142) \in \Sigma_{4}$ and $\sigma_{2} = (213) \in \Sigma_{3}$,
\item $\sigma = (\tau_1\otimes \tau_2) \circ \zeta^{-1} $, with $\tau_{1} = (1243) \in \Sigma_{4}$, $\tau_{2} = (132) \in \Sigma_{3}$ and $\zeta = (2456137) \in Sh(4,3)$.
\end{itemize} }
\vspace{0.5cm}

We here briefly recall the construction of the Hopf algebra $\FQSym$ of free quasi-symmetric functions, also called the Malvenuto-Reutenauer Hopf algebra
\cite{Duchamp00,Malvenuto05}. As a vector space, a basis of $\FQSym$ is given by the disjoint union of the symmetric groups $\Sigma_n$, for all $n \geq 0$.
By convention, the unique element of $\Sigma_0$ is denoted by $1$.
The product of $\FQSym$ is given, for $\sigma \in \Sigma_k$, $\tau \in \Sigma_l$, by:
$$\sigma.\tau=\sum_{\zeta \in Sh(k,l)} (\sigma \otimes \tau) \circ \zeta^{-1}.$$
In other words, the product of $\sigma$ and $\tau$ is given by shifting the letters of the word
representing $\tau$ by $k$, and then summing all the possible shufflings of this word and of the word representing $\sigma$.
For example:
\begin{eqnarray*}
(123)(21)&=&(12354)+(12534)+(15234)+(51234)+(12543)\\
&&+(15243)+(51243)+(15423)+(51423)+(54123).
\end{eqnarray*}

Let $\sigma \in \Sigma_n$. For all $0\leq k \leq n$, there exists a unique triple 
$\left(\sigma_1^{(k)},\sigma_2^{(k)},\epsilon_k\right)\in \Sigma_k \times \Sigma_{n-k} \times Sh(k,n-k)$ such that $\sigma=\epsilon_{k} \circ \left(\sigma_1^{(k)} \otimes \sigma_2^{(k)}\right)$. The coproduct of $\FQSym$ is then defined by:
$$\Delta_{\FQSym}(\sigma)=\sum_{k=0}^n \sigma_1^{(k)} \otimes \sigma_2^{(k)}
=\sum_{k=0}^n \sum_{\substack{\sigma=\epsilon \circ (\sigma_1 \otimes \sigma_2)\\ \epsilon \in Sh(k,n-k), \sigma_1 \in \Sigma_k,\sigma_2 \in \Sigma_{n-k}}}
\sigma_1 \otimes \sigma_2.$$
Note that $\sigma_1^{(k)}$ and $\sigma_2^{(k)}$ are obtained by cutting the word representing $\sigma$ between the $k$-th and the $(k+1)$-th letter,
and then {\it standardizing} the two obtained words by the following process. If $v$ is a words of length $n$ whose the letters are distinct integers, then the standardizing of $v$, denoted $Std(v)$, is the word obtained by applying to the letters of $v$ the unique increasing bijection to $\{1,\ldots,n\}$.
For example:
\begin{eqnarray*}
\Delta_{\FQSym}((41325))&=&1\otimes (41325)+Std(4)\otimes Std(1325)+Std(41)\otimes Std(325)\\
&&+Std(413)\otimes Std(25)+Std(4132)\otimes Std(5)+(41325)\otimes 1\\
&=&1\otimes (41325)+(1) \otimes (1324)+(21) \otimes (213)\\
&&+(312)\otimes (12)+(4132) \otimes (1)+(41325) \otimes 1.
\end{eqnarray*}
Then $\FQSym$ is a Hopf algebra. It is graded, with $\FQSym(n)=Vect(\Sigma_n)$ for all $n \geq 0$. \\

It is also possible to give a decorated version of $\FQSym$. Let $\D$ be a nonempty set. A $\D$-decorated permutation is a couple $(\sigma,d)$, where $\sigma \in \Sigma_n$
and $d$ is a map from $\{1,\ldots,n\}$ to $\D$. A $\D$-decorated permutation is represented by two superposed words 
$\left(\substack{a_1\ldots a_n\\v_1\ldots v_n}\right)$,
where $(a_1 \ldots a_n)$ is the word representing $\sigma$ and for all $i$, $v_i=d(a_i)$. The vector space $\FQSym^{\D}$ generated by the set
of $\D$-decorated permutations is a Hopf algebra. For example, if $x,y,z,t \in \D$:
\begin{eqnarray*}
\left(\substack{213\\yxz}\right).\left(\substack{1\\t}\right)
&=&\left(\substack{2134\\yxzt}\right)+ \left(\substack{2143\\yxtz}\right)+ \left(\substack{2413\\ytxz}\right)+ \left(\substack{4213\\tyxz}\right),\\
\Delta_{\FQSym^{\mathcal{D}}} \left( \left(\substack{4321\\tzyx}\right) \right) &=&\left( \substack{4321\\tzyx}\right)\otimes 1
+\left(\substack{321\\tzy}\right)\otimes \left(\substack{1\\x}\right)+\left( \substack{21\\tz}\right)\otimes \left(\substack{21\\yx}\right)+
\left(\substack{1\\t}\right) \otimes \left(\substack{321\\zyx}\right)+1\otimes \left(\substack{4321\\tzyx}\right).
\end{eqnarray*}
In other words, if $(\sigma,d)$ and $(\tau,d')$ are decorated permutations of respective degrees $k$ and $l$:
\begin{eqnarray*}
(\sigma,d).(\tau,d')=\sum_{\zeta \in Sh(k,l)}((\sigma \otimes \tau) \circ \zeta^{-1}, d \otimes d'),
\end{eqnarray*}
where $d \otimes d'$ is defined by $(d \otimes d')(i)=d(i)$ if $1 \leq i \leq k$ and $(d \otimes d')(k+j)=d'(j)$ if $1 \leq j \leq l$.
If $(\sigma,d)$ is a decorated permutation of degree $n$:
\begin{eqnarray*}
\Delta_{\FQSym^{\mathcal{D}}}((\sigma,d))=\sum_{k=0}^n \sum_{\substack{\sigma=\epsilon \circ (\sigma_1 \otimes \sigma_2)\\ 
\epsilon \in Sh(k,l), \sigma_1 \in \Sigma_k,\sigma_2 \in \Sigma_l}} (\sigma_1 ,d') \otimes (\sigma_2 , d''),
\end{eqnarray*}
where $d = (d' \otimes d'') \circ \epsilon^{-1}$.

In some sense, a $\D$-decorated permutation can be seen as a word with a total order on the set of its letters.\\

We can now define the shuffle Hopf algebra $\Sh^{\D}$ (see \cite{Hoffman00,Reutenauer93}). A $\mathcal{D}$-word is a finite sequence of elements taken in $\D$. It is graded by the degree of words, that is to say the number of their letters. As a vector space, $\Sh^{\D}$ is generated by the set of $\D$-words.\\

The surjective linear map from $\FQSym^{\D}$ to $\Sh^{\D}$, sending the decorated permutation $\left(\substack{a_1\ldots a_n\\v_1\ldots v_n}\right)$ to the $\D$-word $(v_{1} \ldots v_{n})$, define a Hopf algebra structure on $\Sh^{\D}$:
\begin{itemize}
\item The product $\shuffle$ of $\mathbf{Sh}^{\D}$ is given in the following way: if $ (v_{1} \hdots v_{k})$ is a $\D$-word of degree $k$, $ (v_{k+1} \hdots v_{k+l})$ is a $\D$-word of degree $l$, then
\begin{eqnarray*}
(v_{1} \hdots v_{k}) \shuffle (v_{k+1} \hdots v_{k+l}) = \sum_{\zeta \in Sh(k,l)} v_{\zeta^{-1}(1)} \hdots v_{\zeta^{-1}(k+l)} .
\end{eqnarray*}
\item The coproduct $\Delta_{\Sh^{\D}}$ of $\mathbf{Sh}^{\D}$ is given on any $\D$-word $w=(v_{1} \hdots v_{n})$ by
\begin{eqnarray*}
\Delta_{\Sh^{\D}}(w) = \sum_{i=0}^{n} (v_{1} \hdots v_{i}) \otimes (v_{i+1} \hdots v_{n}) .
\end{eqnarray*}
\end{itemize}
\vspace{0.5cm}

{\bf Examples.} {\begin{enumerate}
\item If $(v_{1} v_{2} v_{3})$ and $(v_{4} v_{5})$ are two $\D$-words,
\begin{eqnarray*}
(v_1v_2v_3)\shuffle (v_4v_5) &=& (v_1v_2v_3v_4v_5) + (v_1v_2v_4v_3v_5) + (v_1v_2v_4v_5v_3) +(v_1v_4v_2v_3v_5) \\
& & +(v_1v_4v_2v_5v_3) +(v_1v_4v_5v_2v_3) + (v_4v_1v_2v_3v_5) + (v_4v_1v_2v_5v_3)\\
& & +(v_4v_1v_5v_2v_3) + (v_4v_5v_1v_2v_3).
\end{eqnarray*}
\item If $(v_{1} v_{2} v_{3} v_{4})$ is a $\D$-word,
\begin{eqnarray*}
\Delta_{\Sh^{\D}}((v_{1} v_{2} v_{3} v_{4})) & = & (v_{1} v_{2} v_{3} v_{4}) \otimes 1 + (v_{1} v_{2} v_{3}) \otimes (v_{4}) + (v_{1} v_{2}) \otimes (v_{3} v_{4})\\
& & + (v_{1}) \otimes (v_{2} v_{3} v_{4}) + 1 \otimes (v_{1} v_{2} v_{3} v_{4}) .
\end{eqnarray*}
\end{enumerate}
}
\vspace{0.5cm}

There is a link between the algebras $\hh_{o}$, $\hh_{ho}$ and $\FQSym$ given by the following result (see \cite{Foissy10}):

\begin{prop} \label{dehdansfqsym}
\begin{enumerate}
\item Let $n \geq 0$. For all $(F,\sigma) \in \mathbb{F}_{\hh_{o}}$, let $\mathbb{S}_{F}$ be the set of permutations $\theta \in \Sigma_{n}$ such that for all $a,b \in V(F)$, $(a \twoheadrightarrow b) \Rightarrow (\theta^{-1}(\sigma(a)) \leq \theta^{-1}(\sigma(b)))$. Let us define:
\begin{eqnarray*}
\Theta : \left\lbrace \begin{array}{rcl}
\hh_{o} & \rightarrow & \FQSym \\
F \in \mathbb{F}_{\hh_{o}} & \mapsto & \displaystyle\sum_{\theta \in \mathbb{S}_{F}} \theta .
\end{array}
\right. 
\end{eqnarray*}
Then $\Theta : \hh_{o} \rightarrow \FQSym$ is a Hopf algebra morphism, homogeneous of degree $0$.
\item The restriction of $\Theta$ to $\hh_{ho}$ is an isomorphism of graded Hopf algebras.
\end{enumerate}
\end{prop}

\subsection{Hopf algebra of packed words and quasi-shuffles}

Recall the construction of the Hopf algebra $\WQSyms$ of free packed words (see \cite{Novelli06}).
\\

{\bf Notations.} {\begin{enumerate}
\item Let $n \geq 0$. We denote by $Surj_{n}$ the set of maps $\sigma : \{1,\ldots,n \} \rightarrow \mathbb{N}^{\ast}$, such that $\sigma (\{1,\ldots,n \}) = \{1,\ldots,k \}$ for a certain $k \in \mathbb{N}$. In this case, $k$ is the maximum of $\sigma$ and is denoted by $ \max (\sigma)$ and $n$ is the length of $\sigma$. We represent the element $\sigma$ of $Surj_{n}$ by the packed word $(\sigma(1) \ldots \sigma(n))$.
\item Let $k,l$ be two integers. A {\it $(k,l)$-surjective shuffle} is an element $\epsilon$ of $Surj_{k+l}$ such that $\epsilon(1) < \hdots < \epsilon(k)$ and $\epsilon(k+1) < \hdots < \epsilon(k+l)$. The set of $(k,l)$-surjective shuffles will be denoted by $SjSh(k,l)$. For example, $SjSh(2,1) = \{ (121), (122) , (123), (132), (231) \}$.
\end{enumerate}
}

Let $v$ be a word such that the letters occuring in $v$ are integers $a_{1} < a_{2} < \hdots < a_{k}$. The {\it packing} of $v$, denoted by $pack(v)$, is the image of letters of $v$ by the application $a_{i} \mapsto i$. For example, $pack((22539)) = (11324)$, $pack((831535)) = (421323)$.
\\

{\bf Remark.} {Let $k,l$ be two integers and $\sigma \in Surj_{k+l}$. We set $p_{k}= \max (pack((\sigma(1) \hdots \sigma(k))))$ and $q_{k}= \max (pack( (\sigma(k+1) \hdots \sigma(k+l))))$. Then $\sigma$ can be uniquely written as $\epsilon \circ (\sigma_1 \otimes \sigma_2)$, where 
$\sigma_1 \in Surj_k$, $\sigma_2 \in Surj_l$, and $\epsilon \in SjSh(p_{k},q_{k})$. For instance, if $k=4$, $l=3$ and $\sigma = (2311223)$ then $p_{4} = 3$, $q_{4} = 2$ and $\sigma = \epsilon \circ (\sigma_1 \otimes \sigma_2)$ with $\epsilon = (12323) \in SjSh(3,2)$, $\sigma_{1} = (2311) \in Surj_{4}$ and $\sigma_{2} = (112) \in Surj_{3}$.}
\\

As a vector space, a basis of $\WQSyms$ is given by the disjoint union of the sets $Surj_{n}$, for all $n \geq 0$. By convention, the unique element of $Surj_{0}$ is denoted by $1$. The product of $\WQSyms$ is given, for $\sigma \in Surj_{k}$ and $\tau \in Surj_{l}$, by:
\begin{eqnarray*}
\sigma . \tau = \sum_{\zeta \in Sh(k,l)} (\sigma \otimes \tau) \circ \zeta^{-1} .
\end{eqnarray*}
In other words, as in the $\FQSym$ case, the product of $\sigma$ and $\tau$ is given by shifting the letters of the word representing $\tau$ by $k$, and summing all the possible shuffings of this word and of the word representing $\sigma$. For example:
\begin{eqnarray*}
(112)(21) & = & (11243) + (11423) + (14123) + (41123) + (11432) + (14132)\\
& & + (41132) + (14312) + (41312) + (43112)
\end{eqnarray*}

Let $\sigma \in Surj_{n}$. For all $0 \leq k \leq n$, there exists a unique triple $\left( \sigma_{1}^{(k)}, \sigma_{2}^{(k)}, \epsilon_{k} \right) \in Surj_{k} \times Surj_{n-k} \times SjSh(p_{k},q_{k})$ such that $\sigma = \epsilon_{k} \circ \left( \sigma_{1}^{(k)} \otimes \sigma_{2}^{(k)} \right)$. The coproduct of $\WQSyms$ is then given by:
\begin{eqnarray*}
\Delta_{\WQSyms}(\sigma)=\sum_{k=0}^n \sigma_1^{(k)} \otimes \sigma_2^{(k)}
=\sum_{k=0}^n \sum_{\substack{\sigma=\epsilon \circ (\sigma_1 \otimes \sigma_2) \\ \epsilon \in SjSh(p_{k},q_{k}), \sigma_1 \in Surj_k,\sigma_2 \in Surj_{n-k}}}
\sigma_1 \otimes \sigma_2.
\end{eqnarray*}
Note that $\sigma_1^{(k)}$ and $ \sigma_2^{(k)}$ are obtained by cutting the word representing $\sigma$ between the $k$-th and the $(k+1)$-th letter, and then packing the two obtained words. For example:
\begin{eqnarray*}
\Delta_{\WQSyms} ((21132)) & = & 1 \otimes (21132) + pack((2)) \otimes pack((1132)) + pack((21)) \otimes pack((132)) \\
& & + pack((211)) \otimes pack((32)) + pack((2113)) \otimes pack((2)) + (21132) \otimes 1 \\
& = & 1 \otimes (21132) + (1) \otimes (1132) + (21) \otimes (132) + (211) \otimes (21) \\
& & + (2113) \otimes (1) + (21132) \otimes 1 .
\end{eqnarray*}
Then $\WQSyms$ is a graded Hopf algebra, with $\WQSyms (n) = Surj_{n}$ for all $n \geq 0$. We give some numerical values: if $f^{\WQSyms}_{n}$ is the number of packed words of length $n$, then

$$\begin{array}{c|c|c|c|c|c|c|c|c}
n&0&1&2&3&4&5&6&7\\
\hline f^{\WQSyms}_{n} &1&1&3&13&75&541&4683&47293
\end{array}$$

These is the sequence A000670 in \cite{Sloane}.\\

$\WQSyms$ is the gradued dual of $\WQSym$, described as follows. A basis of $\WQSym$ is given by the disjoint union of the sets $Surj_{n}$. The product of $\WQSym$ is given, for $\sigma \in Surj_{k}$, $\tau \in Surj_{l}$ by:
\begin{eqnarray*}
\sigma . \tau = \sum_{\substack{\gamma = \gamma_{1} \hdots \gamma_{k+l} \\ pack(\gamma_{1} \hdots \gamma_{k}) = \sigma , \ pack(\gamma_{k+1} \hdots \gamma_{k+l})= \tau}} \gamma
\end{eqnarray*}
In other terms, the product of $\sigma$ and $\tau$ is given by shifting certain letters of the words representing $\sigma$ and $\tau$ and then summing all concatenations of obtained words. For example:
\begin{eqnarray*}
(112) (21) & = & (11221) + (11321) + (22321)+ (33421)+ (11231) + (22331) + (22431)\\
& & + (11232) + (11332) + (11432) + (22341) + (11342) + (11243)
\end{eqnarray*}
If $\sigma \in Surj_{n}$, then the coproduct of $\WQSym$ is given by:
\begin{eqnarray*}
\Delta_{\WQSym} (\sigma) = \sum_{0 \leq k \leq max(\sigma)} \sigma_{|[1,k]} \otimes pack(\sigma_{|[k+1,max(\sigma)]}) ,
\end{eqnarray*}
where $\sigma_{|\mathcal{A}}$ is the subword obtained by tacking in $\sigma$ the letters from the subset $\mathcal{A}$ of $[1 , max(\sigma)]$. For example:
\begin{eqnarray*}
\Delta_{\WQSym} ((21312245)) & = & 1 \otimes (21312245) + (11) \otimes pack((232245)) + (21122) \otimes pack((345)) \\
& & + (213122) \otimes pack((45)) + (2131224) \otimes pack((5)) + (21312245) \otimes 1 \\
& = & 1 \otimes (21312245) + (11) \otimes (121134) + (21122) \otimes (123) \\
& & + (213122) \otimes (12) + (2131224) \otimes (1) + (21312245) \otimes 1 .
\end{eqnarray*}
Then $\WQSym$ is a Hopf algebra.\\

We give a decorated version of $\WQSym$. Let $\mathcal{D}$ be a nonempty set. A $\mathcal{D}$-decorated surjection is a couple $(\sigma,d)$, where $\sigma \in Surj_{n}$ and $d$ is a map from $\{1,\hdots,n \}$ to $\mathcal{D}$. As in the $\FQSym^{\mathcal{D}}$ case, we represent a $\mathcal{D}$-decorated surjection by two superposed words $\left(\substack{a_1\ldots a_n\\v_1\ldots v_n}\right)$, where $(a_{1} \ldots a_{n})$ is the packed word representing $\sigma$ and for all $i$, $v_{i} = d(a_{i})$. The vector space $\WQSym^{\mathcal{D}}$ generated by the set of $\mathcal{D}$-decorated surjections is a Hopf algebra. For example, if $x,y,z,t \in \mathcal{D}$:
\begin{eqnarray*}
\left(\substack{211\\yxz}\right).\left(\substack{1\\t}\right) & = & \left(\substack{2111\\yxzt}\right) + \left(\substack{2112\\yxzt}\right) + \left(\substack{2113\\yxzt}\right) + \left(\substack{3221\\yxzt}\right) +
\left(\substack{3112\\yxzt}\right) . \\
\Delta_{\WQSym^{\mathcal{D}}} \left( \left(\substack{2113\\yxzt}\right) \right) & = & \left(\substack{2113\\yxzt}\right) \otimes 1 + \left(\substack{11\\xz}\right) \otimes \left(\substack{12\\yt}\right) + \left(\substack{211\\yxz}\right) \otimes \left(\substack{1\\t}\right) + 1 \otimes \left(\substack{2113\\yxzt}\right) .
\end{eqnarray*}

In other words, if $(\sigma,d)$ and $(\tau,d')$ are decorated surjections of respective degrees $k$ and $l$:
\begin{eqnarray*}
(\sigma,d).(\tau,d')=\sum_{\substack{\gamma = \gamma_{1} \hdots \gamma_{k+l} \\ pack(\gamma_{1} \hdots \gamma_{k}) = \sigma , \ pack(\gamma_{k+1} \hdots \gamma_{k+l})= \tau}}(\gamma , d \otimes d'),
\end{eqnarray*}
where $d \otimes d'$ is defined by $(d \otimes d')(i)=d(i)$ if $1 \leq i \leq k$ and $(d \otimes d')(k+j)=d'(j)$ if $1 \leq j \leq l$.
If $(\sigma,d)$ is a decorated surjection of degree $n$:
\begin{eqnarray*}
\Delta_{\WQSym^{\mathcal{D}}}((\sigma,d))=\sum_{0 \leq k \leq \max (\sigma)} (\sigma_{| [1,k]} , d') \otimes (\sigma_{| [k+1, \max (\sigma)]}, d'').
\end{eqnarray*}
where $d'$ and $d$ take the same values on $ \sigma^{-1}(\{1,\hdots ,k \})$ and $d''$ and $d$ take the same values on $ \sigma^{-1}(\{k+1,\hdots , \max (\sigma) \})$.

In some sense, a $\mathcal{D}$-decorated surjection can be seen as a packed word with a preorder on the set of its letters.\\

Suppose that $\mathcal{D}$ is equipped with an associative and commutative product $ \left[ \cdot, \cdot \right]: (a,b) \in \mathcal{D}^{2} \rightarrow [ab] \in \mathcal{D}$. We define by induction $\left[ \cdot, \cdot \right]^{(k)}$:
$$ \left[ \cdot, \cdot \right]^{(0)} = Id , \ \left[ \cdot, \cdot \right]^{(1)} = \left[ \cdot, \cdot \right] \mbox{ and }\left[ \cdot, \cdot \right]^{(k)} = \left[ \cdot, \left[ \cdot, \cdot \right]^{(k-1)} \right] .$$

We can now define the quasi-shuffle Hopf algebra $\Csh^{\mathcal{D}}$ (see \cite{Hoffman00}). $\Csh^{\mathcal{D}}$ is, as a vector space, generated by the set of $\mathcal{D}$-words.\\

Let $\varphi$ be the surjective linear map from $\WQSym^{\D}$ to $\Csh^{\D}$ defined, for $(\sigma,d)$ a decorated surjection of maximum $k$, by $\varphi((\sigma,d)) = (w_{1} \ldots w_{k})$ where $w_{j} = \left[ d(i_{1}) \ldots d(i_{p}) \right]^{(p)} $ with $ \sigma^{-1} (j) = \{i_{1}, \ldots ,i_{p} \} $. For instance,
\begin{eqnarray*}
\varphi \left( \left( \substack{2114324\\yxztvwu}\right) \right) = \left( [xz] \ [yw] \ v \ [tu] \right) 
\end{eqnarray*}

{\bf Notations.} {Let $k,l$ be integers. A $(k,l)$-quasi-shuffle of type $r$ is a surjective map $\zeta: \{1, \hdots,k+l\} \twoheadrightarrow \{1,\hdots,k+l-r\}$ such that
\begin{eqnarray*}
\left\lbrace \begin{array}{l}
\zeta (1) < \hdots < \zeta (k), \\
\zeta (k+1) < \hdots < \zeta (k+l).
\end{array} \right. 
\end{eqnarray*}
Remark that $\zeta^{-1}(j)$ contains $1$ or $2$ elements for all $1 \leq j \leq k+l-r$. The set of $(k,l)$-quasi-shuffles of type $r$ is denoted by $Csh(p,q,r)$. Remark that $Csh(k,l,0) = Sh(k,l)$.\\

\noindent $\varphi$ define a Hopf algebra structure on $\Csh^{\D}$:
\begin{itemize}
\item The product $\shuffle \hspace{-.34cm} - $ of $\Csh^{\D}$ is given in the following way: if $(v_{1} \ldots v_{k})$ is a $\mathcal{D}$-word of degree $k$, $(v_{k+1} \hdots v_{k+l})$ is a $\mathcal{D}$-word of degree $l$, then
\begin{eqnarray*}
(v_{1} \hdots v_{k}) \shuffle \hspace{-.43cm} - (v_{k+1} \hdots v_{k+l}) = \sum_{r \geq 0} \sum_{\zeta \in Csh(k,l,r)} (w_{1} \hdots w_{k+l-r}) ,
\end{eqnarray*}
where $w_{j} = v_{i} $ if $\zeta^{-1}(j) = \{ i \} $ and $w_{j} = [v_{i_{1}} v_{i_{2}}]$ if $\zeta^{-1}(j) = \{ i_{1},i_{2} \}$.
\item The coproduct $\Delta_{\Csh^{\D}}$ of $\Csh^{\D}$ is given on any $\D$-word $v = (v_{1} \hdots v_{n})$ by
\begin{eqnarray*}
\Delta_{\Csh^{\D}}(v) = \sum_{i=0}^{n} (v_{1} \hdots v_{i}) \otimes (v_{i+1} \hdots v_{n}) .
\end{eqnarray*}
This is the same coproduct as for $\Sh^{\D}$.
\end{itemize}
\vspace{0.5cm}

{\bf Example.} {If $(v_{1} v_{2})$ and $(v_{3} v_{4})$ are two $\D$-words,
\begin{eqnarray*}
(v_{1} v_{2}) \shuffle \hspace{-.43cm} - (v_{3} v_{4}) & = & (v_{1} v_{2} v_{3} v_{4}) + (v_{1} v_{3} v_{2} v_{4}) + (v_{3} v_{1} v_{2} v_{4}) + (v_{1} v_{3} v_{4} v_{2}) \\
& & + (v_{3} v_{1} v_{4} v_{2}) + (v_{3} v_{4} v_{1} v_{2}) + (v_{1} \left[ v_{2} v_{3} \right] v_{4}) + (\left[ v_{1} v_{3} \right] v_{2} v_{4}) \\
& & + (v_{1} v_{3} \left[ v_{2} v_{4} \right]) + (v_{3} \left[ v_{1} v_{4} \right] v_{2}) + (\left[ v_{1} v_{3} \right] \left[ v_{2} v_{4} \right])
\end{eqnarray*}
}

\section{Preordered forests}

\subsection{Preordered and heap-preordered forests}

A preorder is a binary, reflexive and transitive relation. In particular, an antisymmetric preorder is an order. A preorder is total if two elements are always comparable. We introduce another version of ordered forests, the preordered forests.

\begin{defi}
A preordered forest $F$ is a rooted forest $F$ with a total preorder on $V(F)$. The set of preordered forests is denoted by $ \mathbb{F}_{\hh_{po}} $ and the $ \mathbb{K} $-vector space generated by $ \mathbb{F}_{\hh_{po}} $ is denoted by $ \hh_{po} $.
\end{defi}

{\bf Remarks and notations.} {\begin{enumerate}
\item Let $F$ be a preordered forest. We denote by $ \leq $ the total preorder on $V(F)$. Remark that the antisymmetric relation "$x \leq y$ and $y \leq x$" is an equivalence relation denoted by $\mathcal{R}$ and the quotient set $V(F)/\mathcal{R}$ is totally ordered. We denote by $q$ the cardinality of this quotient set. Let $\overline{\sigma}$ be the unique increasing map from $V(F)/\mathcal{R}$ to $\{1 ,\hdots ,q \}$. There exists a unique surjection $ \sigma:V(F) \rightarrow \{1, \hdots ,q \}$, compatible with the equivalence $\mathcal{R}$, such that the induced map on $V(F)/\mathcal{R}$ is $\overline{\sigma}$. In the sequel, we shall note $q= \max(F) $ (and we have always $ q \leq \left| F \right|_{v} $).

Reciprocally, if $F$ is a rooted forest and if $\sigma : V(F) \rightarrow \{1, \hdots , q \}$ is a surjection, $q \leq \left| F \right|_{v} $, then $\sigma$ define a total preorder on $V(F)$ and $F$ is a preordered forest.

As in the ordered case, we shall denote a preordered forest by $F$ or $(F,\sigma)$.
\item An ordered forest is also a preordered forest. Conversely, a preordered forest $(F,\sigma)$ is an ordered forest if $\left| F \right|_{v} = \max(F)$.
\end{enumerate}
}
\vspace{0.5cm}

{\bf Examples.} {Preordered forests of vertices degree $ \leq 3 $:
$$1 , \tdun{$1$} ,\tdun{$1$} \tdun{$1$} , \tdun{$1$}\tdun{$2$},\tddeux{$1$}{$1$},\tddeux{$1$}{$2$} , \tddeux{$2$}{$1$} ,\tdun{$1$}\tdun{$1$}\tdun{$1$} ,\tdun{$1$}\tdun{$1$}\tdun{$2$},\tdun{$1$}\tdun{$2$}\tdun{$2$}, \tdun{$1$}\tdun{$2$}\tdun{$3$}, \tdun{$1$} \tddeux{$1$}{$1$}, \tdun{$1$} \tddeux{$1$}{$2$}, \tdun{$1$} \tddeux{$2$}{$1$} , \tdun{$1$} \tddeux{$2$}{$2$}, \tdun{$1$} \tddeux{$2$}{$3$}, \tdun{$1$} \tddeux{$3$}{$2$}, $$

$$\tdun{$2$} \tddeux{$1$}{$1$}, \tdun{$2$} \tddeux{$1$}{$2$}, \tdun{$2$} \tddeux{$2$}{$1$}, \tdun{$2$} \tddeux{$1$}{$3$}, \tdun{$2$} \tddeux{$3$}{$1$}, \tdun{$3$} \tddeux{$1$}{$2$}, \tdun{$3$} \tddeux{$2$}{$1$},
\tdtroisun{$1$}{$1$}{$1$}, \tdtroisun{$1$}{$2$}{$1$}, \tdtroisun{$2$}{$1$}{$1$}, \tdtroisun{$1$}{$2$}{$2$}, \tdtroisun{$2$}{$2$}{$1$}, \tdtroisun{$1$}{$3$}{$2$}, \tdtroisun{$2$}{$3$}{$1$}, \tdtroisun{$3$}{$2$}{$1$}, 
\tdtroisdeux{$1$}{$1$}{$1$}, \tdtroisdeux{$1$}{$1$}{$2$},$$

$$\tdtroisdeux{$1$}{$2$}{$1$}, \tdtroisdeux{$2$}{$1$}{$1$}, \tdtroisdeux{$1$}{$2$}{$2$}, \tdtroisdeux{$2$}{$1$}{$2$}, \tdtroisdeux{$2$}{$2$}{$1$}, \tdtroisdeux{$1$}{$2$}{$3$}, \tdtroisdeux{$1$}{$3$}{$2$}, \tdtroisdeux{$2$}{$1$}{$3$}, \tdtroisdeux{$3$}{$1$}{$2$}, \tdtroisdeux{$2$}{$3$}{$1$}, \tdtroisdeux{$3$}{$2$}{$1$} .$$
}
\\

Let $ (F, \sigma^{F}) $ and $ (G, \sigma^{G}) $ be two preordered forests with $ \sigma^{F} : V(F) \rightarrow \{1, \hdots , q \} $, $ \sigma^{G} : V(G) \rightarrow \{1, \hdots , r \} $, $ q = max(F) $ and $ r = max (G)$. Then $FG$ is also a preordered forest $ (F G , \sigma^{FG}) $ where
\begin{eqnarray} \label{prodforetpreord}
\sigma^{FG} = \sigma^{F} \otimes \sigma^{G} :
\left\lbrace \begin{array}{rcl}
V(F) \bigcup V(G) & \rightarrow & \{1 , \hdots , q+r \} \\
a \in V(F) & \mapsto & \sigma^{F}(a) \\
a \in V(G) & \mapsto & \sigma^{G}(a) + q .
\end{array}
\right.
\end{eqnarray}
In other terms, we keep the initial preorder on the vertices of $F$ and $G$ and we assume that the vertices of $F$ are smaller than the vertices of $G$. In this way, we define a noncommutative product on the set of preordered forests. For example, the product of $ \tddeux{$1$}{$3$} \tdun{$2$} $ and $ \tdtroisun{$1$}{$2$}{$1$} $ gives $ \tddeux{$1$}{$3$} \tdun{$2$} \tdtroisun{$4$}{$5$}{$4$} $, whereas the product of $\tdtroisun{$1$}{$2$}{$1$}$ and $\tddeux{$1$}{$3$} \tdun{$2$}$ gives $\tdtroisun{$1$}{$2$}{$1$} \tddeux{$3$}{$5$} \tdun{$4$}$. Remark that, if $F$ and $G$ are two preordered forests, $ \max(FG) = \max(F) + \max(G)$. This product is linearly extended to $ \hh_{po} $, which in this way becomes an algebra, gradued by the number of vertices.
\\

{\bf Remark.} {The formula (\ref{prodforetpreord}) on the preordered forests extends the formula (\ref{prodforetord}) on the ordered forests.}
\\

We give some numerical values: if $ f^{\hh_{po}}_{n} $ is the number of preordered forests of vertices degree $ n $,

$$\begin{array}{c|c|c|c|c|c|c}
n&0&1&2&3&4&5 \\
\hline f_{n}^{\hh_{po}} &1&1&5&38&424&6284
\end{array}$$

If $ F $ is a preordered forest, then any subforest $G$ of $ F $ is also preordered: the total preorder on $V(G)$ is the restriction of the total preorder of $V(F)$. So we can define a coproduct $ \Delta_{\hh_{po}} : \hh_{po} \rightarrow \hh_{po} \otimes \hh_{po} $ on $ \hh_{po} $ in the following way: for all $ F \in \mathbb{F}_{\hh_{po}} $,
$$ \Delta_{\hh_{po}}(F)=\sum_{\boldsymbol{v} \models V(F)} Lea_{\boldsymbol{v}} (F) \otimes Roo_{\boldsymbol{v}} (F) .$$

For example,
\begin{eqnarray*}
\Delta_{\hh_{po}} ( \tdquatredeux{$1$}{$3$}{$2$}{$1$}) = \tdquatredeux{$1$}{$3$}{$2$}{$1$} \otimes 1 + 1 \otimes \tdquatredeux{$1$}{$3$}{$2$}{$1$} + \tdun{$1$} \otimes \tdtroisun{$1$}{$3$}{$2$} + \tddeux{$2$}{$1$} \otimes \tddeux{$1$}{$2$} + \tdun{$1$} \otimes \tdtroisdeux{$1$}{$2$}{$1$} + \tdun{$1$} \tdun{$2$} \otimes \tddeux{$1$}{$2$} + \tddeux{$2$}{$1$} \tdun{$3$} \otimes \tdun{$1$} .
\end{eqnarray*}

With this coproduct, $ \hh_{po} $ is a Hopf algebra. Remark that $ \hh_{o} $ is a Hopf subalgebra of $ \hh_{po} $.

\begin{defi}
A heap-preordered forest is a preordered forest $F$ such that if $a,b \in V(F)$, $a \neq b$ and $a \twoheadrightarrow b$, then $a$ is strictly greater than $b$ for the total preorder on $V(F)$. The set of heap-preordered forests is denoted by $\mathbb{F}_{\hh_{hpo}}$
\end{defi}

{\bf Examples.} {Heap-preordered forests of vertices degree $ \leq 3 $:
$$ 1 , \tdun{$1$} ,\tdun{$1$} \tdun{$1$} , \tdun{$1$}\tdun{$2$},\tddeux{$1$}{$2$} ,\tdun{$1$}\tdun{$1$}\tdun{$1$} ,\tdun{$1$}\tdun{$1$}\tdun{$2$},\tdun{$1$}\tdun{$2$}\tdun{$2$}, \tdun{$1$}\tdun{$2$}\tdun{$3$}, \tdun{$1$} \tddeux{$1$}{$2$} , \tdun{$1$} \tddeux{$2$}{$3$}, \tdun{$2$} \tddeux{$1$}{$2$}, \tdun{$2$} \tddeux{$1$}{$3$}, \tdun{$3$} \tddeux{$1$}{$2$},\tdtroisun{$1$}{$2$}{$2$}, \tdtroisun{$1$}{$3$}{$2$}, \tdtroisdeux{$1$}{$2$}{$3$}.$$
}

\begin{defi} \label{ordreprelineaire}
Let $ F $ be a nonempty rooted forest and $ q $ an integer $ \leq \left| F \right|_{v} $. A linear preorder is a surjection $ \sigma : V(F) \rightarrow \{1, \hdots , q \} $ such that if $ a , b \in V(F) $, $a \neq b$ and $ a \twoheadrightarrow b $ then $ \sigma (a) > \sigma (b) $. We denote by $ \mathcal{O}_{p}(F) $ the set of linear preorders on the nonempty rooted forest $ F $.
\end{defi}

{\bf Remarks.} {If $(F,\sigma)$ is a heap-preordered forest, the surjection $\sigma : V(F) \rightarrow \{1, \hdots , \max(F) \}$ is a linear preorder on $F$. Reciprocally, if $F$ is a rooted forest and $\sigma \in \mathcal{O}_{p}(F)$, then $\sigma$ define a total preorder on $V(F)$ such that $(F,\sigma)$ is a heap-preordered forest.}
\\

If $F$ and $G$ are two heap-preordered forests, then $FG$ is also heap-preordered. Moreover, any subforest $G$ of a heap-preordered forest $ F $ is also a heap-preordered forest by restriction on $V(G)$ of the total preorder of $V(F)$. So the subspace $\hh_{hpo}$ of $\hh_{po}$ generated by the heap-preordered forests is a graded Hopf subalgebra of $\hh_{po}$.\\

We give some numerical values: if $ f^{\hh_{hpo}}_{n} $ is the number of preordered forests of vertices degree $ n $,

$$\begin{array}{c|c|c|c|c|c|c}
n&0&1&2&3&4&5 \\
\hline f_{n}^{\hh_{hpo}} &1&1&3&12&64&428
\end{array}$$

We have the following diagram
\begin{eqnarray} \label{diaginclusion}
\xymatrix{
\hh_{NCK} \ar@{^{(}->}[r] & \hh_{ho} \ar@{^{(}->}[d] \ar@{^{(}->}[r] & \hh_{o} \ar@{^{(}->}[d]  \\
& \hh_{hpo} \ar@{^{(}->}[r] &  \hh_{po}
}
\end{eqnarray}
where the arrows $ \hookrightarrow $ are injective morphisms of Hopf algebras (for the cut coproduct).

\subsection{A morphism from $\hh_{po}$ to $\WQSyms$}

In this section, we give a similar result of proposition \ref{dehdansfqsym} in the preordered case.

\begin{defi} \label{defivarphipreordered}
Let $(F,\sigma)$ be a nonempty preordered forest of vertices degree $n$ and $\tau \in Surj_{n}$. Then we denote by $\mathbb{S}_{(F,\sigma)}^{\tau}$ the set of bijective maps $\varphi : V(F) \rightarrow \{1, \hdots , n \}$ such that:
\begin{enumerate}
\item if $v \in V(F)$, then $\sigma (v) = \tau(\varphi (v))$,
\item if $v,v' \in V(F)$, $v' \twoheadrightarrow v$, then $\varphi(v) \geq \varphi (v')$.
\end{enumerate}
\end{defi}

{\bf Remark.} {\begin{enumerate}
\item If $ \max (F) \neq \max (\tau)$, then $\mathbb{S}_{(F,\sigma)}^{\tau} = \emptyset$.
\item Let $F,G \in \mathbb{F}_{\hh_{po}}$, $\left| F \right|_{v} = k$, $\left| G \right|_{v} = l$. If $\varphi_{1} : V(F) \rightarrow \{1,\hdots,k\}$ and $\varphi_{2} : V(G) \rightarrow \{1,\hdots,l\}$ are two bijective maps  and $\zeta \in Sh(k,l)$, then $ \zeta \circ (\varphi_{1} \otimes \varphi_{2}) : V(FG) \rightarrow \{1,\hdots,k+l\}$, where $\varphi_{1} \otimes \varphi_{2}$ is defined in formula (\ref{prodforetord}), is also a bijective map. Similary, considering a bijective map $\varphi :V(FG) \rightarrow \{1, \hdots,k+l\} $ and $\zeta \in Sh(k,l)$. Then $\varphi$ can be uniquely written as $\zeta \circ (\varphi_{1} \otimes \varphi_{2})$, where $\varphi_{1} : V(F) \rightarrow \{1,\hdots,k\}$ and $\varphi_{2} : V(G) \rightarrow \{1,\hdots,l\}$ are two bijective maps.
\end{enumerate}}

\begin{theo} \label{dehdanswgsym}
Let us define:
\begin{eqnarray} \label{Theta}
\Phi : \left\lbrace 
\begin{array}{rcl}
\hh_{po} & \rightarrow & \WQSyms \\
(F,\sigma) \in \mathbb{F}_{\hh_{po}} & \mapsto & \displaystyle\sum_{\tau \in Surj_{\left| F \right|_{v}}} {\rm card}(\mathbb{S}_{(F,\sigma)}^{\tau}) \ \tau .
\end{array} \right.
\end{eqnarray}
Then $\Phi : \hh_{po} \rightarrow \WQSyms$ is a Hopf algebra morphism, homogeneous of degree 0.
\end{theo}

{\bf Examples.} {\begin{itemize}
\item In vertices degree 1: $\Phi(\tdun{$1$}) = (1)$.
\item In vertices degree 2:
\begin{eqnarray*}
\Phi (\tdun{$1$} \tdun{$1$}) = 2 (1 1) , \hspace{0.5cm} \Phi (\tdun{$1$} \tdun{$2$}) = (1 2) + (2 1) , \hspace{0.5cm} \Phi (\tddeux{$b$}{$a$}) = (a b).
\end{eqnarray*}
\item In vertices degree 3:
\begin{eqnarray*}
\begin{array}{rclrcl}
\Phi (\tdun{$1$} \tdun{$1$} \tdun{$1$}) & = & 6 (1 1 1) & \Phi (\tddeux{$1$}{$2$} \tdun{$2$}) & = & (2 1 2) + 2 (2 2 1) \\
\Phi (\tdtroisun{$2$}{$1$}{$2$}) & = & (1 2 2) + (2 1 2) & \Phi (\tdun{$2$} \tddeux{$3$}{$1$}) & = & (2 1 3) + (1 2 3) + (1 3 2) \\
\Phi (\tdtroisdeux{$1$}{$3$}{$2$}) & = & (2 3 1)  & \Phi (\tdun{$1$} \tddeux{$3$}{$2$}) & = & (1 2 3) + (2 1 3) + (2 3 1) \\
\Phi (\tdtroisun{$1$}{$2$}{$2$}) & = & 2 (2 2 1) & \Phi (\tdun{$1$} \tdun{$1$} \tdun{$2$}) & = & 2 \left[ (1 1 2) + (1 2 1) + (2 1 1) \right] 
\end{array}
\end{eqnarray*}
\end{itemize}
\vspace{0.5cm}

\begin{proof}
Obviously, $\Phi$ is homogeneous of degree 0. Let $(F,\sigma^{F}), (G,\sigma^{G}) \in \mathbb{F}_{\hh_{po}}$, $\left| F \right|_{v} = k$, $\left| G \right|_{v} = l$ and $\tau \in Surj_{k+l}$. $\tau$ can be uniquely written as $\tau = (\tau_{1} \otimes \tau_{2}) \circ \zeta^{-1}$ with $\tau_{1} \in Surj_{k}$, $\tau_{2} \in Surj_{l}$ and $\zeta \in Sh(k,l)$.\\

Let $\varphi \in \mathbb{S}_{(FG,\sigma^{FG})}^{\tau}$. Then $\varphi$ can be uniquely written as $\zeta \circ (\varphi_{1} \otimes \varphi_{2})$ with $\varphi_{1} : V(F) \rightarrow \{1, \hdots , k \}$ and $\varphi_{2}: V(G) \rightarrow \{1, \hdots ,l \}$ two bijective maps.

\begin{enumerate}
\item \begin{enumerate}
\item If $v \in V(F)$, then 
$$\sigma^{F}(v) = \sigma^{FG}(v) = \tau (\varphi(v)) = (\tau_{1} \otimes \tau_{2}) \circ \zeta^{-1} \circ \zeta \circ (\varphi_{1} \otimes \varphi_{2})(v) = \tau_{1}(\varphi_{1}(v)) .$$
Note that with this equality, we also have that $\max (F) = \max (\tau_{1})$.
\item If $v \in V(G)$, then
\begin{eqnarray*}
\sigma^{G}(v) + \max (F) & = & \sigma^{FG}(v) = \tau (\varphi(v)) = (\tau_{1} \otimes \tau_{2}) \circ \zeta^{-1} \circ \zeta \circ (\varphi_{1} \otimes \varphi_{2})(v) \\
& = & \tau_{2}(\varphi_{2}(v)) + \max (\tau_{1}) .
\end{eqnarray*}
As $ \max (F) = \max (\tau_{1})$, $\sigma^{G}(v) = \tau_{2}(\varphi_{2}(v))$.
\end{enumerate}
\item \begin{enumerate}
\item If $v' \twoheadrightarrow v$ in $F$, then $v'\twoheadrightarrow v$ in $FG$, so:
\begin{eqnarray*}
\varphi (v) & \geq & \varphi (v') \\
\zeta \circ (\varphi_{1} \otimes \varphi_{2})(v) & \geq & \zeta \circ (\varphi_{1} \otimes \varphi_{2})(v') \\
\zeta \left( \varphi_{1} (v) \right) & \geq & \zeta \left( \varphi_{1}(v') \right) \\
\varphi_{1} (v) & \geq & \varphi_{1}(v'),
\end{eqnarray*}
as $\zeta$ is increasing on $\{1,\hdots,k\}$.
\item If $v' \twoheadrightarrow v$ in $G$, then $v'\twoheadrightarrow v$ in $FG$, so:
\begin{eqnarray*}
\varphi (v) & \geq & \varphi (v') \\
\zeta \circ (\varphi_{1} \otimes \varphi_{2})(v) & \geq & \zeta \circ (\varphi_{1} \otimes \varphi_{2})(v') \\
\zeta \left( k + \varphi_{2} (v) \right) & \geq & \zeta \left( k + \varphi_{2}(v') \right) \\
\varphi_{2} (v) & \geq & \varphi_{2}(v'),
\end{eqnarray*}
as $\zeta$ is increasing on $\{k+1,\hdots,k+l\}$.
\end{enumerate}
\end{enumerate}
So $\varphi_{1} \in \mathbb{S}^{\tau_{1}}_{(F,\sigma^{F})}$ and $\varphi_{2} \in \mathbb{S}^{\tau_{2}}_{(G,\sigma^{G})}$.\\

Conversely, if $\varphi = \zeta \circ (\varphi_{1} \otimes \varphi_{2})$, with $\varphi_{1} \in \mathbb{S}^{\tau_{1}}_{(F,\sigma^{F})}$ and $\varphi_{2} \in \mathbb{S}^{\tau_{2}}_{(G,\sigma^{G})}$, the same computations shows that $\varphi \in \mathbb{S}^{(\tau_{1} \otimes \tau_{2}) \circ \zeta^{-1}}_{(FG,\sigma^{FG})}$.

So
$$ {\rm card} (\mathbb{S}_{(FG,\sigma^{FG})}^{\tau}) = {\rm card} (\mathbb{S}_{(F,\sigma^{F})}^{\tau_{1}}) \times {\rm card} (\mathbb{S}_{(G,\sigma^{G})}^{\tau_{2}})$$
and
\begin{eqnarray*}
\Phi ((FG,\sigma^{FG})) & = & \sum_{\tau \in Surj_{k+l}} \mbox{card}(\mathbb{S}_{(FG,\sigma^{FG})}^{\tau}) \ \tau \\
& = & \sum_{\zeta \in Sh(k,l)} \sum_{\tau_{1} \in Surj_{k}} \sum_{\tau_{2} \in Surj_{l}} \mbox{card}(\mathbb{S}_{(FG,\sigma^{FG})}^{(\tau_{1} \otimes \tau_{2}) \circ \zeta^{-1}}) \ (\tau_{1} \otimes \tau_{2}) \circ \zeta^{-1} \\
& = & \sum_{\zeta \in Sh(k,l)} \sum_{\tau_{1} \in Surj_{k}} \sum_{\tau_{2} \in Surj_{l}} \mbox{card}(\mathbb{S}_{(F,\sigma^{F})}^{\tau_{1}}) \times \mbox{card}(\mathbb{S}_{(G,\sigma^{G})}^{\tau_{2}}) \ (\tau_{1} \otimes \tau_{2}) \circ \zeta^{-1} \\
& = & \left( \sum_{\tau_{1} \in Surj_{k}} \mbox{card}(\mathbb{S}_{(F,\sigma^{F})}^{\tau_{1}}) \ \tau_{1} \right) \left( \sum_{\tau_{2} \in Surj_{l}} \mbox{card}(\mathbb{S}_{(G,\sigma^{G})}^{\tau_{2}}) \ \tau_{2} \right) \\
& = & \Phi ((F,\sigma^{F})) \Phi ((G,\sigma^{G})) .
\end{eqnarray*}
So $\Phi$ is an algebra morphism.\\

Let $(F,\sigma) \in \mathbb{F}_{\hh_{po}}$ be a preordered forest such that $\left| F \right|_{v} = n$ and let $\boldsymbol{v}$ be an admissible cut of $F$. We obtain two preordered forests $(Lea_{\boldsymbol{v}}(F),\sigma_{1})$ and $(Roo_{\boldsymbol{v}}(F),\sigma_{2})$. We set $k = \left| Lea_{\boldsymbol{v}}(F) \right|_{v}$ and $l = \left| Roo_{\boldsymbol{v}}(F) \right|_{v}$.

Let $\tau_{1} \in Surj_{k}$, $\tau_{2} \in Surj_{l}$ and $\varphi_{1} \in \mathbb{S}^{\tau_{1}}_{(Lea_{\boldsymbol{v}}(F),\sigma_{1})}$, $\varphi_{2} \in \mathbb{S}^{\tau_{2}}_{(Roo_{\boldsymbol{v}}(F),\sigma_{2})}$. We set $\varphi = \varphi_{1} \otimes \varphi_{2} $ and we define $\tau$ by $\tau = \sigma \circ \varphi^{-1}$. $\tau \in Surj_{n}$ and $ \max (\tau) = \max (F)$. Let us show that $\varphi \in \mathbb{S}^{\tau}_{(F,\sigma)}$.
\begin{enumerate}
\item By definition, $\tau = \sigma \circ \varphi^{-1}$. So $\sigma (v) = \tau (\varphi (v))$ for all $v \in V(F)$.
\item If $v' \twoheadrightarrow v$ in $F$, then three cases are possible:
\begin{enumerate}
\item $v$ and $v'$ belong to $V(Lea_{\boldsymbol{v}}(F))$. As $\varphi_{1} \in \mathbb{S}^{\tau_{1}}_{(Lea_{\boldsymbol{v}}(F),\sigma_{1})}$ , $\varphi_{1}(v) \geq \varphi_{1}(v')$. Then $\varphi (v) = (\varphi_{1} \otimes \varphi_{2})(v) = \varphi_{1}(v) \geq \varphi_{1}(v') = (\varphi_{1} \otimes \varphi_{2})(v') = \varphi (v')$.
\item $v$ and $v'$ belong to $V(Roo_{\boldsymbol{v}}(F))$. As $\varphi_{2} \in \mathbb{S}^{\tau_{2}}_{(Lea_{\boldsymbol{v}}(F),\sigma_{2})}$ , $\varphi_{2}(v) \geq \varphi_{2}(v')$. Then $\varphi (v) = (\varphi_{1} \otimes \varphi_{2})(v) = \varphi_{2}(v) + k \geq \varphi_{2}(v') + k = (\varphi_{1} \otimes \varphi_{2})(v') = \varphi (v')$.
\item $v'$ belong to $V(Lea_{\boldsymbol{v}}(F))$ and $v$ belong to $V(Roo_{\boldsymbol{v}}(F))$. Then $\varphi (v') = (\varphi_{1} \otimes \varphi_{2})(v') = \varphi_{1}(v') \in \{1, \hdots , k \}$ and $\varphi(v) = (\varphi_{1} \otimes \varphi_{2})(v) = \varphi_{2}(v) + k \in \{k+1, \hdots , k+l \}$. So $\varphi(v) > \varphi (v')$.
\end{enumerate}
In any case, $\varphi (v) \geq \varphi (v')$.
\end{enumerate}

Conversely, let $(F,\sigma) \in \mathbb{F}_{\hh_{po}}$ be a preordered forest of vertices degree $n$, $\tau \in Surj_{n}$ and $\varphi \in \mathbb{S}^{\tau}_{(F,\sigma)}$. Let $k \in \{0, \hdots ,n\}$ be an integer.

We set $\tau_{1}^{(k)}$ and $\tau_{2}^{(k)}$ the words obtained by cutting the word representing $\tau$ between the $k$-th and the $(k+1)$-th letter, and then packing the two obtained words.

Moreover, we define $\boldsymbol{v}$ a subset of $\varphi^{-1}(\{1, \hdots ,k \})$ such that $v \twoheadrightarrow w \hspace{-.7cm} / \hspace{.7cm}$ for any couple $(v,w)$ of two different elements of $\boldsymbol{v}$. Then $\boldsymbol{v} \models V(F)$ and we considere the two preordered forests $(Lea_{\boldsymbol{v}}(F),\sigma_{1}^{(k)})$ and $(Roo_{\boldsymbol{v}}(F),\sigma_{2}^{(k)})$. Remark that, with the second point of definition \ref{defivarphipreordered}, $V(Lea_{\boldsymbol{v}}(F)) = \varphi^{-1}(\{1, \hdots ,k \})$ and $V(Roo_{\boldsymbol{v}}(F)) = \varphi^{-1}(\{k+1, \hdots ,n \})$.

We set $\varphi_{1}^{(k)} : v \in V(Lea_{\boldsymbol{v}}(F)) \rightarrow \varphi (v) \in \{1, \hdots ,k\}$ and $\varphi_{2}^{(k)} : v \in V(Roo_{\boldsymbol{v}}(F)) \rightarrow \varphi (v) - k \in \{1, \hdots ,n-k\}$. Thus $\varphi = \varphi_{1}^{(k)} \otimes \varphi_{2}^{(k)}$.

Let us prove that $\varphi_{1}^{(k)} \in \mathbb{S}^{\tau_{1}^{(k)}}_{(Lea_{\boldsymbol{v}}(F),\sigma_{1}^{(k)})}$ and $\varphi_{2}^{(k)} \in \mathbb{S}^{\tau_{2}^{(k)}}_{(Roo_{\boldsymbol{v}}(F),\sigma_{2}^{(k)})}$.
\begin{enumerate}
\item \begin{enumerate}
\item If $v \in V(Lea_{\boldsymbol{v}}(F))$, $\varphi (v) = \varphi_{1}^{(k)}(v) \in \{1,\hdots,k\}$ and then
$$ \sigma_{1}^{(k)}(v) = pack \circ \sigma (v) = pack \circ \tau \circ \varphi (v) = \tau_{1}^{(k)} \circ \varphi_{1}^{(k)} (v) .$$
\item If $v \in V(Roo_{\boldsymbol{v}}(F))$, $\varphi (v) = \varphi_{2}^{(k)}(v) + k \in \{k+1, \hdots , n \}$ and then
$$ \sigma_{2}^{(k)}(v) = pack \circ \sigma (v) = pack \circ \tau \circ \varphi (v) = \tau_{2}^{(k)} \circ \varphi_{2}^{(k)}(v) .$$
\end{enumerate}
\item \begin{enumerate}
\item If $v' \twoheadrightarrow v$ in $Lea_{\boldsymbol{v}}(F)$, then $v' \twoheadrightarrow v$ in $F$ and $\varphi_{1}^{(k)}(v) = \varphi(v) \geq \varphi(v') = \varphi_{2}^{(k)}(v')$.
\item If $v' \twoheadrightarrow v$ in $Roo_{\boldsymbol{v}}(F)$, then $v' \twoheadrightarrow v$ in $F$ and $\varphi_{2}^{(k)}(v) = \varphi(v)-k \geq \varphi(v')-k = \varphi_{2}^{(k)}(v')$.
\end{enumerate}
\end{enumerate}

Hence, there is a bijection:
\begin{eqnarray*}
\left\lbrace  \begin{array}{rcl}
\mathbb{S}^{\tau}_{(F,\sigma)} \times \{0,\hdots ,\left| F \right|_{v} \} & \rightarrow & \displaystyle\bigsqcup_{\boldsymbol{v} \models V(F)} \mathbb{S}^{\tau_{1}^{(k)}}_{(Lea_{\boldsymbol{v}}(F),\sigma_{1}^{(k)})} \times \mathbb{S}^{\tau_{2}^{(k)}}_{(Roo_{\boldsymbol{v}}(F),\sigma_{2}^{(k)})} \\
(\varphi , k) & \mapsto & \left( \varphi_{1}^{(k)}, \varphi^{(k)}_{2} \right) .
\end{array}
\right.
\end{eqnarray*}
Finally,
\begin{eqnarray*}
& & \Delta_{\WQSyms} \circ \Phi ((F,\sigma)) \\
& = & \sum_{\tau \in Surj_{\left| F \right|_{v}}} \sum_{0 \leq k \leq n} {\rm card} (\mathbb{S}^{\tau}_{(F,\sigma)}) \ \tau_{1}^{(k)} \otimes \tau_{2}^{(k)} \\
& = & \sum_{\boldsymbol{v} \models V(F)} \sum_{\tau_{1} \in Surj_{\left| Lea_{\boldsymbol{v}}(F) \right|_{v}}} \sum_{\tau_{2} \in Surj_{\left| Roo_{\boldsymbol{v}}(F) \right|_{v}}} {\rm card} (\mathbb{S}^{\tau_{1}}_{(Lea_{\boldsymbol{v}}(F),\sigma_{1}^{(k)})}) \ \tau_{1} \otimes {\rm card} (\mathbb{S}^{\tau_{2}}_{(Roo_{\boldsymbol{v}}(F),\sigma_{2}^{(k)})} ) \ \tau_{2} \\
& = & (\Phi \otimes \Phi) \circ \Delta_{\hh_{po}} .
\end{eqnarray*}
So $\Phi$ is a coalgebra morphism.
\end{proof}

\begin{theo}
The restriction of $\Phi$ defined in formula (\ref{Theta}) to $\hh_{hpo}$ is an injection of graded Hopf algebras.
\end{theo}

\begin{proof}
We introduce a lexicographic order on the words with letters $\in \mathbb{N}^{\ast}$. Let $u = (u_{1} \hdots u_{k})$ and $v = (v_{1} \hdots v_{l}) $ be two words. Then
\begin{itemize}
\item if $u_{k} = v_{k}, u_{k-1} = v_{k-1} , \hdots , u_{i+1} = v_{i+1}$ and $u_{i} > v_{i} $ (resp. $u_{i} < v_{i}$) with $i \in \{1,\hdots , \min (k,l) \}$, then $u > v$ (resp. $u < v$),
\item if $u_{i} = v_{i}$ for all $i \in \{1,\hdots , \min (k,l) \}$ and if $k > l$ (resp. $k<l$) then $u > v$ (resp. $u<v$).
\end{itemize}
For example,
\begin{eqnarray*}
\begin{array}{ccccc}
(541) < (22), & (433) < (533), & (5362) < (72), & (8225) < (1327), & (215) < (1215) .
\end{array}
\end{eqnarray*}
If $u$ and $v$ are two words, we denote by $uv$ the concatenation of $u$ and $v$.\\

In this proof, if $(F,\sigma)$ is a preordered forest, we consider $F$ as a decorated forest where the vertices are decorated by integers. Consider
$$ \mathbb{F} = \left\lbrace (F,d) \ \mid \ F \in \mathbb{F}_{\mathbf{H}_{CK}}, d:V(F) \rightarrow \mathbb{N}^{\ast} \mbox{ such that if } v \rightarrow w \mbox{ then } d(v) > d(w) \right\rbrace $$
the set of forests with their vertices decorated by nonzero integers and with an increasing condition.
\\

Let $(F,d) \in \mathbb{F}$ be a forest of vertices degree $n$ and if $u = (u_{1} \hdots u_{n})$ is a word of length $n$ with $u_{i} \in \mathbb{N}^{\ast}$. In the same way that definition \ref{defivarphipreordered}, we define $\mathbb{S}^{u}_{(F,d)}$ as the set of bijective maps $\varphi : V(F) \rightarrow \{1, \hdots ,n\}$ such that:
\begin{enumerate}
\item if $v \in V(F)$, then $ d(v) = u_{\varphi (v)}$,
\item if $v,v' \in V(F)$, $v' \twoheadrightarrow v$, then $\varphi(v) \geq \varphi (v')$.
\end{enumerate}
For example,
\begin{itemize}
\item if $(F,d) = \tdquatredeux{$2$}{$4$}{$3$}{$7$} \in \mathbb{F}$, then the words $u$ such that $\mathbb{S}^{u}_{(F,d)} \neq \emptyset$ are $(7342), (7432),(4732)$.
\item if $(F,d) = \tdtroisun{$1$}{$4$}{$3$} \tddeux{$3$}{$6$} \in \mathbb{F} $, then the words $u$ such that $\mathbb{S}^{u}_{(F,d)} \neq \emptyset$ are
\begin{eqnarray*}
\begin{array}{c}
(43163), (43613), (46313), (64313) , (43631), (46331), (64331), (63431),  \\
(34163), (34613), (36413), (63413), (34631), (36431), (36341), (63341) .
\end{array}
\end{eqnarray*}
\end{itemize}

Let $(F,d)$ be a forest of $\mathbb{F}$. Then we set
$$ m((F,d)) = \max \left( \{ u \ \mid \ \mathbb{S}^{u}_{(F,d)} \neq \emptyset \} \right) .$$
For example, for $(F,d) = \tdquatredeux{$2$}{$4$}{$3$}{$7$} \in \mathbb{F}$, $m((F,d)) = (7342)$ and for $(F,d) = \tdtroisun{$1$}{$4$}{$3$} \tddeux{$3$}{$6$} \in \mathbb{F}$, $m((F,d)) = (34163)$.
\vspace{0.2cm}

If $(F,d) \in \mathbb{F}$ is the empty tree, $m((F,\sigma)) = 1$. Let $(F,d) \in \mathbb{F}$ be a nonempty tree of vertices degree $n$. We denote by $(G,d')$ the forest of $\mathbb{F}$ obtained by deleting the root of $F$. Then, if $m((F,d)) = (u_{1} \hdots u_{n})$, we have $m((G,d')) = (u_{1} \hdots u_{n-1})$ and $u_{n} = d(R_{F})$ the decoration of the root of $F$. Let $(F,d)$ be a forest of vertices degree $n$, $(F,d)$ is the disjoint union of trees with their vertices decorated by nonzero integers $ (F_{1},d_{1}), \hdots ,(F_{k},d_{k})$ ordered such that $m((F_{1},d_{1})) \leq \hdots \leq m((F_{k},d_{k}))$. Then $m((F,d)) = m((F_{1},d_{1})) \hdots m((F_{k},d_{k})) $:
\begin{itemize}
\item By definition, $\mathbb{S}^{m((F_{i},d_{i}))}_{(F_{i},d_{i})} \neq \emptyset$ and if $\varphi_{i} \in \mathbb{S}^{m((F_{i},d_{i}))}_{(F_{i},d_{i})}$ then $\varphi : V(F) \rightarrow \{1, \hdots ,n\}$, defined for all $1 \leq i \leq k$ and $v \in V(F_{i})$ by $\varphi (v) = \varphi_{i}(v)$, is an element of $ \mathbb{S}^{m((F_{1},d_{1})) \hdots m((F_{k},d_{k}))}_{(F,d)}$ and $\mathbb{S}^{m((F_{1},d_{1})) \hdots m((F_{k},d_{k}))}_{(F,d)} \neq \emptyset$. So $m((F,d)) \geq m((F_{1},d_{1})) \hdots m((F_{k},d_{k})) $.
\item If $\mathbb{S}^{u}_{(F,d)} \neq \emptyset$, $u$ is the shuffle of $u_{1}, \hdots , u_{k}$ such that $\mathbb{S}^{u_{i}}_{(F_{i},d_{i})} \neq \emptyset$ (see the proof of theorem \ref{dehdanswgsym}). In particular, $u_{i} \leq m((F_{i},d_{i}))$, so $u \leq m((F_{1},d_{1})) \hdots m((F_{k},d_{k}))$ and $m((F,d)) \leq m((F_{1},d_{1})) \hdots m((F_{k},d_{k}))$.
\end{itemize}

Let $(F,d) \in \mathbb{F}$ be a forest of vertices degree $n$ and $m((F,d)) = (u_{1} \hdots u_{n})$. Let $i_{1}$ be the smallest index such that $u_{1} , \hdots , u_{i_{1}-1} > u_{i_{1}} $ and, for all $j > i_{1}$, $u_{i_{1}} \leq u_{j}$. By construction, there exists a connected component $(F_{1},d_{1})$ of $(F,d)$ such that $m((F_{1},d_{1})) = (u_{1} \hdots u_{i_{1}})$. Consider the word $(u_{i_{1}+1} \hdots u_{n})$. Let $i_{2} > i_{1}$ be the smallest index such that $u_{i_{1}+1} , \hdots , u_{i_{2}-1} > u_{i_{2}} $ and, for all $j > i_{2}$, $u_{i_{2}} \leq u_{j}$. Then there exists a connected component $(F_{2},d_{2})$ (different from $(F_{1},d_{1})$) such that $m((F_{2},d_{2})) = (u_{i_{1} + 1} \hdots u_{i_{2}})$. In the same way, we construct $i_{3}, \hdots , i_{k}$ and $(F_{3},d_{3}), \hdots, (F_{k},d_{k})$. Then we have $m((F,d)) = m((F_{1},d_{1})) \hdots m((F_{k},d_{k}))$
\vspace{0.2cm}

Let us prove that $m$ is injective on $\mathbb{F}$ by induction on the vertices degree. If $(F,d)$ is the empty tree, it is obvious. Let $(F,d)$ be a nonempty forest of $\mathbb{F}$ of vertices degree $n$.
\begin{itemize}
\item If $(F,d)$ is a tree, $m((F,d)) = (u_{1} \hdots u_{n-1} u_{n})$ with $u_{n} = d(R_{F})$ the decoration of the root of $F$. Let $(G,d')$ be the forest of $\mathbb{F}$ obtained by deleting the root of $F$. Then $m((G,d')) = (u_{1} \hdots u_{n-1})$. By induction hypothesis, $(G,d')$ is the unique forest of $\mathbb{F}$ such that $m((G,d')) = (u_{1} \hdots u_{n-1})$. So $(F,d)$ is also the unique forest of $\mathbb{F}$ such that $m((F,\sigma)) = (u_{1} \hdots u_{n-1} d(R_{F}) )$.
\item If $(F,d)$ is not a tree, then $(F,d)$ is the product of trees $ (F_{1},d_{1}), \hdots ,(F_{k},d_{k})$ of $\mathbb{F}$ ordered such that $m((F_{1},d_{1})) \leq \hdots \leq m((F_{k},d_{k}))$. So $m((F,d)) = m((F_{1},d_{1})) \hdots m((F_{k},d_{k}))$. By the induction hypothesis, for all $1 \leq i \leq k$, $(F_{i},d_{i})$ is the unique tree of $\mathbb{F}$ such that its image by $m$ is $m((F_{i},d_{i}))$. So the product $(F,d)$ of $(F_{i},d_{i})$'s is the unique forest of $\mathbb{F}$ such that its image by $m$ is $m((F,d))$.
\end{itemize}

So $m$ is injective on $\mathbb{F}$. By triangularity, $m$ is injective on $\mathbb{F}_{\hh_{hpo}}$ and we deduce that the restriction of $\Phi$ to $\mathbf{H}_{hpo}$ is an injection of graded Hopf algebras.
\end{proof}

\section{Hopf algebras of contractions}

\subsection{Commutative case}

In \cite{Calaque11}, D. Calaque, K. Ebrahimi-Fard and D. Manchon introduce a new coproduct, called in this paper the contraction coproduct, on the augmentation ideal of $ \hh_{CK} $ (see also \cite{Manchon08}).

\begin{defi} \label{defipartcom}
Let $ F $ be a nonempty rooted forest and $ \boldsymbol{e} $ a subset of $ E(F) $. Then we denote by
\begin{enumerate}
\item $ Part_{\boldsymbol{e}}(F) $ the subforest of $ F $ obtained by keeping all the vertices of $ F $ and the edges of $ \boldsymbol{e} $,
\item $ Cont_{\boldsymbol{e}}(F) $ the forest obtained by contracting each edge of $ \boldsymbol{e} $ in $F$ and identifying the two extremities of each edge of $ \boldsymbol{e} $.
\end{enumerate}
We shall say that $ \boldsymbol{e} $ is a contraction of $ F $, $ Part_{\boldsymbol{e}}(F) $ is the partition of $ F $ by $ \boldsymbol{e} $ and $ Cont_{\boldsymbol{e}}(F) $ is the contracted of $ F $ by $ \boldsymbol{e} $. Each vertex of $ Cont_{\boldsymbol{e}}(F) $ can be identified to a connected component of $ Part_{\boldsymbol{e}}(F) $.
\end{defi}

{\bf Remarks.} {
\begin{itemize}
\item If $ \boldsymbol{e} = \emptyset $, then $ Part_{\boldsymbol{e}}(F) = \underbrace{\tun \hdots \tun}_{\left| F \right|_{v} \times} $ and $ Cont_{\boldsymbol{e}}(F) = F $: this is the \textit{empty contraction} of $ F $.
\item If $ \boldsymbol{e} = E(F) $, then $ Part_{\boldsymbol{e}}(F) = F $ and $ Cont_{\boldsymbol{e}}(F) = \tun $: this is the \textit{total contraction} of $ F $.
\end{itemize}
}
\vspace{0.5cm}

{\bf Notations.} {We shall write $ \boldsymbol{e} \models E(F) $ if $ \boldsymbol{e} $ is a contraction of $ F $ and $ \boldsymbol{e} \mmodels E(F) $ if $ \boldsymbol{e} $ is a nonempty, nontotal contraction of $ F $.}
\\

{\bf Example.} {Let $ T = \tquatredeux $ be a rooted tree. Then
$$\begin{array}{|c|c|c|c|c|c|c|c|c|c|}
\hline \mbox{contraction } \boldsymbol{e} &\tquatredeux&\tquatredeuxa&\tquatredeuxb&\tquatredeuxc
&\tquatredeuxd&\tquatredeuxe&\tquatredeuxf&\tquatredeuxg\\
\hline Part_{\boldsymbol{e}}(T)&\tquatredeux&\tdeux \tdeux & \tun \ttroisun& \tun \ttroisdeux& \tun \tun \tdeux & \tun \tun \tdeux &\tun \tun \tdeux & \tun \tun \tun \tun \\
\hline Cont_{\boldsymbol{e}}(T)& \tun & \tdeux & \tdeux & \tdeux & \ttroisdeux & \ttroisun& \ttroisun & \tquatredeux \\
\hline \end{array}$$
where, in the first line, the edges not belonging to $ \boldsymbol{e} $ are striked out.
}
\\

{\bf Remarks.} {Let $ F $ be a nonempty rooted forest and $ \boldsymbol{e} \models E(F) $.
\begin{enumerate}
\item We have the following relation on the vertices degrees:
\begin{eqnarray*}
\left| F \right|_{v} = \left| Cont_{\boldsymbol{e}}(F) \right|_{v} + \left| Part_{\boldsymbol{e}}(F) \right|_{v} - l(Part_{\boldsymbol{e}}(F) ) .
\end{eqnarray*}
\item Note $\overline{\boldsymbol{e}}$ the complementary to $\boldsymbol{e}$ in $E(F)$. Then $E(Part_{\boldsymbol{e}}(F)) = \boldsymbol{e}$ and $E(Cont_{\boldsymbol{e}}(F)) = \overline{\boldsymbol{e}}$ and
\begin{eqnarray} \label{relatdeg}
\left| F \right|_{e} = \left| Cont_{\boldsymbol{e}}(F) \right|_{e} + \left| Part_{\boldsymbol{e}}(F) \right|_{e} .
\end{eqnarray}
\end{enumerate}
}

Let $ \cc_{CK} $ be the quotient algebra $ \hh_{CK}/I_{CK} $ where $ I_{CK} $ is the ideal spanned by $ \tun - 1 $. In others terms, one identifies the unit $ 1 $ (for the concatenation) with the tree $ \tun $. We note in the same way a rooted forest and his class in $\cc_{CK}$. Then we define on $ \cc_{CK} $ a contraction coproduct on each forest $ F \in \cc_{CK} $:
\begin{eqnarray*}
\Delta_{\cc_{CK}} (F) & = & \sum_{\boldsymbol{e} \models E(F)} Part_{\boldsymbol{e}}(F) \otimes Cont_{\boldsymbol{e}}(F) ,\\
& = & F \otimes \tun + \tun \otimes F + \sum_{\boldsymbol{e} \mmodels E(F)} Part_{\boldsymbol{e}}(F) \otimes Cont_{\boldsymbol{e}}(F) .
\end{eqnarray*}
In particular, $ \Delta_{\cc_{CK}} (\tun) = \tun \otimes \tun $.\\

{\bf Example.} {
\begin{eqnarray*}
\Delta_{\cc_{CK}} ( \tquatredeux ) = \tun \otimes \tquatredeux + \tquatredeux \otimes \tun + 2 \tdeux \otimes \ttroisun + \tdeux \otimes \ttroisdeux + \ttroisdeux \otimes \tdeux + \ttroisun \otimes \tdeux + \tdeux \tdeux \otimes \tdeux .
\end{eqnarray*}
}
\\

We define an algebra morphism $ \varepsilon $:
\begin{eqnarray*}
\varepsilon : \left\lbrace \begin{array}{rcl}
\cc_{CK} & \rightarrow & \mathbb{K} \\
F \mbox{ forest } & \mapsto & \delta_{F,\tun}.
\end{array} \right. 
\end{eqnarray*}

Then $ (\cc_{CK},\Delta_{\cc_{CK}},\varepsilon) $ is a commutative Hopf algebra graded by the number of edges. $ \cc_{CK} $ is non cocommutative (see for example the coproduct of $ \tquatredeux $).\\

{\bf Remark.} {We define inductively:
\begin{eqnarray*}
\Delta_{\cc_{CK}}^{(0)} = Id, \hspace{1cm} \Delta_{\cc_{CK}}^{(1)} = \Delta_{\cc_{CK}}, \hspace{1cm} \Delta_{\cc_{CK}}^{(k)} = (\Delta_{\cc_{CK}} \otimes Id^{\otimes (k-1)}) \circ \Delta_{\cc_{CK}}^{(k-1)}.
\end{eqnarray*}
For all $ k \in \mathbb{N} $, $ \Delta_{\cc_{CK}}^{(k)} : \cc_{CK} \rightarrow \cc_{CK}^{\otimes (k+1)} $. If $ F $ is a rooted forest with $ n $ edges, there are $ (k+1)^{n} $ terms in the expression of $ \Delta_{\cc_{CK}}^{(k)} (F) $:
\begin{itemize}
\item If $ k=0 $, this is obvious.
\item If $ k > 0 $, we have $ \left( \substack{n \\ l} \right) $ tensors $ F^{(1)} \otimes F^{(2)} $ in $ \Delta_{\cc_{CK}} (F) $ such that the left term $ F^{(1)} $ have $ l $ edges. By the induction hypothesis, there are $ k^{l} $ terms in $ \Delta_{\cc_{CK}}^{(k-1)} (F^{(1)}) $. So there are $ \sum_{0 \leq l \leq n} \left( \substack{n \\ l} \right) k^{l} = (k+1)^{n} $ terms in the expression of $ \Delta_{\cc_{CK}}^{(k)} (F) $.
\end{itemize}
}
\vspace{0.5cm}

We give the first numbers of trees $ t_{n}^{\cc_{CK}} $ and forests $ f_{n}^{\cc_{CK}} $:

$$\begin{array}{c|c|c|c|c|c|c|c|c|c|c|c}
n&0&1&2&3&4&5&6&7&8&9&10\\
\hline t_{n}^{\cc_{CK}} &1&1&2&4&9&20&48&115&286&719&1842 \\
\hline f_{n}^{\cc_{CK}} &1&1&3&7&19&47&127&330&889&2378&6450
\end{array} $$

The first sequence is the sequence A000081 in \cite{Sloane}.\\

We recall a combinatorial description of the antipode $ S_{\cc_{CK}} : \cc_{CK} \rightarrow \cc_{CK} $ (see \cite{Calaque11}):

\begin{prop} \label{antipodecom}
The antipode $ S_{\cc_{CK}} : \cc_{CK} \rightarrow \cc_{CK} $ of the Hopf algebra $ (\cc_{CK},\Delta_{\cc_{CK}},\varepsilon) $ is given (recursively with respect to number of edges) by the following formulas: for all forest $ F \in \cc_{CK} $,
\begin{eqnarray*}
S_{\cc_{CK}}(F) & = & - F - \sum_{\boldsymbol{e} \mmodels E(F)} S_{\cc_{CK}}(Part_{\boldsymbol{e}}(F)) Cont_{\boldsymbol{e}}(F) \\
& = & - F - \sum_{\boldsymbol{e} \mmodels E(F)} Part_{\boldsymbol{e}}(F) S_{\cc_{CK}}(Cont_{\boldsymbol{e}}(F)) .
\end{eqnarray*}
\end{prop}

{\bf Examples.} {
\begin{eqnarray*}
S_{\cc_{CK}}(\tun) & = & \tun ,\\
S_{\cc_{CK}}(\tdeux) & = & - \tdeux - \tun ,\\
S_{\cc_{CK}}(\ttroisun) & = & - \ttroisun + 2 \tdeux \tdeux + 2 \tdeux ,\\
S_{\cc_{CK}}(\ttroisdeux) & = & - \ttroisdeux + 2 \tdeux \tdeux + 2 \tdeux ,\\
S_{\cc_{CK}}(\tquatredeux) & = & - \tquatredeux + 3 \tdeux \ttroisun + 2 \ttroisun + 2 \tdeux \ttroisdeux + \ttroisdeux - 5 \tdeux \tdeux \tdeux - 6 \tdeux \tdeux - \tdeux .
\end{eqnarray*}
}

We now give a decorated version of $\cc_{CK}$. Let $\mathcal{D}$ be a nonempty set. A rooted forest with their edges decorated by $\mathcal{D}$ is a couple $(F,d)$ where $F$ is a forest of $\cc_{CK}$ and $d : E(F) \rightarrow \mathcal{D} $ is a map. We denote by $\cc^{\mathcal{D}}_{CK}$ the $\mathbb{K}$-vector space spanned by rooted forests with edges decorated by $\mathcal{D}$.\\

{\bf Examples.} {\begin{enumerate}
\item Rooted trees decorated by $ \mathcal{D} $ with edges degree smaller than $ 3 $:
$$ \addeux{$a$}, a \in \mathcal{D}, \hspace{0.5cm} \adtroisun{$b$}{$a$}, \adtroisdeux{$a$}{$b$} , (a,b) \in \mathcal{D}^{2} , \hspace{0.5cm} \adquatreun{$a$}{$b$}{$c$} , \adquatredeux{$b$}{$a$}{$c$} , \adquatretrois{$b$}{$a$}{$c$} , \adquatrequatre{$a$}{$b$}{$c$}, \adquatrecinq{$a$}{$b$}{$c$}, (a,b,c) \in \mathcal{D}^{3} .$$
\item Rooted forests decorated by $ \mathcal{D} $ with edges degree smaller than $ 3 $:
$$ \addeux{$a$}, a \in \mathcal{D} , \hspace{0.5cm} \addeux{$a$} \addeux{$b$} , \adtroisun{$b$}{$a$}, \adtroisdeux{$a$}{$b$} , (a,b) \in \mathcal{D}^{2}, $$
$$ \addeux{$a$} \addeux{$b$} \addeux{$c$}, \addeux{$a$} \adtroisun{$c$}{$b$} , \addeux{$a$} \adtroisdeux{$b$}{$c$} , \adquatreun{$a$}{$b$}{$c$} , \adquatredeux{$b$}{$a$}{$c$} , \adquatretrois{$b$}{$a$}{$c$} , \adquatrequatre{$a$}{$b$}{$c$}, \adquatrecinq{$a$}{$b$}{$c$}, (a,b,c) \in \mathcal{D}^{3} .$$
\end{enumerate}}

If $F \in \cc_{CK}^{\mathcal{D}}$ $ \boldsymbol{e} \models E(F)$, then $Part_{\boldsymbol{e}}(F)$ and $Cont_{\boldsymbol{e}}(F)$ are naturally rooted forests with their edges decorated by $\mathcal{D}$: we keep the decoration of each edges. The vector space $ \cc_{CK}^{\mathcal{D}} $ is a Hopf algebra. Its product is given by the concatenation and its coproduct is the contraction coproduct. For example: if $ (a,b,c) \in \mathcal{D}^{3} $,
\begin{eqnarray*}
\Delta_{\cc_{CK}^{\mathcal{D}}} (\adquatredeux{$b$}{$a$}{$c$})&=&\adquatredeux{$b$}{$a$}{$c$} \otimes \tun +\tun \otimes \adquatredeux{$b$}{$a$}{$c$}+ \addeux{$c$} \otimes \adtroisun{$b$}{$a$} + \addeux{$a$} \otimes \adtroisun{$b$}{$c$} + \addeux{$b$} \otimes \adtroisdeux{$a$}{$c$} + \addeux{$c$} \addeux{$b$} \otimes \addeux{$a$} \\
& & + \adtroisun{$b$}{$a$} \otimes \addeux{$c$} + \adtroisdeux{$a$}{$c$} \otimes \addeux{$b$} .
\end{eqnarray*}

{\bf Notation.} {The set of nonempty trees of $\cc_{CK}$ (that is to say with at least one edge) will be denoted by $\mathbb{T}_{\cc_{CK}}$. The set of nonempty trees with their edges decorated by $\mathcal{D}$ of $\cc_{CK}^{\mathcal{D}}$ will be denoted by $\mathbb{T}^{\mathcal{D}}_{\cc_{CK}}$.}

\subsection{Insertion operations}

Let $\mathbf{T}_{CK}^{\mathcal{D}}$ be the $\mathbb{K}$-vector space having for basis $\mathbb{T}^{\mathcal{D}}_{\cc_{CK}}$. In this section, we prove that $\mathbf{T}_{CK}^{\mathcal{D}}$ is equiped with two operations $\curlyvee$ and $\rhd$ such that $(\mathbf{T}_{CK}^{\mathcal{D}}, \curlyvee , \rhd)$ is a commutative prelie algebra. 
 
\begin{defi}
\begin{enumerate}
\item A commutative prelie algebra is a $\mathbb{K}$-vector space $A$ together with two $\mathbb{K}$-linear maps $\curlyvee , \rhd : A \otimes A \rightarrow A$ such that $x \curlyvee y = y \curlyvee x$ for all $x,y \in A$ (that is to say $\curlyvee$ is commutative) and satisfying the following relations : for all $x,y,z \in A$,
\begin{eqnarray} \label{relationComPrelie}
\left\lbrace \begin{array}{l}
(x \curlyvee y) \curlyvee z = x \curlyvee (y \curlyvee z) ,\\
x \rhd (y \rhd z) - (x \rhd y) \rhd z = y \rhd (x \rhd z) - (y \rhd x) \rhd z ,\\
x \rhd (y \curlyvee z) = (x \rhd y) \curlyvee z + (x \rhd z) \curlyvee y .
\end{array} \right. 
\end{eqnarray}
In other words, $(A, \curlyvee , \rhd)$ is a commutative prelie algebra if $(A,\curlyvee)$ is a commutative algebra and $(A,\rhd)$ is a left prelie algebra with a relationship between $\curlyvee$ and $\rhd$.
\item The commutative prelie operad, denoted $\mathcal{C}om \mathcal{P}reLie$, is the operad such that $\mathcal{C}om \mathcal{P}reLie$-algebras are commutative prelie algebras.
\end{enumerate}
\end{defi}

{\bf Remark.} From this definition, it is clear that the operad $\mathcal{C}om \mathcal{P}reLie$ is binary and quadratic (see \cite{LodayV} for a definition).
\\

{\bf Notations.} \begin{enumerate}
\item Let $T \in \mathbb{T}_{\cc_{CK}}$ be a tree with at least one edge. We denote by $V^{\ast}(T) = V(T) \setminus \{ R_{T} \}$ the set of vertices of $T$ different from the root of $T$.
\item Let $T_{1}, T_{2} \in \mathbb{T}_{\cc_{CK}}$ and $v \in V(T_{2})$. Then $T_{1} \circ_{v} T_{2}$ is the tree obtained by identifying the root $R_{T_{1}}$ of $T_{1}$ and the vertex $v$ of $T_{2}$.
\end{enumerate}
\vspace{0.5cm}

We define two $\mathbb{K}$-linear maps $\curlyvee : \mathbf{T}_{CK}^{\mathcal{D}} \otimes \mathbf{T}_{CK}^{\mathcal{D}} \rightarrow \mathbf{T}_{CK}^{\mathcal{D}}$ and $\rhd : \mathbf{T}_{CK}^{\mathcal{D}} \otimes \mathbf{T}_{CK}^{\mathcal{D}} \rightarrow \mathbf{T}_{CK}^{\mathcal{D}}$ as follow: if $T_{1}, T_{2} \in \mathbb{T}^{\mathcal{D}}_{\cc_{CK}}$,
\begin{eqnarray*}
T_{1} \curlyvee T_{2} & = & T_{1} \circ_{R_{T_{2}}} T_{2} ,\\
T_{1} \rhd T_{2} & = & \sum_{s \in V^{\ast}(T_{2})} T_{1} \circ_{s} T_{2} .
\end{eqnarray*}
\vspace{0.5cm}

{\bf Examples.} \begin{enumerate}
\item For the map $\curlyvee : \mathbf{T}_{CK}^{\mathcal{D}} \otimes \mathbf{T}_{CK}^{\mathcal{D}} \rightarrow \mathbf{T}_{CK}^{\mathcal{D}}$ :
\begin{eqnarray*}
\begin{array}{rcl|rcl|rcl}
\addeux{$a$} \curlyvee \addeux{$b$} & = & \adtroisun{$b$}{$a$} & \addeux{$a$} \curlyvee \adtroisdeux{$b$}{$c$} & = & \adquatretrois{$b$}{$a$}{$c$} & \adtroisdeux{$a$}{$b$} \curlyvee \addeux{$c$} & = & \adquatredeux{$c$}{$a$}{$b$} \\
\addeux{$a$} \curlyvee \adtroisun{$c$}{$b$} & = & \adquatreun{$a$}{$b$}{$c$} & \adtroisun{$b$}{$a$} \curlyvee \addeux{$c$} & = & \adquatreun{$a$}{$b$}{$c$} & \adquatrequatre{$a$}{$c$}{$d$} \curlyvee \addeux{$b$} & = & \adcinqsix{$a$}{$b$}{$c$}{$d$}
\end{array}
\end{eqnarray*}
\item For the map $\rhd : \mathbf{T}_{CK}^{\mathcal{D}} \otimes \mathbf{T}_{CK}^{\mathcal{D}} \rightarrow \mathbf{T}_{CK}^{\mathcal{D}}$ :
\begin{eqnarray*}
\begin{array}{rcl|rcl|rcl}
\addeux{$a$} \rhd \addeux{$b$} & = & \adtroisdeux{$b$}{$a$} & \addeux{$a$} \rhd \adtroisun{$c$}{$b$} & = & \adquatredeux{$c$}{$b$}{$a$} + \adquatretrois{$c$}{$b$}{$a$} & \adtroisun{$b$}{$a$} \rhd \addeux{$c$} & = & \adquatrequatre{$c$}{$a$}{$b$} \\
\adtroisdeux{$a$}{$b$} \rhd \addeux{$c$} & = & \adquatrecinq{$c$}{$a$}{$b$} & \addeux{$a$} \rhd \adtroisdeux{$b$}{$c$} & = & \adquatrequatre{$b$}{$a$}{$c$} + \adquatrecinq{$b$}{$c$}{$a$} & \adtroisun{$b$}{$a$} \rhd \adtroisun{$d$}{$c$} & = & \adcinqsix{$c$}{$d$}{$a$}{$b$} + \adcinqsept{$c$}{$d$}{$b$}{$a$}
\end{array}
\end{eqnarray*}
\end{enumerate}

\begin{prop}
$(\mathbf{T}_{CK}^{\mathcal{D}}, \curlyvee , \rhd )$ is a $\mathcal{C}om \mathcal{P}reLie$-algebra.
\end{prop}

\begin{proof}
Let $T_{1}, T_{2}, T_{3} \in \mathbb{T}^{\mathcal{D}}_{\cc_{CK}}$. Then
\begin{eqnarray*}
T_{1} \curlyvee T_{2} = T_{1} \circ_{R_{T_{2}}} T_{2} = T_{2} \circ_{R_{T_{1}}} T_{1} = T_{2} \curlyvee T_{1} .
\end{eqnarray*}
Moreover,
\begin{eqnarray*}
(T_{1} \curlyvee T_{2}) \curlyvee T_{3} = (T_{1} \circ_{R_{T_{2}}} T_{2}) \circ_{R_{T_{3}}} T_{3} = T_{1} \circ_{R_{(T_{2} \circ_{R_{T_{3}}} T_{3})}} (T_{2} \circ_{R_{T_{3}}} T_{3}) = T_{1} \curlyvee (T_{2} \curlyvee T_{3}).
\end{eqnarray*}
Therefore $(\mathbf{T}_{CK}^{\mathcal{D}}, \curlyvee)$ is a commutative algebra.
\begin{eqnarray*}
T_{1} \rhd (T_{2} \rhd T_{3}) & = & \sum_{\substack{v \in V^{\ast}(T_{3})\\ w \in V^{\ast}(T_{2}) \cup V^{\ast}(T_{3})}} T_{1} \circ_{w} (T_{2} \circ_{v} T_{3}) \\
& = & \sum_{v \in V^{\ast}(T_{3}), w \in V^{\ast}(T_{2})} T_{1} \circ_{w} (T_{2} \circ_{v} T_{3}) + \sum_{v,w \in V^{\ast}(T_{3})} T_{1} \circ_{w} (T_{2} \circ_{v} T_{3}) \\
& = & \sum_{v \in V^{\ast}(T_{3}), w \in V^{\ast}(T_{2})} (T_{1} \circ_{w} T_{2}) \circ_{v} T_{3} + \sum_{v,w \in V^{\ast}(T_{3})} T_{1} \circ_{w} (T_{2} \circ_{v} T_{3}) \\
& = & (T_{1} \rhd T_{2}) \rhd T_{3} + \sum_{v,w \in V^{\ast}(T_{3})} T_{1} \circ_{w} (T_{2} \circ_{v} T_{3}) .
\end{eqnarray*}
So
\begin{eqnarray*}
T_{1} \rhd (T_{2} \rhd T_{3}) - (T_{1} \rhd T_{2}) \rhd T_{3} & = & \sum_{v,w \in V^{\ast}(T_{3})} T_{1} \circ_{w} (T_{2} \circ_{v} T_{3}) \\
& = & \sum_{v,w \in V^{\ast}(T_{3})} T_{2} \circ_{v} (T_{1} \circ_{w} T_{3}) \\
& = & T_{2} \rhd (T_{1} \rhd T_{3}) - (T_{2} \rhd T_{1}) \rhd T_{3}.
\end{eqnarray*}
Therefore, $(\mathbf{T}_{CK}^{\mathcal{D}},\rhd)$ is a left prelie algebra.\\

It remains to prove the last relation of (\ref{relationComPrelie}):
\begin{eqnarray*}
T_{1} \rhd (T_{2} \curlyvee T_{3}) & = & \sum_{v \in V^{\ast}(T_{2} \circ_{R_{T_{3}}} T_{3})} T_{1} \circ_{v} (T_{2} \circ_{R_{T_{3}}} T_{3}) \\
& = & \sum_{v \in V^{\ast}(T_{2})} T_{1} \circ_{v} (T_{2} \circ_{R_{T_{3}}} T_{3}) + \sum_{v \in V^{\ast}(T_{3})} T_{1} \circ_{v} (T_{2} \circ_{R_{T_{3}}} T_{3}) \\
& = & \sum_{v \in V^{\ast}(T_{2})} T_{1} \circ_{v} (T_{2} \circ_{R_{T_{3}}} T_{3}) + \sum_{v \in V^{\ast}(T_{3})} T_{1} \circ_{v} (T_{3} \circ_{R_{T_{2}}} T_{2}) \\
& = & \left( \sum_{v \in V^{\ast}(T_{2})} T_{1} \circ_{v} T_{2} \right) \circ_{R_{T_{3}}} T_{3} + \left( \sum_{v \in V^{\ast}(T_{3})} T_{1} \circ_{v} T_{3} \right) \circ_{R_{T_{2}}} T_{2} \\
& = & (T_{1} \rhd T_{2}) \curlyvee T_{3} + (T_{1} \rhd T_{3}) \curlyvee T_{2} .
\end{eqnarray*}
\end{proof}

\begin{theo}
$(\mathbf{T}_{CK}^{\mathcal{D}}, \curlyvee , \rhd )$ is generated as $\mathcal{C}om \mathcal{P}reLie$-algebra by $\addeux{$d$}$, $d \in \mathcal{D}$.
\end{theo}

{\bf Notation.} To prove the previous proposition, we introduce a notation. Let $T_{1}, \hdots , T_{k} $ are trees (possibly empty) of $\cc_{CK}^{\mathcal{D}}$ and $d_{1}, \hdots , d_{k} \in \mathcal{D}$. Then $B_{d_{1} \otimes \hdots \otimes d_{k}}(T_{1} \otimes \hdots \otimes T_{k})$ is the tree obtained by grafting each $T_{i}$ on a common root with an edge decorated by $d_{i}$. For examples, if $a,b,c,d \in \mathcal{D}$,
\begin{eqnarray*}
\begin{array}{rcl|rcl|rcl}
B_{a}(\tun) & = & \addeux{$a$} & B_{a \otimes b}(\tun \otimes \tun) & = & \adtroisun{$b$}{$a$} & B_{a}(\addeux{$b$}) & = & \adtroisdeux{$a$}{$b$} \\
B_{a}(\adtroisdeux{$b$}{$c$}) & = & \adquatrecinq{$a$}{$b$}{$c$} & B_{a \otimes b}(\addeux{$c$} \otimes \tun) & = & \adquatredeux{$b$}{$a$}{$c$} & B_{a \otimes b}(\tun \otimes \addeux{$c$}) & = & \adquatretrois{$b$}{$a$}{$c$} \\
B_{a \otimes b \otimes c}(\tun \otimes \tun \otimes \tun) & = & \adquatreun{$a$}{$b$}{$c$} & B_{a}(\adtroisun{$c$}{$b$}) & = & \adquatrequatre{$a$}{$b$}{$c$} & B_{a \otimes b}(\adtroisun{$d$}{$c$} \otimes \tun) & = & \adcinqsix{$a$}{$b$}{$c$}{$d$} 
\end{array}
\end{eqnarray*}

\begin{proof}
Let us prove that $(\mathbf{T}_{CK}^{\mathcal{D}}, \curlyvee , \rhd )$ is generated as $\mathcal{C}om \mathcal{P}reLie$-algebra by $\addeux{$d$}$, $d \in \mathcal{D}$ by induction on the edges degree $n$. If $n=1$, this is obvious. Let $T \in \mathbf{T}_{CK}^{\mathcal{D}}$ be a tree of edges degree $n \geq 2$. Let $k$ be an integer such that $T = B_{d_{1} \otimes \hdots \otimes d_{k}}(T_{1} \otimes \hdots \otimes T_{k})$ with $d_{1}, \hdots , d_{k} \in \mathcal{D}$ and $T_{1}, \hdots , T_{k} $ trees (possibly empty) of $\cc_{CK}^{\mathcal{D}}$. Then:
\begin{enumerate}
\item If $k=1$, $T = B_{d_{1}}(T_{1})$ with $\left| T_{1} \right|_{e} = n-1 \geq 1$. By induction hypothesis, $T_{1}$ can be constructed from trees $\addeux{$d$}$, $d \in \mathcal{D}$, with the operations $\curlyvee$ and $\rhd$. So $T = T_{1} \rhd \addeux{$d_{1}$}$ can be also constructed from trees $\addeux{$d$}$, $d \in \mathcal{D}$, with the operations $\curlyvee$ and $\rhd$.
\item Suppose that $k \geq 2$. Then, for all $i$, $ 1 \leq \left| B_{d_{i}}(T_{i}) \right|_{e} \leq n-1 $. By induction hypothesis, the trees $B_{d_{i}}(T_{i})$ can be constructed from trees $\addeux{$d$}$, $d \in \mathcal{D}$, with the operations $\curlyvee$ and $\rhd$. So $T = B_{d_{1}}(T_{1}) \curlyvee \hdots \curlyvee B_{d_{k}}(T_{k})$ can be also constructed from trees $\addeux{$d$}$, $d \in \mathcal{D}$, with the operations $\curlyvee$ and $\rhd$.
\end{enumerate}
We conclude with the induction principle.
\end{proof}
\vspace{0.5cm}

{\bf Remarks.} \begin{enumerate}
\item $(\mathbf{T}_{CK}^{\mathcal{D}}, \curlyvee , \rhd )$ is not the free $\mathcal{C}om \mathcal{P}reLie$-algebra generated by $\addeux{$d$}$, $d \in \mathcal{D}$. For example,
$$ \addeux{$a$} \rhd (\addeux{$b$} \rhd \addeux{$c$}) = \adquatrequatre{$c$}{$a$}{$b$} + \adquatrecinq{$c$}{$b$}{$a$} = (\addeux{$a$} \curlyvee \addeux{$b$}) \rhd \addeux{$c$} + (\addeux{$a$} \rhd \addeux{$b$}) \rhd \addeux{$c$} . $$
\item A description of the free $\mathcal{C}om \mathcal{P}reLie$-algebra is given in \cite{Foissy13}.
\end{enumerate}

\subsection{Noncommutative case}

We give a noncommutative version of $ \cc_{CK} $. To do this, we work on the algebra $ \hh_{po} $.

\begin{defi} \label{defipart}
Let $ (F,\sigma^{F}) $ be a nonempty preordered forest. In particular, $ F $ is a nonempty rooted forest. Let $ \boldsymbol{e} $ be a contraction of $ F $, $ Part_{\boldsymbol{e}}(F) $ the partition of $ F $ by $\boldsymbol{e}$ and $ Cont_{\boldsymbol{e}}(F) $ the contracted of $ F $ by $\boldsymbol{e}$ (see definition \ref{defipartcom}). Then:
\begin{enumerate}
\item $ Part_{\boldsymbol{e}}(F) $ is a preordered forest $ (Part_{\boldsymbol{e}}(F),\sigma^{P}) $ where $ \sigma^{P} : v \in V(Part_{\boldsymbol{e}}(F)) \mapsto \sigma^{F}(v) $. In other words, we keep the initial preorder of the vertices of $ F $ in $ Part_{\boldsymbol{e}}(F) $.
\item $ Cont_{\boldsymbol{e}}(F) $ is also a preordered forest $ (Cont_{\boldsymbol{e}}(F), \sigma^{C}) $ where $ \sigma^{C} : V(Cont_{\boldsymbol{e}}(F)) \rightarrow \{1, \hdots, p \} $ is the surjection ($ p \leq \left| Cont_{\boldsymbol{e}}(F) \right|_{v} $) such that if $ A,B $ are two connected components of $ Part_{\boldsymbol{e}}(F) $, if $ a $ (resp. $ b $) is the vertex obtained by contracting $ A $ (resp. $ B $) in $ F $, then
\begin{eqnarray} \label{defisigcont}
\left\lbrace 
\begin{array}{rcl}
\sigma^{F}(R_A) < \sigma^{F}(R_B) & \Longrightarrow & \sigma^{C}(a) < \sigma^{C}(b),\\
\sigma^{F}(R_A) = \sigma^{F}(R_B) & \Longrightarrow & \sigma^{C}(a) = \sigma^{C}(b),\\
\sigma^{F}(R_A) > \sigma^{F}(R_B) & \Longrightarrow & \sigma^{C}(a) > \sigma^{C}(b).
\end{array}
\right. 
\end{eqnarray}
In other words, we contract each connected component of $ Part_{\boldsymbol{e}}(F) $ to its root and we keep the initial preorder of the roots.
\end{enumerate}
\end{defi}

{\bf Example.} {Let $ T = \tdquatredeux{$2$}{$3$}{$3$}{$1$} $ be a preordered tree. Then
$$\begin{array}{|c|c|c|c|c|c|c|c|c|c|}
\hline \mbox{contraction } \boldsymbol{e} &\tdquatredeux{$2$}{$3$}{$3$}{$1$}&\tdquatredeuxa{$2$}{$3$}{$3$}{$1$}&\tdquatredeuxb{$2$}{$3$}{$3$}{$1$}&\tdquatredeuxc{$2$}{$3$}{$3$}{$1$}
&\tdquatredeuxd{$2$}{$3$}{$3$}{$1$}&\tdquatredeuxe{$2$}{$3$}{$3$}{$1$}&\tdquatredeuxf{$2$}{$3$}{$3$}{$1$}&\tdquatredeuxg{$2$}{$3$}{$3$}{$1$}\\
\hline Part_{\boldsymbol{e}}(T)&\tdquatredeux{$2$}{$3$}{$3$}{$1$}&\tddeux{$2$}{$3$} \tddeux{$3$}{$1$} & \tdun{$1$} \tdtroisun{$2$}{$3$}{$3$}& \tdtroisdeux{$2$}{$3$}{$1$} \tdun{$3$}& \tdun{$1$} \tddeux{$2$}{$3$} \tdun{$3$} & \tdun{$2$} \tdun{$3$} \tddeux{$3$}{$1$} &\tdun{$1$} \tddeux{$2$}{$3$} \tdun{$3$} & \tdun{$1$} \tdun{$2$} \tdun{$3$} \tdun{$3$} \\
\hline Cont_{\boldsymbol{e}}(T)& \tdun{$1$} & \tddeux{$1$}{$2$} & \tddeux{$2$}{$1$} & \tddeux{$1$}{$2$} & \tdtroisdeux{$2$}{$3$}{$1$} & \tdtroisun{$1$}{$2$}{$2$} & \tdtroisun{$2$}{$3$}{$1$} & \tdquatredeux{$2$}{$3$}{$3$}{$1$} \\
\hline \end{array}$$
where, in the first line, the edges not belonging to $ \boldsymbol{e} $ are striked out.}
\\

Let $ I_{po} $ be the ideal of $\hh_{po}$ generated by the elements $ F \tdun{$i$} - \tilde{F} $ with $ F \tdun{$i$} \in \hh_{po} $ and $ \tilde{F} $ the forest contructed from $ F \tdun{$i$} $ by deleting the vertex $ \tdun{$i$} $ and keeping the same preorder on $ V(F) $. For example,
\begin{itemize}
\item if $ F \tdun{$i$} = \tdtroisun{$3$}{$3$}{$1$} \tdun{$2$} $ then $ \tilde{F} = \tdtroisun{$2$}{$2$}{$1$} $,
\item if $ F \tdun{$i$} = \tdquatredeux{$2$}{$3$}{$1$}{$2$} \tdun{$2$} $ then $\tilde{F} = \tdquatredeux{$2$}{$3$}{$1$}{$2$} $.
\end{itemize}

Let $ \cc_{po} $ be the quotient algebra $ \hh_{po}/I_{po} $. So one identifies the unit $ 1 $ (for the concatenation) with the tree $ \tdun{$1$} $. Note that $ \cc_{po} $ is a graded algebra by the number of edges. We note in the same way a forest and his class in $\cc_{po}$. We define on $ \cc_{po} $ a contraction coproduct on each preordered forest $ F \in \cc_{po} $:
\begin{eqnarray*}
\Delta_{\cc_{po}} (F) & = & \sum_{\boldsymbol{e} \models E(F)} Part_{\boldsymbol{e}}(F) \otimes Cont_{\boldsymbol{e}}(F) ,\\
& = & F \otimes \tdun{$1$} + \tdun{$1$} \otimes F + \sum_{\boldsymbol{e} \mmodels E(F)} Part_{\boldsymbol{e}}(F) \otimes Cont_{\boldsymbol{e}}(F) .
\end{eqnarray*}

{\bf Examples.} {\begin{eqnarray*}
\Delta_{\cc_{po}} ( \tdun{$1$} ) & = & \tdun{$1$} \otimes \tdun{$1$} \\
\Delta_{\cc_{po}} ( \tddeux{$2$}{$1$} ) & = & \tddeux{$2$}{$1$} \otimes \tdun{$1$} + \tdun{$1$} \otimes \tddeux{$2$}{$1$} \\
\Delta_{\cc_{po}} ( \tdtroisun{$2$}{$2$}{$1$} ) & = & \tdtroisun{$2$}{$2$}{$1$} \otimes \tdun{$1$} + \tdun{$1$} \otimes \tdtroisun{$2$}{$2$}{$1$} + \tddeux{$2$}{$1$} \otimes \tddeux{$1$}{$1$} + \tddeux{$1$}{$1$} \otimes \tddeux{$2$}{$1$} \\
\Delta_{\cc_{po}} ( \tddeux{$2$}{$4$} \tddeux{$3$}{$1$} ) & = & \tddeux{$2$}{$4$} \tddeux{$3$}{$1$} \otimes \tdun{$1$} + \tdun{$1$} \otimes \tddeux{$2$}{$4$} \tddeux{$3$}{$1$} + \tddeux{$1$}{$2$} \otimes \tddeux{$2$}{$1$} + \tddeux{$2$}{$1$} \otimes \tddeux{$1$}{$2$} \\
\Delta_{\cc_{po}} ( \tdquatredeux{$2$}{$3$}{$3$}{$1$}) & = & \tdquatredeux{$2$}{$3$}{$3$}{$1$} \otimes \tdun{$1$} + \tdun{$1$} \otimes \tdquatredeux{$2$}{$3$}{$3$}{$1$} + \tddeux{$2$}{$3$} \tddeux{$3$}{$1$} \otimes \tddeux{$1$}{$2$} + \tdtroisun{$1$}{$2$}{$2$} \otimes \tddeux{$2$}{$1$} + \tdtroisdeux{$2$}{$3$}{$1$} \otimes \tddeux{$1$}{$2$} + \tddeux{$1$}{$2$} \otimes \tdtroisdeux{$2$}{$3$}{$1$} \\
& & + \tddeux{$2$}{$1$} \otimes \tdtroisun{$1$}{$2$}{$2$} + \tddeux{$1$}{$2$} \otimes \tdtroisun{$2$}{$3$}{$1$} \\
\Delta_{\cc_{po}} ( \tdtroisun{$3$}{$5$}{$2$} \tddeux{$4$}{$1$} ) & = & \tdtroisun{$3$}{$5$}{$2$} \tddeux{$4$}{$1$} \otimes \tdun{$1$} + \tdun{$1$} \otimes \tdtroisun{$3$}{$5$}{$2$} \tddeux{$4$}{$1$} + \tddeux{$2$}{$1$} \otimes \tddeux{$2$}{$4$} \tddeux{$3$}{$1$} + \tddeux{$1$}{$2$} \otimes \tddeux{$3$}{$2$} \tddeux{$4$}{$1$} + \tddeux{$2$}{$1$} \otimes \tdtroisun{$2$}{$3$}{$1$} \\
& & + \tddeux{$3$}{$2$} \tddeux{$4$}{$1$} \otimes \tddeux{$1$}{$2$} + \tddeux{$2$}{$4$} \tddeux{$3$}{$1$} \otimes \tddeux{$2$}{$1$} + \tdtroisun{$2$}{$3$}{$1$} \otimes \tddeux{$2$}{$1$}
\end{eqnarray*}}

{\bf Remark.} {$ \Delta_{\cc_{po}} $ is non cocommutative (see for example the coproduct of $ \tdquatredeux{$2$}{$3$}{$3$}{$1$}$). In particular, if $ T $ is a preordered tree and $ \boldsymbol{e} \models E(T) $, $ Cont_{\boldsymbol{e}}(T) $ is a preordered tree and $ Part_{\boldsymbol{e}}(T) $ can be disconnected. The second component of the coproduct is linear: a tree instead of a polynomial in trees. This is a right combinatorial Hopf algebra (see \cite{Loday10}).}

\begin{prop}
\begin{enumerate}
\item $ \Delta_{\cc_{po}} $ is a graded algebra morphism.
\item $ \Delta_{\cc_{po}} $ is coassociative.
\end{enumerate}
\end{prop}

\begin{proof}
\begin{enumerate}
\item Let $ F,G $ be two preordered forests. Then
\begin{eqnarray*}
\Delta_{\cc_{po}} (FG) & = & \sum_{\boldsymbol{e} \models E(FG)} Part_{\boldsymbol{e}}(FG) \otimes Cont_{\boldsymbol{e}}(FG) \\
& = & \sum_{\boldsymbol{e} \models E(F), \boldsymbol{f} \models E(G)} \left( Part_{\boldsymbol{e}}(F) Part_{\boldsymbol{f}}(G) \right) \otimes \left( Cont_{\boldsymbol{e}}(F) Cont_{\boldsymbol{f}}(G) \right) \\
& = & \left( \sum_{\boldsymbol{e} \models E(F)} Part_{\boldsymbol{e}}(F) \otimes Cont_{\boldsymbol{e}}(F) \right) \left( \sum_{\boldsymbol{f} \models E(G)} Part_{\boldsymbol{f}}(G) \otimes Cont_{\boldsymbol{f}}(G) \right) \\
& = & \Delta_{\cc_{po}} (F) \Delta_{\cc_{po}}(G) ,
\end{eqnarray*}
and $ \Delta_{\cc_{po}} $ is an algebra morphism. It is a graded algebra morphism with (\ref{relatdeg}).

\item Let $ F $ be a nonempty preordered forest. Then
\begin{eqnarray*}
& & (\Delta_{\cc_{po}} \otimes Id) \circ \Delta_{\cc_{po}} (F) \\
& = & \sum_{\boldsymbol{e} \models E(F)} \Delta_{\cc_{po}} (Part_{\boldsymbol{e}}(F)) \otimes Cont_{\boldsymbol{e}}(F) \\
& = & \sum_{\boldsymbol{e} \models E(F)} \sum_{\boldsymbol{f} \models E(Part_{\boldsymbol{e}}(F))} Part_{\boldsymbol{f}}(Part_{\boldsymbol{e}}(F)) \otimes Cont_{\boldsymbol{f}}(Part_{\boldsymbol{e}}(F)) \otimes Cont_{\boldsymbol{e}}(F) \\
& = & \sum_{\boldsymbol{f} \subseteq \boldsymbol{e} \subseteq E(F)} Part_{\boldsymbol{f}}(F) \otimes Cont_{\boldsymbol{f}}(Part_{\boldsymbol{e}}(F)) \otimes Cont_{\boldsymbol{e}}(F) ,
\end{eqnarray*}
and
\begin{eqnarray*}
& & (Id \otimes \Delta_{\cc_{po}}) \circ \Delta_{\cc_{po}} (F) \\
& = & \sum_{\boldsymbol{f} \models E(F)} Part_{\boldsymbol{f}}(F) \otimes \Delta_{\cc_{po}} (Cont_{\boldsymbol{f}}(F)) \\
& = & \sum_{\boldsymbol{f} \models E(F)} \sum_{\boldsymbol{e} \models E(Cont_{\boldsymbol{f}}(F))} Part_{\boldsymbol{f}}(F) \otimes Part_{\boldsymbol{e}}(Cont_{\boldsymbol{f}}(F)) \otimes Cont_{\boldsymbol{e}}(Cont_{\boldsymbol{f}}(F)) \\
& = & \sum_{\boldsymbol{f} \models E(F), \boldsymbol{e} \subseteq \overline{\boldsymbol{f}}} Part_{\boldsymbol{f}}(F) \otimes Part_{\boldsymbol{e}}(Cont_{\boldsymbol{f}}(F)) \otimes Cont_{\boldsymbol{e} \cup \boldsymbol{f}}(F) ,
\end{eqnarray*}
where to the last equality we use that $ E(Cont_{\boldsymbol{f}}(F)) = \overline{\boldsymbol{f}} $ the complement of $ \boldsymbol{f} $ in $ E(F) $ and $ Cont_{\boldsymbol{e}}(Cont_{\boldsymbol{f}}(F)) = Cont_{\boldsymbol{e} \cup \boldsymbol{f}}(F) $.\\

Remark that $ \{ (\boldsymbol{e} , \boldsymbol{f}) \:\mid \: \boldsymbol{f} \subseteq \boldsymbol{e} \subseteq E(F) \} $ and $ \{ (\boldsymbol{e} , \boldsymbol{f}) \:\mid \: \boldsymbol{f} \models E(F), \boldsymbol{e} \subseteq \overline{\boldsymbol{f}} \} $ are in bijection:
\begin{eqnarray*}
\left\lbrace 
\begin{array}{rcl}
\{ (\boldsymbol{e} , \boldsymbol{f}) \:\mid \: \boldsymbol{f} \subseteq \boldsymbol{e} \subseteq E(F) \} & \rightarrow & \{ (\boldsymbol{e} , \boldsymbol{f}) \:\mid \: \boldsymbol{f} \models E(F), \boldsymbol{e} \subseteq \overline{\boldsymbol{f}} \} \\
(\boldsymbol{e} , \boldsymbol{f}) & \rightarrow & (\boldsymbol{e} \setminus \boldsymbol{f} , \boldsymbol{f}) \\
(\boldsymbol{e} \cup \boldsymbol{f}, \boldsymbol{f}) & \leftarrow & (\boldsymbol{e} , \boldsymbol{f}) .
\end{array}
\right. 
\end{eqnarray*}
Moreover,
\begin{itemize}
\item in $ Cont_{\boldsymbol{f}}(Part_{\boldsymbol{e}}(F)) $ with $ \boldsymbol{f} \subseteq \boldsymbol{e} \subseteq E(F) $: the edges belong to $ \boldsymbol{e} \cap \overline{\boldsymbol{f}} = \boldsymbol{e} \setminus \boldsymbol{f} $; the vertices are the connected components of $ Part_{\boldsymbol{e} \cap \boldsymbol{f}}(F) = Part_{\boldsymbol{f}}(F) $. The preorder on the vertices is given by the preorder on the roots of the connected components of $Part_{\boldsymbol{f}}(F)$.
\item in $ Part_{\boldsymbol{e}}(Cont_{\boldsymbol{f}}(F)) $ with $ \boldsymbol{f} \models E(F), \boldsymbol{e} \subseteq \overline{\boldsymbol{f}} $: the edges belong to $ \overline{\boldsymbol{f}} \cap \boldsymbol{e} = \boldsymbol{e} \setminus \boldsymbol{f} = \boldsymbol{e} $; the vertices are the connected components of $ Part_{\boldsymbol{f}}(F) $. As in the precedent case, the preorder on the vertices is given by the preorder on the roots of the connected components of $Part_{\boldsymbol{f}}(F)$.
\end{itemize}
So $ Cont_{\boldsymbol{f}}(Part_{\boldsymbol{e}}(F)) $ and $ Part_{\boldsymbol{e}}(Cont_{\boldsymbol{f}}(F)) $ are the same forests with the same preorder on the vertices.\\

Therefore $ (\Delta_{\cc_{po}} \otimes Id) \circ \Delta_{\cc_{po}} (F) = (Id \otimes \Delta_{\cc_{po}}) \circ \Delta_{\cc_{po}} (F) $.
\end{enumerate}
\end{proof}
\\

We now define
\begin{eqnarray*}
\varepsilon : \left\lbrace \begin{array}{rcl}
\cc_{po} & \rightarrow & \mathbb{K} \\
F \mbox{ forest } & \mapsto & \delta_{F,\tdun{$1$}}.
\end{array} \right. 
\end{eqnarray*}
$ \varepsilon $ is an algebra morphism.

\begin{prop}
$ \varepsilon $ is a counit for the coproduct $ \Delta_{\cc_{po}} $.
\end{prop}

\begin{proof}
Let $ F $ be a forest $  \in \cc_{po} $. We use the Sweedler notation:
\begin{eqnarray*}
\Delta_{\cc_{po}} (F) = F \otimes \tdun{$1$} + \tdun{$1$} \otimes F + \sum_{F} F^{(1)} \otimes F^{(2)} .
\end{eqnarray*}
Then
\begin{eqnarray*}
(\varepsilon \otimes Id) \circ \Delta_{\cc_{po}} (F) & = & \varepsilon (F) \tdun{$1$} + \varepsilon (\tdun{$1$}) F + \sum_{F} \varepsilon (F^{(1)}) \otimes F^{(2)} = F ,\\
(Id \otimes \varepsilon) \circ \Delta_{\cc_{po}} (F) & = & F \varepsilon(\tdun{$1$}) + \tdun{$1$} \varepsilon (F) + \sum_{F} F^{(1)} \varepsilon (F^{(2)}) = F .
\end{eqnarray*}
Therefore $ \varepsilon $ is a counit for the coproduct $ \Delta_{\cc_{po}} $.
\end{proof}
\\

As $ (\cc_{po}, \Delta_{\cc_{po}} , \varepsilon ) $ is gradued (by the number of edges) and connected, we have the following theorem:

\begin{theo}
$ (\cc_{po}, \Delta_{\cc_{po}} , \varepsilon ) $ is a Hopf algebra.
\end{theo}

We denote the antipode of the Hopf algebra $\cc_{po}$ by $S_{\cc_{po}}$. We have the same combinatorial description of $S_{\cc_{po}}$ as the commutative case (see proposition \ref{antipodecom}). We give some values of $S_{\cc_{po}}$:
\begin{itemize}
\item In edges degree 0, $S_{\cc_{po}}(\tdun{$1$})= \tdun{$1$}$.
\item In edges degree 1, $S_{\cc_{po}} (\tddeux{$1$}{$1$}) = - \tddeux{$1$}{$1$} - \tdun{$1$}$, $S_{\cc_{po}} (\tddeux{$1$}{$2$}) = - \tddeux{$1$}{$2$} - \tdun{$1$}$ and $S_{\cc_{po}} (\tddeux{$2$}{$1$}) = - \tddeux{$2$}{$1$} - \tdun{$1$}$.
\item In edges degree 2,
\begin{eqnarray*}
S_{\cc_{po}} (\tdtroisun{$1$}{$2$}{$2$}) & = & - \tdtroisun{$1$}{$2$}{$2$} + 2 \tddeux{$1$}{$2$} \tddeux{$3$}{$4$} + 2 \tddeux{$1$}{$2$},\\
S_{\cc_{po}} (\tdtroisun{$2$}{$3$}{$1$}) & = & - \tdtroisun{$2$}{$3$}{$1$} + \tddeux{$2$}{$1$} \tddeux{$3$}{$4$} + \tddeux{$1$}{$2$} \tddeux{$4$}{$3$} + \tddeux{$1$}{$2$} + \tddeux{$2$}{$1$},\\
S_{\cc_{po}} (\tdtroisdeux{$1$}{$2$}{$1$}) & = & - \tdtroisdeux{$1$}{$2$}{$1$} +\tddeux{$1$}{$2$} \tddeux{$3$}{$3$} + \tddeux{$2$}{$1$} \tddeux{$3$}{$4$} + \tddeux{$1$}{$1$} + \tddeux{$1$}{$2$} ,\\
S_{\cc_{po}} (\tddeux{$2$}{$3$} \tddeux{$3$}{$1$}) & = & - \tddeux{$2$}{$3$} \tddeux{$3$}{$1$} + \tddeux{$1$}{$2$} \tddeux{$4$}{$3$} + \tddeux{$2$}{$1$} \tddeux{$3$}{$4$} + \tddeux{$2$}{$1$} + \tddeux{$1$}{$2$}.
\end{eqnarray*}
\item In edges degree 3,
\begin{eqnarray*}
S_{\cc_{po}} (\tdquatredeux{$2$}{$3$}{$3$}{$1$}) & = & - \tdquatredeux{$2$}{$3$}{$3$}{$1$}
+ \tddeux{$2$}{$3$} \tddeux{$3$}{$1$} \tddeux{$4$}{$5$} - \tddeux{$1$}{$2$} \tddeux{$4$}{$3$} \tddeux{$5$}{$6$} - \tddeux{$2$}{$1$} \tddeux{$3$}{$4$} \tddeux{$5$}{$6$} - \tddeux{$2$}{$1$} \tddeux{$3$}{$4$} - \tddeux{$1$}{$2$} \tddeux{$3$}{$4$}
+ \tdtroisun{$1$}{$2$}{$2$} \tddeux{$4$}{$3$} \\
& & - 2 \tddeux{$1$}{$2$} \tddeux{$3$}{$4$} \tddeux{$6$}{$5$} - 2 \tddeux{$1$}{$2$} \tddeux{$4$}{$3$}
+ \tdtroisdeux{$2$}{$3$}{$1$} \tddeux{$3$}{$4$} - \tddeux{$2$}{$1$} \tddeux{$3$}{$4$} \tddeux{$5$}{$6$}
- \tddeux{$1$}{$2$} \tddeux{$3$}{$4$}
- \tddeux{$1$}{$2$} \tddeux{$4$}{$3$} \tddeux{$5$}{$6$}
- \tddeux{$2$}{$1$} \tddeux{$3$}{$4$} \\
& & + \tddeux{$1$}{$2$} \tdtroisdeux{$4$}{$5$}{$3$} + \tdtroisdeux{$2$}{$3$}{$1$}
+ \tddeux{$2$}{$1$} \tdtroisun{$3$}{$4$}{$4$} + \tdtroisun{$1$}{$2$}{$2$}
+ \tddeux{$1$}{$2$} \tdtroisun{$4$}{$5$}{$3$} + \tdtroisun{$2$}{$3$}{$1$} .
\end{eqnarray*}
\end{itemize}

Let $ \cc'_{hpo} $ be the $ \mathbb{K} $-algebra spanned by nonempty heap-preordered forests, $ \cc'_{o} $ be the $ \mathbb{K} $-algebra spanned by nonempty ordered forests, $ \cc'_{ho} $ be the $ \mathbb{K} $-algebra spanned by nonempty heap-ordered forests and $ \cc'_{NCK} $ be the $ \mathbb{K} $-algebra spanned by nonempty planar forests. We consider the quotients $ \cc_{hpo} = \cc'_{hpo} / (I_{po} \cap \cc'_{hpo}) $, $ \cc_{o} = \cc'_{o} / (I_{po} \cap \cc'_{o}) $, $ \cc_{ho} = \cc'_{ho} / (I_{po} \cap \cc'_{ho}) $ and $ \cc_{NCK} = \cc'_{NCK} / (I_{po} \cap \cc'_{NCK}) $. We have in this case a similar diagram to (\ref{diaginclusion}):
\begin{eqnarray*}
\xymatrix{
\cc_{NCK} \ar@{^{(}->}[r] & \cc_{ho} \ar@{^{(}->}[d] \ar@{^{(}->}[r] & \cc_{o} \ar@{^{(}->}[d]  \\
& \cc_{hpo} \ar@{^{(}->}[r] &  \cc_{po}
}
\end{eqnarray*}
where the arrows $ \hookrightarrow $ are injective morphisms of algebras. But they are not always morphisms of Hopf algebras (for the contraction coproduct):

\begin{theo}
\begin{enumerate}
\item $ \cc_{hpo} $ is a Hopf subalgebra of the Hopf algebra $ \cc_{po} $.
\item $ \cc_{o} $ is a Hopf subalgebra of the Hopf algebra $ \cc_{po} $.
\item $ \cc_{ho} $ is a Hopf subalgebra of the Hopf algebra $ \cc_{o} $ and of the Hopf algebra $ \cc_{hpo} $.
\item $ \cc_{NCK} $ is a left comodule of the Hopf algebra $ \cc_{ho} $.
\end{enumerate}
\end{theo}
\vspace{0.5cm}

{\bf Notations.} {We denote by $ \Delta_{\cc_{hpo}} , \Delta_{\cc_{o}}, \Delta_{\cc_{ho}} $ the restrictions of $ \Delta_{\cc_{po}} $ to $ \cc_{hpo} , \cc_{o} , \cc_{ho} $.}
\\

{\bf Remark.} {$ \cc_{NCK} $ is not a Hopf subalgebra of the Hopf algebra $ \cc_{ho} $. For example, $ \tdquatredeux{$1$}{$4$}{$2$}{$3$} \in \cc_{NCK} $ and
\begin{eqnarray*}
\Delta_{\cc_{ho}} (\tdquatredeux{$1$}{$4$}{$2$}{$3$}) & = & \tdquatredeux{$1$}{$4$}{$2$}{$3$} \otimes \tdun{$ 1 $} + \tdun{$ 1 $} \otimes \tdquatredeux{$1$}{$4$}{$2$}{$3$} + \tdtroisdeux{$1$}{$2$}{$3$} \otimes \tddeux{$ 1 $}{$ 2 $} + \tdtroisun{$1$}{$3$}{$2$} \otimes \tddeux{$1$}{$2$} + 2 \tddeux{$1$}{$2$} \otimes \tdtroisun{$1$}{$3$}{$2$} \\
& & + \tddeux{$1$}{$2$} \otimes \tdtroisdeux{$1$}{$2$}{$3$} + \tddeux{$1$}{$4$} \tddeux{$2$}{$3$} \otimes \tddeux{$1$}{$2$} .
\end{eqnarray*}
Then $ \tddeux{$1$}{$4$} \tddeux{$2$}{$3$} \otimes \tddeux{$1$}{$2$} \notin \cc_{NCK} \otimes \cc_{NCK} $.
}
\\

\begin{proof}
\begin{enumerate}
\item $ \cc_{hpo} $ is a subalgebra of $ \cc_{po} $. Let us prove that if $ (F,\sigma^{F}) \in \cc_{hpo} $ and $ \boldsymbol{e} \models E(F) $ then $ (Cont_{\boldsymbol{e}}(F),\sigma^{C}) $ and $ (Part_{\boldsymbol{e}}(F),\sigma^{P}) \in \cc_{hpo} $.

If $ a,b \in V(Part_{\boldsymbol{e}}(F)) $, $a \neq b$, such that $ a \twoheadrightarrow b $ then $ a,b $ are the vertices of a subtree of $ (F,\sigma^{F}) \in \cc_{hpo} $ and $ \sigma^{F}(a) > \sigma^{F}(b) $. With definition \ref{defipart}, $ \sigma^{P}(a) > \sigma^{P}(b) $. So $ (Part_{\boldsymbol{e}}(F),\sigma^{P}) \in \cc_{hpo} $.

If $ a,b \in V(Cont_{\boldsymbol{e}}(F)) $, $a \neq b$, such that $ a \twoheadrightarrow b $, then $ a $ and $ b $ are the vertices obtained by contracting two connected components $ A $ and $ B $ of $ Part_{\boldsymbol{e}}(F) $. As $ a \twoheadrightarrow b $, $ R_{A} \twoheadrightarrow R_{B} $ and as $ (F,\sigma^{F}) \in \cc_{hpo} $, $ \sigma^{F}(R_{A}) > \sigma^{F}(R_{B}) $. Then, by definition \ref{defipart}, $ \sigma^{C}(a) > \sigma^{C}(b) $. So $ (Cont_{\boldsymbol{e}}(F),\sigma^{C}) \in \cc_{hpo} $.

Therefore if $ (F,\sigma^{F}) \in \cc_{hpo} $, $ \Delta_{\cc_{po}}(F) \in \cc_{hpo} \otimes \cc_{hpo} $ and $ \cc_{hpo} $ is a Hopf subalgebra of $ \cc_{po} $.

\item $ \cc_{o} $ is a subalgebra of $ \cc_{po} $. Let $ (F,\sigma^{F}) \in \cc_{o} $ and $ \boldsymbol{e} \models E(F) $. Let us show that $ (Cont_{\boldsymbol{e}}(F),\sigma^{C}) $ and $ (Part_{\boldsymbol{e}}(F),\sigma^{P}) \in \cc_{o} $, that is to say that $ \sigma^{C} $ and $ \sigma^{P} $ are bijective.

By definition \ref{defipart}, $ \sigma^{P} $ is bijective because we keep the initial order of the vertices of $ F $ in $ Part_{\boldsymbol{e}}(F) $. By definition, $ \sigma^{C} $ is a surjection. Let $ a,b \in V(Cont_{\boldsymbol{e}}(F)) $ such that $ \sigma^{C}(a) = \sigma^{C}(b) $ and $ A $ and $ B $ be the two connected components of $ Part_{\boldsymbol{e}}(F) $ associated with $ a $ and $ b $. With (\ref{defisigcont}), $ \sigma^{F}(R_{A}) = \sigma^{F}(R_{B}) $ and $ R_{A} = R_{B} $ because $ \sigma^{F} $ is bijective. So $ A = B $, $ a = b $ and $ \sigma^{C} $ is injective.

Therefore $ \sigma^{C} $ and $ \sigma^{P} $ are bijective and $ \cc_{o} $ is a Hopf subalgebra of $ \cc_{ho} $.

\item As $ \cc_{hpo} $ is a Hopf subalgebra of the Hopf algebra $ \cc_{po} $ and $ \cc_{o} $ is a Hopf subalgebra of the Hopf algebra $ \cc_{po} $, $\cc_{ho} = \cc_{hpo} \cap \cc_{o}$ is a Hopf subalgebra of $\cc_{hpo}$ and $\cc_{o}$.

\item Let us prove that if $ (F,\sigma^{F}) \in \cc_{NCK} $ and $ \boldsymbol{e} \models E(F) $ then $ (Cont_{\boldsymbol{e}}(F),\sigma^{C}) \in \cc_{NCK} $. As $ \cc_{ho} $ is a Hopf algebra, $ (Cont_{\boldsymbol{e}}(F),\sigma^{C}) \in \cc_{ho} $. So, if $ a,b \in V(Cont_{\boldsymbol{e}}(F)) $, such that $ a \twoheadrightarrow b $ then $ \sigma^{C}(a) \geq \sigma^{C}(b) $.

Moreover, if $ a,b,c \in V(Cont_{\boldsymbol{e}}(F)) $ three distinct vertices such that $ a \twoheadrightarrow c $, $ b \twoheadrightarrow c $ and $ a $ is on the left of $ b $. The vertices $ a,b $ and $ c $ are obtained by contracting of connected components $ A, B $ and $ C $ in $ F $. As $ a \twoheadrightarrow c $, $ b \twoheadrightarrow c $ and $ a $ is on the left of $ b $, $ R_{A} \twoheadrightarrow R_{C} $, $ R_{B} \twoheadrightarrow R_{C} $ and $ R_{A} $ is on the left of $ R_{B} $. As $ (F,\sigma^{F}) \in \cc_{NCK} $, $ \sigma^{F}(R_{A}) < \sigma^{F}(R_{B}) $. So $ \sigma^{C}(a) < \sigma^{C}(b) $.

Therefore if $ (F,\sigma^{F}) \in \cc_{NCK} $ and $ \boldsymbol{e} \models E(F) $ then $ (Cont_{\boldsymbol{e}}(F),\sigma^{C}) \in \cc_{NCK} $. Consequently, $ \Delta_{\cc_{ho}} (\cc_{NCK}) \subseteq \cc_{ho} \otimes \cc_{NCK} $.
\end{enumerate}
\end{proof}

\subsection{Formal series}

The algebras $\mathbf{C}_{po}$, $\mathbf{C}_{hpo}$, $ \mathbf{C}_{o} $, $ \mathbf{C}_{ho} $ and $ \mathbf{C}_{NCK} $ are graded by the number of edges. 
\vspace{0.2cm}

In the ordered case, we give some values in a small degrees :

$$\begin{array}{c|c|c|c|c|c|c|c|c}
n&1&2&3&4&5&6&7&8\\
\hline f_{n}^{\mathbf{C}_{o}} &2&9&76&805&10626&167839&3091768&65127465
\end{array}$$

These is the sequence A105785 in \cite{Sloane}.\\

Let us now study the heap-ordered case. We denote by $f_{n,l}^{\mathbf{C}_{ho}}$ the forests of $\mathbf{C}_{ho}$ of edges degree $n$ and of length $l$, and by $f_{n}^{\mathbf{C}_{ho}}$ the forests of $\mathbf{C}_{ho}$ of edges degree $n$. In small degree, we have the following values:
\begin{eqnarray*}
\left\lbrace \begin{array}{l}
f_{0,0}^{\mathbf{C}_{ho}} = f_{1,1}^{\mathbf{C}_{ho}} = 1 ,\\
f_{0,l}^{\mathbf{C}_{ho}} = 0 \hspace{1cm} \mbox{ for all } l \geq 1 ,\\
f_{1,l}^{\mathbf{C}_{ho}} = 0 \hspace{1cm} \mbox{ for all } l \neq 1 ,\\
f_{n,0}^{\mathbf{C}_{ho}} = 0 \hspace{1cm} \mbox{ for all } n \neq 1.
\end{array}
\right. 
\end{eqnarray*}
Let $n$ and $l$ be two integers $\geq 1$. To obtain a forest $F \in \mathbf{C}_{ho}$ of edges degree $n$ and of length $l$ (so $\left| F \right|_{v} = n+l$), we have two cases :
\begin{enumerate}
\item We consider a forest $G \in \mathbf{C}_{ho}$ of edges degree $n-1$ and of length $l$ and we graft the vertex $n+l$ on the vertex $i$ of $G$. For each forest $G$, we have $n+l-1$ possibilities.
\item We consider a forest $G \in \mathbf{C}_{ho}$ of edges degree $n-1$ and of length $l-1$. Then, for all $i \in \{ 1, \hdots , n+l-1 \}$, the forest $\tilde{G} \tddeux{$i$}{$n+l$} \ \ $ of edges degree $n$ and of vertices degree $l$ is an element of $\mathbf{C}_{ho}$ (where $\tilde{G}$ is the same forest of $G$ with for all $j \geq i$ the vertex $j$ in $G$ is the vertex $j+1$ in $\tilde{G}$). For each forest $G$, we have $n+l-1$ possibilities.
\end{enumerate}
So
$$f_{n,l}^{\mathbf{C}_{ho}} = (n+l-1) f_{n-1,l}^{\mathbf{C}_{ho}} + (n+l-1) f_{n-1,l-1}^{\mathbf{C}_{ho}} .$$
\vspace{0.2cm}

\noindent We give some values of $f_{n,l}^{\mathbf{C}_{ho}}$ in a small degrees and in a small lengths :

$$ \begin{tabular}{c|c|c|c|c|c|c}
$ n \backslash l $ & 0 & 1 & 2 & 3 & 4 & 5 \\
\hline 0 & 1 & 0 & 0 & 0 & 0 & 0 \\
\hline 1 & 0 & 1 & 0 & 0 & 0 & 0 \\
\hline 2 & 0 & 2 & 3 & 0 & 0 & 0 \\
\hline 3 & 0 & 6 & 20 & 15 & 0 & 0 \\
\hline 4 & 0 & 24 & 130 & 210 & 105 & 0 \\
\hline 5 & 0 & 120 & 924 & 2380 & 2520 & 945
\end{tabular} $$

Note that $f_{n,1}^{\mathbf{C}_{ho}} = n!$ for all $n \geq 1$. With the formula $f_{n}^{\mathbf{C}_{ho}} = \displaystyle\sum_{l \geq 0} f_{n,l}^{\mathbf{C}_{ho}}$, we obtain the number of forests of edges degree $n$. This gives:

$$\begin{array}{c|c|c|c|c|c|c|c}
n&0&1&2&3&4&5&6\\
\hline f_{n}^{\cc_{ho}} & 0&1&5&41&469&6889&123605
\end{array}$$

This is the sequence A032188 in \cite{Sloane}.\\

{\bf Remark.} Consider the map $\varphi : \mathbb{F}_{\mathbf{H}_{ho}} \rightarrow \Sigma$ defined by induction as follows. If $F=1$, $\varphi (F) = 1$ and if $F = \tdun{$1$} $, $\varphi (F) = (1)$. Let $F \in \mathbf{H}_{ho}$ be a forest of vertices degree $n$ and $v$ the vertex indexed by $n$. As $F$ is a heap-ordered forest, two cases are possible:
\begin{itemize}
\item The vertex $v$ is an isolated vertex.  We denote by $G$ the heap-ordered forest obtained by deleting the vertex $v$ of $F$. Thus $\varphi (G) = \tau'$ is well defined by induction. Then $\varphi (F)$ is the permutation $\tau$ defined by
\begin{eqnarray*}
\left\lbrace \begin{array}{rcl}
\tau(i) & = & \tau'(i) \hspace{0.5cm} \mbox{ if } i \neq n \\
\tau (n) & = & n.
\end{array} \right. 
\end{eqnarray*}
\item The vertex $v$ is a leaf and we denote by $k$ the indexe of $v'$ with $v \rightarrow v'$. Similarly, we denote by $G$ the heap-ordered forest obtained by deleting the vertex $v$ of $F$. $\varphi (G) = \tau'$ is well defined by induction and $\varphi (F)$ is the permutation $\tau$ defined by
\begin{eqnarray*}
\left\lbrace \begin{array}{rcl}
\tau(i) & = & \tau'(i) \hspace{0.5cm} \mbox{ if } i \neq k \\
\tau (k) & = & n\\
\tau (n) & = & \tau'(k) .
\end{array} \right. 
\end{eqnarray*}
\end{itemize}
Then $\varphi : \mathbb{F}_{\mathbf{H}_{ho}} \rightarrow \Sigma$ is a bijective map. Remark that, if $F \in \mathbb{F}_{\mathbf{H}_{ho}}$, each connected component of $F$ corresponds to one cycle in the writing of $\varphi (F)$ in product of disjoint cycles. Moreover, the restriction of $\varphi$ to the forests of $\mathbf{C}_{ho}$ is a bijective map with values in the set of permutations without fixed point.
\\

In the planar case, we can obtain the formal series. Let $t_{n}^{\cc_{NCK}}$ be the number of trees $\in \cc_{NCK}$ of edges degree $n$ and $f_{n}^{\cc_{NCK}}$ be the number of forests $\in \cc_{NCK}$ of edges degree $n$. We put $T_{\cc_{NCK}}(x) = \displaystyle\sum_{k \geq 0} t_{k}^{\cc_{NCK}} x^{k}$ and $F_{\cc_{NCK}}(x) = \displaystyle\sum_{k \geq 0} f_{k}^{\cc_{NCK}} x^{k}$. Then:

\begin{prop}
The formal series $ T_{\cc_{NCK}} $ and $ F_{\cc_{NCK}} $ are given by:
\begin{eqnarray*}
T_{\cc_{NCK}} (x) = \dfrac{1-2x-\sqrt{1-4x}}{2x} , \hspace{1cm} F_{\cc_{NCK}} (x) = \dfrac{2x}{4x-1+\sqrt{1-4x}}.
\end{eqnarray*}
\end{prop}

\begin{proof}
With formula (\ref{seriesNCK}), we deduce that:
\begin{eqnarray*}
T_{\cc_{NCK}} (x) = \dfrac{1-\sqrt{1-4x}}{2x} - 1 = \dfrac{1-2x-\sqrt{1-4x}}{2x} .
\end{eqnarray*}
$ \cc_{NCK} $ is freely generated by the trees, therefore
\begin{eqnarray*}
F_{\cc_{NCK}} (x) & = & \dfrac{1}{1-T_{\cc_{NCK}} (x)} = \dfrac{2x}{4x-1+\sqrt{1-4x}}
\end{eqnarray*}
\end{proof}
\\

Then for all $ n \geq 1 $ $ t_{n}^{\cc_{NCK}} = \frac{1}{n+1} \left( \substack{2n \\ n} \right) $ is the $ n $-Catalan number, $ f_{n}^{\cc_{NCK}} = \left( \substack{ 2n-1 \\ n} \right) $ and this gives:

$$\begin{array}{c|c|c|c|c|c|c|c|c|c|c}
n&1&2&3&4&5&6&7&8&9&10\\
\hline t_{n}^{\cc_{NCK}} &1&2&5&14&42&132&429&1430&4862&16796 \\
\hline f_{n}^{\cc_{NCK}} &1&3&10&35&126&462&1716&6435&24310&92378
\end{array}$$

These are the sequences A000108 and A088218 in \cite{Sloane}.

\section{Hopf algebra morphisms}

Recall that the tensor algebra $T(V)$ over a $\mathbb{K}$-vector space $V$ is the tensor module
$$ T(V) = \mathbb{K} \oplus V \oplus V^{\otimes 2} \oplus \hdots \oplus V^{\otimes n} \oplus \hdots $$
equipped with the concatenation.

Dually, the tensor coalgebra $T^{c}(V)$ over a $\mathbb{K}$-vector space $V$ is the tensor module (as above) equiped with the coassociative coproduct $\Delta_{\mathcal{A}ss}$ called deconcatenation:
$$ \Delta_{\mathcal{A}ss} ((v_{1} \hdots v_{n})) = \sum_{i=0}^{n} (v_{1} \hdots v_{i}) \otimes (v_{i+1} \hdots v_{n}) . $$
We will say that a bialgebra $\mathbf{H}$ is cofree if, as a coalgebra, it is isomorphic to $T^{c}(Prim(\mathbf{H}))$ (for more details, see \cite{Loday05}).\\

We give the useful following lemma:

\begin{lemma} \label{lemcofree}
Let $(A,\Delta,\varepsilon)$ be a cofree Hopf algebra. Then
\begin{eqnarray*}
Ker(\tdelta \otimes Id_{A} - Id_{A} \otimes \tdelta) = Im(\tdelta) .
\end{eqnarray*}
\end{lemma}

\begin{proof}
Indeed, if $ x = \sum a_{w,w'} w \otimes w' \in Ker(\tdelta \otimes Id_{A} - Id_{A} \otimes \tdelta) $,
\begin{eqnarray*}
\sum_{w_{1} w_{2} = w} a_{w,w'} w_{1} \otimes w_{2} \otimes w' = \sum_{w'_{1} w'_{2} = w'} a_{w,w'} w \otimes w'_{1} \otimes w'_{2} .
\end{eqnarray*}
So $ a_{w_{1} w_{2} , w_{3}} = a_{w_{1} , w_{2} w_{3}} $ for all words $ w_{1} , w_{2} , w_{3} $ different from the unit. We put $ b_{w w'} = a_{w , w'} $. Then
\begin{eqnarray*}
x = \sum b_{w} \left(  \sum_{w_{1} w_{2} = w} w_{1} \otimes w_{2} \right) = \tdelta \left( \sum b_{w} w \right) \in Im(\tdelta).
\end{eqnarray*}
The coassociativity of $ \tdelta $ implies the other inclusion.
\end{proof}

\subsection{From $ \hh^{\mathcal{D}}_{CK} $ to $ \Sh^{\mathcal{D}} $} \label{hsh}

Let $ \varphi : \mathbb{K} \left( \mathbb{T}^{\mathcal{D}}_{\hh_{CK}} \right) \rightarrow \mathbb{K} \left( \mathcal{D} \right) $ be a $ \mathbb{K} $-linear map.

\begin{theo} \label{theoremhsh}
There exists a unique Hopf algebra morphism $ \Phi : \hh^{\mathcal{D}}_{CK} \rightarrow \Sh^{\mathcal{D}} $ such that the following diagram
\begin{eqnarray} \label{diaghsh}
\xymatrix{
\mathbb{K} \left( \mathbb{T}^{\mathcal{D}}_{\hh_{CK}} \right) \ar[r]^{\varphi} \ar@{^{(}->}[d]^{i} & \mathbb{K} \left( \mathcal{D} \right) \\
\hh^{\mathcal{D}}_{CK} \ar[r]^{\Phi} & \Sh^{\mathcal{D}} \ar@{->>}[u]^{\pi}
}
\end{eqnarray}
is commutative.
\end{theo}

\begin{proof} \underline{Existence:} We define $ \Phi $ by induction on the number of vertices. We put $ \Phi(1) = 1 \otimes 1 $ and $ \Phi (\tdun{$a$}) = \varphi (\tdun{$a$}) $ for all $ a \in \mathcal{D} $. Suppose that $ \Phi $ is defined for all forest $ F $ of vertices degree $ < n $ and satisfies the condition $ (\Phi \otimes \Phi) \circ \tdelta_{\hh^{\mathcal{D}}_{CK}} (F) = \tdelta_{\Sh^{\mathcal{D}}} \circ \Phi (F) $. Let $ F \in \hh^{\mathcal{D}}_{CK} $ be a forest of vertices degree $ n $. If $ F = F_{1} F_{2} $, we put $ \Phi(F) = \Phi(F_{1}) \Phi(F_{2}) $. Suppose that $ F $ is a tree. By induction hypothesis, $ (\Phi \otimes \Phi) \circ \tdelta_{\hh^{\mathcal{D}}_{CK}} (F) $ is well defined. Moreover,
\begin{eqnarray*}
& & (\tdelta_{\Sh^{\mathcal{D}}} \otimes Id_{\Sh^{\mathcal{D}}} - Id_{\Sh^{\mathcal{D}}} \otimes \tdelta_{\Sh^{\mathcal{D}}}) \circ (\Phi \otimes \Phi) \circ \tdelta_{\hh^{\mathcal{D}}_{CK}} (F) \\
& = & (\Phi \otimes \Phi \otimes \Phi) \circ (\tdelta_{\hh^{\mathcal{D}}_{CK}} \otimes Id_{\hh^{\mathcal{D}}_{CK}} - Id_{\hh^{\mathcal{D}}_{CK}} \otimes \tdelta_{\hh^{\mathcal{D}}_{CK}}) \circ \tdelta_{\hh^{\mathcal{D}}_{CK}} (F) \\
& = & 0,
\end{eqnarray*}
using induction hypothesis in the first equality and the coassociativity in the second equality.

So $ (\Phi \otimes \Phi) \circ \tdelta_{\hh^{\mathcal{D}}_{CK}} (F) \in Ker(\tdelta_{\Sh^{\mathcal{D}}} \otimes Id_{\Sh^{\mathcal{D}}} - Id_{\Sh^{\mathcal{D}}} \otimes \tdelta_{\Sh^{\mathcal{D}}}) $. As $\Sh^{\mathcal{D}}$ is cofree, with lemma \ref{lemcofree}, $ (\Phi \otimes \Phi) \circ \tdelta_{\hh^{\mathcal{D}}_{CK}} (F) \in Im(\tdelta_{\Sh^{\mathcal{D}}}) $ and there exists $ w \in \Sh^{\mathcal{D}} $ such that $ (\Phi \otimes \Phi) \circ \tdelta_{\hh^{\mathcal{D}}_{CK}} (F) = \tdelta_{\Sh^{\mathcal{D}}} (w) $. We put $ \Phi(F) = w - \pi (w) + \varphi(F) $. Then
\begin{eqnarray*}
\pi \circ \Phi (F) & = & \pi(w) - \pi \circ \pi (w) + \pi \circ \varphi(F) = \varphi(F) ,\\
\tdelta_{\Sh^{\mathcal{D}}} \circ \Phi (F) & = & \tdelta_{\Sh^{\mathcal{D}}} (w) - \tdelta_{\Sh^{\mathcal{D}}}(\pi (w)) + \tdelta_{\Sh^{\mathcal{D}}} (\varphi(F)) \\
& = & \tdelta_{\Sh^{\mathcal{D}}} (w) \\
& = & (\Phi \otimes \Phi) \circ \tdelta_{\hh^{\mathcal{D}}_{CK}} (F) .
\end{eqnarray*}
By induction, the result is established.\\

\underline{Uniqueness:} Let $ \Phi_{1} $ and $ \Phi_{2} $ be two Hopf algebra morphisms such that the diagram (\ref{diaghsh}) is commutative. Let us prove that $ \Phi_{1}(T) = \Phi_{2}(T) $ for all tree $ T \in \hh^{\mathcal{D}}_{CK} $ by induction on the vertices degree of $ T $. If $ n=0 $, $ \Phi_{1}(1) = \Phi_{2}(1) = 1 $. If $ n=1 $, for $ i = 1 ,2 $, $ \tdelta_{\Sh^{\mathcal{D}}} \circ \Phi_{i}(\tdun{$ a $}) = (\Phi_{i} \otimes \Phi_{i}) \circ \tdelta_{\hh^{\mathcal{D}}_{CK}} (\tdun{$ a $}) = 0 $. So $ \Phi_{i}(\tdun{$ a $}) \in Vect(\mathcal{D}) $. As the diagram (\ref{diaghsh}) is commutative, $ \Phi_{1}(\tdun{$ a $}) = \Phi_{2}(\tdun{$ a $}) = \varphi(\tdun{$ a $}) $. Suppose that the result is true in vertices degree $ < n $ and let $ T $ be a tree of vertices degree $ n $. Using induction hypothesis in the second equality,
\begin{eqnarray*}
\tdelta_{\Sh^{\mathcal{D}}} \circ \Phi_{1}(T) & = & (\Phi_{1} \otimes \Phi_{1}) \circ \tdelta_{\hh^{\mathcal{D}}_{CK}} (T) \\
& = & (\Phi_{2} \otimes \Phi_{2}) \circ \tdelta_{\hh^{\mathcal{D}}_{CK}} (T) \\
& = & \tdelta_{\Sh^{\mathcal{D}}} \circ \Phi_{2}(T) .
\end{eqnarray*}
So $ \Phi_{1}(T) - \Phi_{2}(T) \in Vect(\mathbb{T}^{\mathcal{D}}_{\hh_{CK}}) $ and $ \Phi_{1}(T) - \Phi_{2}(T) = \pi (\Phi_{1}(T) - \Phi_{2}(T)) = \varphi(T) - \varphi(T) = 0 $.
\end{proof}
\\

{\bf Notation.} {We consider $ F \in \hh_{CK} $, $ \boldsymbol{e} \models E(F) $ and $ \sigma \in \mathcal{O}(Cont_{\boldsymbol{e}}(F)) $ a linear order on $Cont_{\boldsymbol{e}}(F)$ (see definition \ref{ordrelineaire}). For all $ i \in \{ 1 , \hdots , \left| Cont_{\boldsymbol{e}}(F) \right|_{v} \} $, $ \sigma^{-1}(i) $ is the connected component of $ Part_{\boldsymbol{e}}(F) $ such that her image by $ \sigma $ is equal to $ i $.
}
\vspace{0.5cm}

The following proposition give a combinatorial description of the morphism $ \Phi $ defined in theorem \ref{theoremhsh} :

\begin{prop} \label{propcombhsh}
Let $ T $ be a nonempty tree $ \in \hh^{\mathcal{D}}_{CK} $. Then
\begin{eqnarray} \label{formulecomb}
\Phi(T) = \sum_{\boldsymbol{e} \models E(T)} \left(  \sum_{\sigma \in \mathcal{O}(Cont_{\boldsymbol{e}}(T))} \varphi(\sigma^{-1}( \left| Cont_{\boldsymbol{e}}(F)\right|_{v} )) \hdots \varphi(\sigma^{-1}(1 )) \right)  .
\end{eqnarray}
\end{prop}

\begin{proof}
We use the following lemma:

\begin{lemma}
Let $ T $ be a rooted tree of vertices degree $ n $. We define :
\begin{eqnarray*}
\mathbb{E}(T) & = & \{ (\boldsymbol{v} , \sigma_{1} , \sigma_{2}) \:\mid \: \boldsymbol{v} \mmodels V(T) , \sigma_{1} \in \mathcal{O}(Lea_{\boldsymbol{v}}(T)) , \sigma_{2} \in \mathcal{O}(Roo_{\boldsymbol{v}}(T)) \} ,\\
\mathbb{F}(T) & = & \{ (\sigma , p) \:\mid \: \sigma \in \mathcal{O}(T) , p \in \{1, \hdots , n-1 \} \} .
\end{eqnarray*}
Then $ \mathbb{E}(T) $ and $ \mathbb{F}(T) $ are in bijection.
\end{lemma}

\begin{proof}
We define two maps $f$ and $g$.

Let $f$ be the map defined by
\begin{eqnarray*}
f: \left\lbrace 
\begin{array}{rcl}
\mathbb{E}(T) & \rightarrow & \mathbb{F}(T) \\
\left( \boldsymbol{v} , \sigma_{1} , \sigma_{2} \right)  & \mapsto & \left( \sigma , \left| Roo_{\boldsymbol{v}}(T) \right|_{v} \right) 
\end{array}
\right.
\end{eqnarray*}
where $\sigma : V(T) \rightarrow \{1, \hdots ,n \}$ is defined by $\sigma(v) = \sigma_{2}(v)$ for all $v \in V(Roo_{\boldsymbol{v}}(T))$ and $\sigma(v) = \sigma_{1}(v) + \left| Roo_{\boldsymbol{v}}(T) \right|_{v}$ for all $v \in V(Lea_{\boldsymbol{v}}(T))$. By definition, $\sigma \in \mathcal{O}(T)$.

Let $g$ be the map defined by
\begin{eqnarray*}
g: \left\lbrace
\begin{array}{rcl}
\mathbb{F}(T) & \rightarrow & \mathbb{E}(T) \\
\left( \sigma , p \right) & \mapsto & \left( \boldsymbol{v} , \sigma_{1} , \sigma_{2} \right) 
\end{array}
\right. 
\end{eqnarray*}
where
\begin{itemize}
\item $\sigma_{1} : V(Lea_{\boldsymbol{v}}(T)) \rightarrow \{ 1, \hdots , \left| Lea_{\boldsymbol{v}}(T) \right|_{v} \}$ is defined by $\sigma_{1}(v) = \sigma(v) - \left| Roo_{\boldsymbol{v}}(T) \right|_{v} $ for all $v \in V(Lea_{\boldsymbol{v}}(T))$. Then $\sigma_{1} \in \mathcal{O}(Lea_{\boldsymbol{v}}(T))$
\item $\sigma_{2} : V(Roo_{\boldsymbol{v}}(T)) \rightarrow \{ 1, \hdots , \left| Roo_{\boldsymbol{v}}(T) \right|_{v} \}$ is defined by $\sigma_{2}(v) = \sigma(v) $ for all $v \in V(Roo_{\boldsymbol{v}}(T))$. Then $\sigma_{2} \in \mathcal{O}(Roo_{\boldsymbol{v}}(T))$
\item $ \boldsymbol{v} $ is the subset $ \left\lbrace v \in \sigma^{-1}(\{k, \hdots , n \}) \:\mid \: \mbox{if } w \in \sigma^{-1}(\{k, \hdots , n \}) \mbox{ and } v \twoheadrightarrow w \mbox{ then } v = w \right\rbrace $ of $V(T)$. We have $\boldsymbol{v} \mmodels V(T)$.
\end{itemize}
So $f$ and $g$ are well defined. Then we show easily that $ f \circ g = Id_{\mathbb{F}(T)} $ and $ g \circ f = Id_{\mathbb{E}(T)} $.
\end{proof}
\\

Let us show formula (\ref{formulecomb}) by induction on the number $ n $ of vertices. If $ n = 1 $, $ T = \tdun{$ a $} $ with $ a \in \mathcal{D} $. Then $ \Phi(\tdun{$ a $}) = \varphi (\tdun{$ a $}) $ and formula (\ref{formulecomb}) is true. If $ n \geq 2 $,
\begin{eqnarray*}
& & \tdelta_{\Sh^{\mathcal{D}}} \circ \Phi (T)\\
& = & (\Phi \otimes \Phi) \circ \tdelta_{\hh^{\mathcal{D}}_{CK}} (T) \\
& = & \sum_{\boldsymbol{v} \mmodels V(T)} \Phi(Lea_{\boldsymbol{v}}(T)) \otimes \Phi(Roo_{\boldsymbol{v}}(T)) \\
& = & \sum_{\boldsymbol{v} \mmodels V(T)} \left( \sum_{\boldsymbol{e} \models E(Lea_{\boldsymbol{v}}(T))} \left( \sum_{\sigma_{1} \in \mathcal{O}(Cont_{\boldsymbol{e}}(Lea_{\boldsymbol{v}}(T)))} \varphi(\sigma_{1}^{-1}( \left| Cont_{\boldsymbol{e}}(Lea_{\boldsymbol{v}}(T)) \right|_{v} )) \hdots \varphi(\sigma_{1}^{-1}(1)) \right) \right) \\
& & \otimes \left( \sum_{\boldsymbol{f} \models E(Roo_{\boldsymbol{v}}(T))} \left( \sum_{\sigma_{2} \in \mathcal{O}(Cont_{\boldsymbol{f}}(Roo_{\boldsymbol{v}}(T)))} \varphi(\sigma_{2}^{-1}( \left| Cont_{\boldsymbol{f}}(Roo_{\boldsymbol{v}}(T)) \right|_{v} )) \hdots \varphi(\sigma_{2}^{-1}(1)) \right) \right) \\
& = & \sum_{\boldsymbol{e} \models E(T)} \sum_{(\boldsymbol{v}, \sigma_{1} , \sigma_{2} ) \in \mathbb{E}(Cont_{\boldsymbol{e}}(T))} \varphi(\sigma^{-1}_{1}( \left| Lea_{\boldsymbol{v}}(Cont_{\boldsymbol{e}}(T)) \right|_{v} )) \hdots \varphi(\sigma^{-1}_{1}(1 )) \\
& & \hspace{5cm} \otimes \varphi(\sigma^{-1}_{2}( \left| Roo_{\boldsymbol{v}}(Cont_{\boldsymbol{e}}(T)) \right|_{v} )) \hdots \varphi(\sigma^{-1}_{2}(1)) \\
& = & \sum_{\boldsymbol{e} \models E(T)} \sum_{(\sigma , p) \in \mathbb{F}(Cont_{\boldsymbol{e}}(T))} \varphi(\sigma^{-1}(\left| Cont_{\boldsymbol{e}}(T) \right|_{v})) \hdots \varphi(\sigma^{-1}(p+1)) \otimes \varphi(\sigma^{-1}(p)) \hdots \varphi(\sigma^{-1}( 1 )) .
\end{eqnarray*}
So
\begin{eqnarray*}
\Phi(T) = \sum_{\boldsymbol{e} \models E(T)} \left(  \sum_{\sigma \in \mathcal{O}(Cont_{\boldsymbol{e}}(T))} \varphi(\sigma^{-1}(\left| Cont_{\boldsymbol{e}}(F) \right|_{v} )) \hdots \varphi(\sigma^{-1}( 1)) \right)
\end{eqnarray*}
and by induction, we have the result.
\end{proof}
\\

{\bf Examples.} {\begin{itemize}
\item In vertices degree $1$, $\Phi(\tdun{$ a $}) = \varphi (\tdun{$ a $})$.
\item In vertices degree $2$,
\begin{eqnarray*}
\Phi(\tddeux{$ a $}{$ b $}) & = & \varphi (\tdun{$ b $}) \varphi (\tdun{$ a $}) + \varphi(\tddeux{$ a $}{$ b $}) \\
\Phi(\tdun{$ a $} \tdun{$ b $}) & = & \varphi (\tdun{$ a $}) \varphi (\tdun{$ b $}) + \varphi (\tdun{$ b $}) \varphi (\tdun{$ a $}) .
\end{eqnarray*}
\item In vertices degree $3$,
\begin{eqnarray*}
\Phi(\tdtroisun{$a$}{$c$}{$b$}) & = & \varphi (\tdun{$ b $}) \varphi (\tdun{$ c $}) \varphi (\tdun{$ a $}) + \varphi (\tdun{$ c $}) \varphi (\tdun{$ b $}) \varphi (\tdun{$ a $}) + \varphi (\tdun{$ b $}) \varphi(\tddeux{$ a $}{$ c $}) + \varphi (\tdun{$ c $}) \varphi(\tddeux{$ a $}{$ b $}) + \varphi(\tdtroisun{$a$}{$c$}{$b$}) \\
\Phi(\tdtroisdeux{$a$}{$b$}{$c$}) & = & \varphi (\tdun{$ c $}) \varphi (\tdun{$ b $}) \varphi (\tdun{$ a $}) + \varphi (\tdun{$ c $}) \varphi(\tddeux{$ a $}{$ b $}) + \varphi(\tddeux{$ b $}{$ c $}) \varphi (\tdun{$ a $}) + \varphi (\tdtroisdeux{$a$}{$b$}{$c$}) \\
\Phi(\tdquatredeux{$a$}{$d$}{$b$}{$c$}) & = & \varphi (\tdun{$ c $}) \varphi (\tdun{$ b $}) \varphi (\tdun{$ d $}) \varphi (\tdun{$ a $}) + \varphi (\tdun{$ c $}) \varphi (\tdun{$ d $}) \varphi (\tdun{$ b $}) \varphi (\tdun{$ a $}) + \varphi (\tdun{$ d $}) \varphi (\tdun{$ c $}) \varphi (\tdun{$ b $}) \varphi (\tdun{$ a $}) \\
& & + \varphi (\tdun{$ c $}) \varphi (\tdun{$ b $}) \varphi (\tddeux{$ a $}{$d$}) + \varphi (\tdun{$ c $}) \varphi (\tdun{$ d $}) \varphi (\tddeux{$ a $}{$b$}) + \varphi (\tdun{$ d $}) \varphi (\tdun{$ c $}) \varphi (\tddeux{$ a $}{$ b $}) + \varphi (\tddeux{$ b $}{$ c $}) \varphi (\tdun{$ d $}) \varphi (\tdun{$ a $}) \\
& & + \varphi (\tdun{$ d $}) \varphi (\tddeux{$ b $}{$ c $}) \varphi (\tdun{$ a $}) + \varphi (\tddeux{$ b $}{$ c $}) \varphi (\tddeux{$ a $}{$ d $}) + \varphi (\tdun{$ c $}) \varphi (\tdtroisun{$ a $}{$ d $}{$ b $}) + \varphi (\tdun{$ d $}) \varphi (\tdtroisdeux{$ a $}{$ b $}{$ c $}) + \varphi(\tdquatredeux{$a$}{$d$}{$b$}{$c$}) .
\end{eqnarray*}
\end{itemize}
}
\vspace{0.5cm}

{\bf Particular case.} {If $ \varphi(\tdun{$ a $}) = a $ for all $ a \in \mathcal{D} $ and $ \varphi(T) = 0 $ if $ \left| T \right|_{v} \geq 1 $, then this is the particular case of arborification (see \cite{Ecalle04}). For example :
\begin{eqnarray*}
\begin{array}{rcl|rcl|rcl}
\Phi(\tdun{$ a $}) & = & a & \Phi(\tddeux{$ a $}{$ b $}) & = & ba & \Phi(\tdun{$a$} \tdun{$b$}) & = & ab + ba\\
\Phi(\tdtroisun{$a$}{$c$}{$b$}) & = & bca + cba & \Phi(\tdtroisdeux{$a$}{$b$}{$c$}) & = & cba & \Phi(\tdquatredeux{$a$}{$d$}{$b$}{$c$}) & = & cbda + cdba + dcba.
\end{array}
\end{eqnarray*}
}

\subsection{From $ \hh^{\mathcal{D}}_{CK} $ to $ \Csh^{\mathcal{D}} $} \label{hcsh}

Let $ \varphi : \mathbb{K} \left( \mathbb{T}^{\mathcal{D}}_{\hh_{CK}} \right) \rightarrow \mathbb{K} \left( \mathcal{D} \right) $ be a $ \mathbb{K} $-linear map. We suppose that $ \mathcal{D} $ is equipped with an associative and commutative product $ \left[ \cdot , \cdot \right] : (a,b) \in \D^{2} \mapsto \left[ a b\right] \in \D $.

\begin{theo} \label{theohcsh}
There exists a unique Hopf algebra morphism $ \Phi : \hh^{\mathcal{D}}_{CK} \rightarrow \Csh^{\mathcal{D}} $ such that the following diagram
\begin{eqnarray} \label{diaghcsh}
\xymatrix{
\mathbb{K} \left( \mathbb{T}^{\mathcal{D}}_{\hh_{CK}} \right) \ar[r]^{\varphi} \ar@{^{(}->}[d]^{i} & \mathbb{K} \left( \mathcal{D} \right) \\
\hh^{\mathcal{D}}_{CK} \ar[r]^{\Phi} & \Csh^{\mathcal{D}} \ar@{->>}[u]^{\pi}
}
\end{eqnarray}
is commutative.
\end{theo}

\begin{proof}
Noting that $ \Csh^{\mathcal{D}} $ is cofree, this is the same proof as for theorem \ref{theoremhsh}.
\end{proof}
\\

{\bf Notation.} {Let $ F \in \hh_{CK} $ be a nonempty rooted forest, $ \boldsymbol{e} \models E(F) $ and $\sigma \in \mathcal{O}_{p}(Cont_{\boldsymbol{e}}(F))$ a lineair preorder on $Cont_{\boldsymbol{e}}(F)$ (see definition \ref{ordreprelineaire}), $ \sigma : V(Cont_{\boldsymbol{e}}(F)) \rightarrow \{ 1,\hdots,q \} $ surjective. For all $ i \in \{1, \hdots , q \} $, $ \sigma^{-1}(i) $ is the forest $T_{1} \hdots T_{n}$ of all connected components $T_{k}$ of $Part_{\boldsymbol{e}}(F)$ such that $\sigma(T_{k}) = i $ for all $k \in \{1, \hdots , n \} $. In this case, $ \varphi(\sigma^{-1}(i)) $ is the element $ \left[  \varphi(T_{1}) \hdots \varphi(T_{n}) \right]^{(n)} $.
}
\vspace{0.5cm}

Now, we give a combinatorial description of the morphism $ \Phi $ defined in theorem \ref{theohcsh}:

\begin{prop}
Let $ T $ be a nonempty tree $ \in \hh_{CK}^{\mathcal{D}} $. Then
\begin{eqnarray} \label{formulecombhcsh}
\Phi(T) = \sum_{\boldsymbol{e} \models E(T)} \Bigg(  \sum_{\substack{\sigma \in \mathcal{O}_{p}(Cont_{\boldsymbol{e}}(T)) \\ Im(\sigma) = \{ 1, \hdots , q\}}} \varphi(\sigma^{-1}( q)) \hdots \varphi(\sigma^{-1}( 1)) \Bigg)  .
\end{eqnarray}
\end{prop}

\begin{proof}
It suffices to resume the proof of proposition \ref{propcombhsh}. Note that, if $T$ is a rooted tree and $\boldsymbol{v} \models V(T)$, $Roo_{\boldsymbol{v}}(T)$ is a tree and $Lea_{\boldsymbol{v}}(T)$ is a forest. So there is possibly contractions for the product $\left[ \cdot , \cdot \right]$ to the left of $(\Phi \otimes \Phi) \circ \tdelta_{\mathbf{H}_{CK}^{\mathcal{D}}} (T)$. We deduce formula (\ref{formulecombhcsh}).
\end{proof}
\\

{\bf Examples.} {\begin{itemize}
\item In vertices degree $1$, $\Phi(\tdun{$ a $}) = \varphi (\tdun{$ a $})$.
\item In vertices degree $2$,
\begin{eqnarray*}
\Phi(\tdun{$ a $} \tdun{$ b $}) & = & \varphi (\tdun{$ a $}) \varphi (\tdun{$ b $}) + \varphi (\tdun{$ b $}) \varphi (\tdun{$ a $}) + \left[ \varphi(\tdun{$ a $}) \varphi(\tdun{$ b $}) \right] \\
\Phi(\tddeux{$ a $}{$ b $}) & = & \varphi (\tdun{$ b $}) \varphi (\tdun{$ a $}) + \varphi(\tddeux{$ a $}{$ b $}) .
\end{eqnarray*}
\item In vertices degree $3$,
\begin{eqnarray*}
\Phi(\tdtroisun{$a$}{$c$}{$b$}) & = & \varphi (\tdun{$ b $}) \varphi (\tdun{$ c $}) \varphi (\tdun{$ a $}) + \varphi (\tdun{$ c $}) \varphi (\tdun{$ b $}) \varphi (\tdun{$ a $}) + \left[ \varphi (\tdun{$ c $}) \varphi (\tdun{$ b $}) \right] \varphi (\tdun{$ a $}) + \varphi (\tdun{$ b $}) \varphi(\tddeux{$ a $}{$ c $}) \\
& & + \varphi (\tdun{$ c $}) \varphi(\tddeux{$ a $}{$ b $}) + \varphi(\tdtroisun{$a$}{$c$}{$b$}) \\
\Phi(\tdtroisdeux{$a$}{$b$}{$c$}) & = & \varphi (\tdun{$ c $}) \varphi (\tdun{$ b $}) \varphi (\tdun{$ a $}) + \varphi (\tdun{$ c $}) \varphi(\tddeux{$ a $}{$ b $}) + \varphi(\tddeux{$ b $}{$ c $}) \varphi (\tdun{$ a $}) + \varphi (\tdtroisdeux{$a$}{$b$}{$c$}) \\
\Phi(\tdquatredeux{$a$}{$d$}{$b$}{$c$}) & = & \varphi (\tdun{$ c $}) \varphi (\tdun{$ b $}) \varphi (\tdun{$ d $}) \varphi (\tdun{$ a $}) + \varphi (\tdun{$ c $}) \left[ \varphi (\tdun{$ b $}) \varphi (\tdun{$ d $}) \right] \varphi (\tdun{$ a $}) + \varphi (\tdun{$ c $}) \varphi (\tdun{$ d $}) \varphi (\tdun{$ b $}) \varphi (\tdun{$ a $}) \\
& & + \left[ \varphi (\tdun{$ c $}) \varphi (\tdun{$ d $}) \right] \varphi (\tdun{$ b $}) \varphi (\tdun{$ a $}) + \varphi (\tdun{$ d $}) \varphi (\tdun{$ c $}) \varphi (\tdun{$ b $}) \varphi (\tdun{$ a $}) + \varphi (\tdun{$ c $}) \varphi (\tdun{$ b $}) \varphi (\tddeux{$ a $}{$d$}) \\
& & + \varphi (\tdun{$ c $}) \varphi (\tdun{$ d $}) \varphi (\tddeux{$ a $}{$b$}) + \varphi (\tdun{$ d $}) \varphi (\tdun{$ c $}) \varphi (\tddeux{$ a $}{$ b $}) + \varphi (\tddeux{$ b $}{$ c $}) \varphi (\tdun{$ d $}) \varphi (\tdun{$ a $}) \\
& & + \left[ \varphi (\tddeux{$ b $}{$ c $}) \varphi (\tdun{$ d $}) \right]  \varphi (\tdun{$ a $}) + \varphi (\tdun{$ d $}) \varphi (\tddeux{$ b $}{$ c $}) \varphi (\tdun{$ a $}) + \varphi (\tddeux{$ b $}{$ c $}) \varphi (\tddeux{$ a $}{$ d $}) + \varphi (\tdun{$ c $}) \varphi (\tdtroisun{$ a $}{$ d $}{$ b $}) \\
& & + \varphi (\tdun{$ d $}) \varphi (\tdtroisdeux{$ a $}{$ b $}{$ c $}) + \varphi(\tdquatredeux{$a$}{$d$}{$b$}{$c$}) .
\end{eqnarray*}
\end{itemize}
}

\subsection{From $ \cc^{\mathcal{D}}_{CK} $ to $ \Sh^{\mathcal{D}} $} \label{csh}

Let $ \varphi : \mathbb{K} \left(  \mathbb{T}^{\mathcal{D}}_{\cc_{CK}} \right) \rightarrow \mathbb{K} \left( \mathcal{D} \right) $ be a $ \mathbb{K} $-linear map.

\begin{theo} \label{theocsh}
There exists a unique Hopf algebra morphism $ \Phi : \cc^{\mathcal{D}}_{CK} \rightarrow \Sh^{\mathcal{D}} $ such that the following diagram
\begin{eqnarray} \label{diagcsh}
\xymatrix{
\mathbb{K} \left( \mathbb{T}^{\mathcal{D}}_{\cc_{CK}} \right)  \ar[r]^{\varphi} \ar@{^{(}->}[d]^{i} & \mathbb{K} \left( \mathcal{D} \right) \\
\cc^{\mathcal{D}}_{CK} \ar[r]^{\Phi} & \Sh^{\mathcal{D}} \ar@{->>}[u]^{\pi}
}
\end{eqnarray}
is commutative.
\end{theo}

\begin{proof}
This is the same proof as for theorem \ref{theoremhsh}.
\end{proof}
\\

As in the sections \ref{hsh} and \ref{hcsh}, we give a combinatorial description of the morphism $\Phi$ defined in theorem \ref{theocsh}. We need the following definition:

\begin{defi} \label{partg}
Let $F$ be a nonempty rooted forest of $\cc_{CK}$. A generalized partition of $F$ is a $k$-uplet $(\boldsymbol{e}_{1}, \hdots , \boldsymbol{e}_{k})$ of subsets of $E(F)$, $ 1 \leq k \leq \left| F \right|_{e} $, such that:
\begin{enumerate}
\item $\boldsymbol{e}_{i} \neq \emptyset $, $\boldsymbol{e}_{i} \cap \boldsymbol{e}_{j} = \emptyset $ if $i \neq j$ and $\cup_{i} \boldsymbol{e}_{i} = E(F) $,
\item the edges $\in \boldsymbol{e}_{i}$ are the edges of the same connected component of $F$,
\item if $v$ and $w$ are two vertices of $Part_{\boldsymbol{e}_{i}}(F)$ and if the shortest path in $F$ between $v$ and $w$ contains an edge $\in \boldsymbol{e}_{j}$, then $j < i$.
\end{enumerate}
We shall denote by $\mathcal{P}(F)$ the set of generalized partitions of $F$.
\end{defi}

{\bf Remark.} {If $F$ is a nonempty rooted forest and if $(\boldsymbol{e}_{1}, \hdots , \boldsymbol{e}_{k}) \in \mathcal{P}(F)$, $Cont_{\overline{\boldsymbol{e}_{i}}}(F) = Part_{\boldsymbol{e}_{i}}(F)$ is a tree for all $i$ (with the second point of the definition \ref{partg}).}

\begin{prop} \label{propcombcsh}
Let $F$ be a nonempty forest $\in \cc_{CK}^{\D}$. Then
\begin{eqnarray} \label{formulecombcsh}
\Phi (F) = \sum_{(\boldsymbol{e}_{1}, \hdots , \boldsymbol{e}_{k}) \in \mathcal{P}(F)} \varphi(Cont_{\overline{\boldsymbol{e}_{1}}}(F)) \hdots \varphi(Cont_{\overline{\boldsymbol{e}_{k}}}(F)) .
\end{eqnarray}
\end{prop}

\begin{proof}
We use the following lemma:

\begin{lemma} \label{lemint}
If $F \in \cc_{CK}^{\D} $ is a nonempty tree, then the sets
\begin{eqnarray*}
\mathbb{E}(F) & = & \{ \left( (\boldsymbol{e}_{1}, \hdots , \boldsymbol{e}_{k}),p \right) \:\mid \: (\boldsymbol{e}_{1}, \hdots , \boldsymbol{e}_{k}) \in \mathcal{P}(F), 1 \leq p \leq k-1 \} \\
\mathbb{F}(F) & = & \{ \left( \boldsymbol{e} , (\boldsymbol{f}_{1}, \hdots , \boldsymbol{f}_{q}) , (\boldsymbol{g}_{1}, \hdots , \boldsymbol{g}_{r}) \right) \:\mid \: \boldsymbol{e} \mmodels E(F), (\boldsymbol{f}_{1}, \hdots , \boldsymbol{f}_{q}) \in \mathcal{P}(Part_{\boldsymbol{e}}(F)),\\
& & (\boldsymbol{g}_{1}, \hdots , \boldsymbol{g}_{r}) \in \mathcal{P}(Cont_{\boldsymbol{e}}(F)) \}
\end{eqnarray*}
are in bijections.
\end{lemma}

\begin{proof}
Consider the following two maps:
\begin{eqnarray*}
f: \left\lbrace 
\begin{array}{rcl}
\mathbb{E}(F) & \rightarrow & \mathbb{F}(F) \\
\left( (\boldsymbol{e}_{1}, \hdots , \boldsymbol{e}_{k}),p \right) & \mapsto & \left( \cup_{1 \leq i \leq p} \boldsymbol{e}_{i} , (\boldsymbol{e}_{1}, \hdots , \boldsymbol{e}_{p}), (\boldsymbol{e}_{p+1}, \hdots , \boldsymbol{e}_{k}) \right)
\end{array}
\right. 
\end{eqnarray*}
and
\begin{eqnarray*}
g : \left\lbrace 
\begin{array}{rcl}
\mathbb{F}(F) & \rightarrow & \mathbb{E}(F) \\
\left(\boldsymbol{e}, (\boldsymbol{f}_{1}, \hdots , \boldsymbol{f}_{q}) , (\boldsymbol{g}_{1}, \hdots , \boldsymbol{g}_{r}) \right) & \mapsto & \left( (\boldsymbol{f}_{1}, \hdots , \boldsymbol{f}_{q}, \boldsymbol{g}_{1}, \hdots , \boldsymbol{g}_{r}) , q \right) .
\end{array}
\right. 
\end{eqnarray*}

\noindent \underline{$f$ is well defined :}

Let $\left( (\boldsymbol{e}_{1}, \hdots , \boldsymbol{e}_{k}),p \right) \in \mathbb{E}(F)$. Then $\boldsymbol{e} = \cup_{1 \leq i \leq p} \boldsymbol{e}_{i} $ is a nonempty nontotal contraction of $F$.
\begin{enumerate}
\item \begin{enumerate}
\item $(\boldsymbol{e}_{1}, \hdots , \boldsymbol{e}_{p})$ is a $p$-uplet of subsets of $E(Part_{\boldsymbol{e}}(F)) = \boldsymbol{e}$. By hypothesis, $(\boldsymbol{e}_{1}, \hdots , \boldsymbol{e}_{k}) \in \mathcal{P}(F)$. So $\boldsymbol{e}_{i} \neq \emptyset $, $\boldsymbol{e}_{i} \cap \boldsymbol{e}_{j} = \emptyset $ and $\cup_{1 \leq i \leq p} \boldsymbol{e}_{i} = E(Part_{\boldsymbol{e}}(F)) $.
\item The edges $\in \boldsymbol{e}_{i}$, $1 \leq i \leq p$, are the edges of the same connected component of $F$ therefore of $Part_{\boldsymbol{e}}(F)$ because $\boldsymbol{e}_{i} \subseteq \boldsymbol{e}$.
\item Let $v$ and $w$ be two vertices of $Part_{\boldsymbol{e}_{i}}(Part_{\boldsymbol{e}}(F)) = Part_{\boldsymbol{e}_{i}}(F)$ (because $\boldsymbol{e}_{i} \subseteq \boldsymbol{e}$). If the shortest path in $Part_{\boldsymbol{e}}(F)$ between $v$ and $w$ contains an edge $\in \boldsymbol{e}_{j}$, then the shortest path in $F$ between $v$ and $w$ contains also an edge $\in \boldsymbol{e}_{j}$. As $(\boldsymbol{e}_{1}, \hdots , \boldsymbol{e}_{k}) \in \mathcal{P}(F)$, we have $j<i$. 
\end{enumerate}
So $(\boldsymbol{e}_{1}, \hdots , \boldsymbol{e}_{p}) \in \mathcal{P}(Part_{\boldsymbol{e}}(F))$.

\item \begin{enumerate}
\item $(\boldsymbol{e}_{p+1}, \hdots , \boldsymbol{e}_{k})$ is a $(k-p)$-uplet of subsets of $E(Cont_{\boldsymbol{e}}(F)) = \overline{\boldsymbol{e}}$. By hypothesis, $(\boldsymbol{e}_{1}, \hdots , \boldsymbol{e}_{k}) \in \mathcal{P}(F)$. So $\boldsymbol{e}_{i} \neq \emptyset $, $\boldsymbol{e}_{i} \cap \boldsymbol{e}_{j} = \emptyset $ and $\cup_{p+1 \leq i \leq k} \boldsymbol{e}_{i} = E(Cont_{\boldsymbol{e}}(F)) $.
\item The edges $\in \boldsymbol{e}_{i}$, $p+1 \leq i \leq k$, are the edges of the same connected component of $F$ therefore of $Cont_{\boldsymbol{e}}(F)$ (we contract in $F$ some connected components).
\item Let $i$ be an integer $\in \{p+1, \hdots , k\}$ and $v$ and $w$ two vertices of $Part_{\boldsymbol{e}_{i}}(Cont_{\boldsymbol{e}}(F)) = Part_{\boldsymbol{e}_{i}}(F)$ (because $\boldsymbol{e}_{i} \cap \boldsymbol{e} = \emptyset$). If the shortest path in $Cont_{\boldsymbol{e}}(F)$ between $v$ and $w$ contains an edge $\in \boldsymbol{e}_{j}$ then the shortest path in $F$ between $v$ and $w$ contains also an edge $\in \boldsymbol{e}_{j}$. As $(\boldsymbol{e}_{1}, \hdots , \boldsymbol{e}_{k}) \in \mathcal{P}(F)$, we have $j<i$. 
\end{enumerate}
Thus $(\boldsymbol{e}_{p+1}, \hdots , \boldsymbol{e}_{k}) \in \mathcal{P}(Cont_{\boldsymbol{e}}(F))$.
\end{enumerate}

So $f \left( \left( (\boldsymbol{e}_{1}, \hdots , \boldsymbol{e}_{k}),p \right) \right) \in \mathbb{F}(F)$.\\

\noindent \underline{$g$ is well defined :}

Let $\left(\boldsymbol{e}, (\boldsymbol{f}_{1}, \hdots , \boldsymbol{f}_{q}) , (\boldsymbol{g}_{1}, \hdots , \boldsymbol{g}_{r}) \right) \in \mathbb{F}(F)$. Let us show that $(\boldsymbol{f}_{1}, \hdots , \boldsymbol{f}_{q},\boldsymbol{g}_{1}, \hdots , \boldsymbol{g}_{r}) \in \mathcal{P}(F)$.
\begin{enumerate}
\item As $(\boldsymbol{f}_{1}, \hdots , \boldsymbol{f}_{q}) \in \mathcal{P}(Part_{\boldsymbol{e}}(F))$ and $ (\boldsymbol{g}_{1}, \hdots , \boldsymbol{g}_{r}) \in \mathcal{P}(Cont_{\boldsymbol{e}}(F))$, $\boldsymbol{f}_{i} \neq \emptyset $, $\boldsymbol{g}_{i} \neq \emptyset$, $\boldsymbol{f}_{i} \cap \boldsymbol{f}_{j} = \emptyset$, $\boldsymbol{g}_{i} \cap \boldsymbol{g}_{j} = \emptyset$ and $\left( \cup_{i} \boldsymbol{f}_{i} \right) \bigcup \left( \cup_{i} \boldsymbol{g}_{i} \right) = E(Part_{\boldsymbol{e}}(F)) \bigcup E(Cont_{\boldsymbol{e}}(F)) = E(F) $. In addition, as $\boldsymbol{f}_{i} \subseteq E(Part_{\boldsymbol{e}}(F)) = \boldsymbol{e} $ and $\boldsymbol{g}_{j} \subseteq E(Cont_{\boldsymbol{e}}(F)) = \overline{\boldsymbol{e}}$, $\boldsymbol{f}_{i} \cap \boldsymbol{g}_{j} = \emptyset$.
\item The edges $\in \boldsymbol{f}_{i}$ are the edges of the same connected component of $Part_{\boldsymbol{e}}(F)$. As all the trees of the forest $Part_{\boldsymbol{e}}(F)$ are subtrees of $F$, the edges $\in \boldsymbol{f}_{i}$ are the edges of the same connected component of $F$. Moreover if the edges $\in \boldsymbol{g}_{i}$ are the edges of the same connected component of $Cont_{\boldsymbol{e}}(F)$, it is also true in $F$.
\item \begin{enumerate}
\item Let $i$ be an integer $\in \{1, \hdots , q\}$ and $v$ and $w$ two vertices of $Part_{\boldsymbol{f}_{i}}(F)$. We have $\boldsymbol{f}_{i} \subseteq \boldsymbol{e}$ therefore $Part_{\boldsymbol{f}_{i}}(F) = Part_{\boldsymbol{f}_{i}}(Part_{\boldsymbol{e}}(F))$. If the shortest path in $F$ between $v$ and $w$ contains:
\begin{enumerate}
\item an edge $\in \boldsymbol{f}_{j}$. As $(\boldsymbol{f}_{1}, \hdots , \boldsymbol{f}_{q}) \in \mathcal{P}(Part_{\boldsymbol{e}}(F))$, $j<i$.
\item an edge $\in \boldsymbol{g}_{j}$. Then the connected component of $Part_{\boldsymbol{e}}(F)$ containing $v$ and $w$ has an edge $\in \boldsymbol{g}_{j}$. This is impossible because $E(Part_{\boldsymbol{e}}(F)) = \boldsymbol{e}$ and $\boldsymbol{g}_{j} \subseteq E(Cont_{\boldsymbol{e}}(F)) = \overline{\boldsymbol{e}}$.
\end{enumerate}
\item Let $i$ be an integer $\in \{1, \hdots , r\}$ and $v$ and $w$ two vertices of $Part_{\boldsymbol{g}_{i}}(F)$. $\boldsymbol{g}_{i} \cap \boldsymbol{e} = \emptyset$ therefore $Part_{\boldsymbol{g}_{i}}(F) = Part_{\boldsymbol{g}_{i}}(Cont_{\boldsymbol{e}}(F))$. If the shortest path in $F$ between $v$ and $w$ contains:
\begin{enumerate}
\item an edge $\in \boldsymbol{g}_{j}$. As $ (\boldsymbol{g}_{1}, \hdots , \boldsymbol{g}_{r}) \in \mathcal{P}(Cont_{\boldsymbol{e}}(F))$, $j<i$.
\item an edge $\in \boldsymbol{f}_{j}$. It is good because $\boldsymbol{f}_{j}$ is before $\boldsymbol{g}_{i}$.
\end{enumerate}
\end{enumerate}
Thus $(\boldsymbol{f}_{1}, \hdots , \boldsymbol{f}_{q},\boldsymbol{g}_{1}, \hdots , \boldsymbol{g}_{r}) \in \mathcal{P}(F)$.
\end{enumerate}

So $g \left( \left(\boldsymbol{e}, (\boldsymbol{f}_{1}, \hdots , \boldsymbol{f}_{q}) , (\boldsymbol{g}_{1}, \hdots , \boldsymbol{g}_{r}) \right) \right) \in \mathbb{E}(F)$

\vspace{0.5cm}

\noindent Finally, we easily see that $ f \circ g = Id_{\mathbb{F}(F)} $ and $ g \circ f = Id_{\mathbb{E}(F)} $.
\end{proof}
\\

We now prove proposition \ref{propcombcsh}. By induction on the edges degree $n$ of $F \in \cc_{CK}^{\D}$. If $n=1$, $F=\addeux{$a$}$ with $a \in \D$. Then $\Phi (\addeux{$a$}) = \varphi(\addeux{$a$})$ and formula (\ref{formulecombcsh}) is true. Suppose that $n \geq 2$ and that the property is true in degrees $k < n$. Then
\begin{eqnarray*}
\tdelta_{\Sh^{\mathcal{D}}} \circ \Phi (F) & = & (\Phi \circ \Phi) \circ \tdelta_{\cc_{CK}^{\D}} (F) \\
& = & \sum_{\boldsymbol{e} \mmodels E(F)} \Phi \left( Part_{\boldsymbol{e}}(F) \right) \otimes \Phi \left( Cont_{\boldsymbol{e}}(F) \right) \\
& = & \sum_{\boldsymbol{e} \mmodels E(F)} \left( \sum_{(\boldsymbol{f}_{1}, \hdots , \boldsymbol{f}_{q}) \in \mathcal{P}(Part_{\boldsymbol{e}}(F))} \varphi(Cont_{\overline{\boldsymbol{f}_{1}}}(F)) \hdots \varphi(Cont_{\overline{\boldsymbol{f}_{q}}}(F)) \right) \\
& & \hspace{2.5cm} \otimes \left( \sum_{(\boldsymbol{g}_{1}, \hdots , \boldsymbol{g}_{r}) \in \mathcal{P}(Cont_{\boldsymbol{e}}(F))} \varphi(Cont_{\overline{\boldsymbol{g}_{1}}}(F)) \hdots \varphi(Cont_{\overline{\boldsymbol{g}_{r}}}(F)) \right)
\end{eqnarray*}
using induction hypothesis in the last equality. So, with lemma \ref{lemint},
\begin{eqnarray*}
\tdelta_{\Sh^{\mathcal{D}}} \circ \Phi (F)
& = & \sum_{\left( \boldsymbol{e} , (\boldsymbol{f}_{1}, \hdots , \boldsymbol{f}_{q}) , (\boldsymbol{g}_{1}, \hdots , \boldsymbol{g}_{r}) \right) \in \mathbb{F}(F)} \varphi(Cont_{\overline{\boldsymbol{f}_{1}}}(F)) \hdots \varphi(Cont_{\overline{\boldsymbol{f}_{q}}}(F)) \\
& & \hspace{5cm} \otimes \varphi(Cont_{\overline{\boldsymbol{g}_{1}}}(F)) \hdots \varphi(Cont_{\overline{\boldsymbol{g}_{r}}}(F)) \\
& = & \sum_{\left( (\boldsymbol{e}_{1}, \hdots , \boldsymbol{e}_{k}),p \right) \in \mathbb{E}(F)} \varphi(Cont_{\overline{\boldsymbol{e}_{1}}}(F)) \hdots \varphi(Cont_{\overline{\boldsymbol{e}_{p}}}(F)) \\
& & \hspace{5cm} \otimes \varphi(Cont_{\overline{\boldsymbol{e}_{p+1}}}(F)) \hdots \varphi(Cont_{\overline{\boldsymbol{e}_{k}}}(F)) \\
& = & \sum_{(\boldsymbol{e}_{1}, \hdots , \boldsymbol{e}_{k}) \in \mathcal{P}(F)} \sum_{1 \leq p \leq k-1} \varphi(Cont_{\overline{\boldsymbol{e}_{1}}}(F)) \hdots \varphi(Cont_{\overline{\boldsymbol{e}_{p}}}(F)) \\
& & \hspace{5cm} \otimes \varphi(Cont_{\overline{\boldsymbol{e}_{p+1}}}(F)) \hdots \varphi(Cont_{\overline{\boldsymbol{e}_{k}}}(F)) .
\end{eqnarray*}
Therefore
\begin{eqnarray*}
\Phi (F) = \sum_{(\boldsymbol{e}_{1}, \hdots , \boldsymbol{e}_{k}) \in \mathcal{P}(F)} \varphi(Cont_{\overline{\boldsymbol{e}_{1}}}(F)) \hdots \varphi(Cont_{\overline{\boldsymbol{e}_{k}}}(F))
\end{eqnarray*}
and by induction, we have the result.
\end{proof}
\\

{\bf Examples.} {We introduce a notation. If $w = w_{1} \hdots w_{n}$ is a $\mathcal{D}$-word, we denote $\mbox{Perm}(w)$ the sum of all $\mathcal{D}$-words whose letters are $w_{1}, \hdots , w_{n}$. For example, $\mbox{Perm}(abc) = abc + acb + bac + bca + cab+ cba$.
\begin{itemize}
\item In edges degree $1$, $\Phi (\addeux{$a$}) = \varphi ( \addeux{$a$})$.
\item In edges degree $2$,
\begin{eqnarray*}
\Phi (\adtroisun{$b$}{$a$}) & = & \varphi (\adtroisun{$b$}{$a$}) + \varphi ( \addeux{$a$}) \varphi ( \addeux{$b$}) + \varphi ( \addeux{$b$}) \varphi ( \addeux{$a$}) \\
\Phi (\adtroisdeux{$a$}{$b$}) & = & \varphi(\adtroisdeux{$a$}{$b$}) + \varphi ( \addeux{$a$}) \varphi ( \addeux{$b$}) + \varphi ( \addeux{$b$}) \varphi ( \addeux{$a$}) .
\end{eqnarray*}
\item In edges degree $3$,
\begin{eqnarray*}
\Phi (\adquatreun{$a$}{$b$}{$c$}) & = & \varphi (\adquatreun{$a$}{$b$}{$c$}) + \varphi (\addeux{$a$}) \varphi (\adtroisun{$c$}{$b$}) + \varphi (\adtroisun{$c$}{$b$}) \varphi (\addeux{$a$}) + \varphi (\addeux{$b$}) \varphi (\adtroisun{$c$}{$a$}) + \varphi (\adtroisun{$c$}{$a$}) \varphi (\addeux{$b$})\\
& & + \varphi (\addeux{$c$}) \varphi (\adtroisun{$b$}{$a$}) + \varphi (\adtroisun{$b$}{$a$}) \varphi (\addeux{$c$}) + \mbox{Perm}(\varphi (\addeux{$a$}) \varphi (\addeux{$b$}) \varphi (\addeux{$c$})) \\
\Phi (\adquatrequatre{$a$}{$b$}{$c$}) & = & \varphi(\adquatrequatre{$a$}{$b$}{$c$}) + \varphi (\addeux{$a$}) \varphi (\adtroisun{$c$}{$b$}) + \varphi (\adtroisun{$c$}{$b$}) \varphi (\addeux{$a$}) + \varphi (\addeux{$c$}) \varphi (\adtroisdeux{$a$}{$b$}) + \varphi (\adtroisdeux{$a$}{$b$}) \varphi (\addeux{$c$}) \\
& & + \varphi (\addeux{$b$}) \varphi (\adtroisdeux{$a$}{$c$}) + \varphi (\adtroisdeux{$a$}{$c$}) \varphi (\addeux{$b$}) + \mbox{Perm}(\varphi (\addeux{$a$})  \varphi (\addeux{$b$}) \varphi (\addeux{$c$})) \\
\Phi (\adquatredeux{$b$}{$a$}{$c$}) & = & \varphi (\adquatredeux{$b$}{$a$}{$c$}) + \varphi (\addeux{$a$}) \varphi (\adtroisun{$b$}{$c$}) + \varphi (\addeux{$b$}) \varphi (\adtroisdeux{$a$}{$c$}) + \varphi (\adtroisdeux{$a$}{$c$}) \varphi (\addeux{$b$}) + \varphi (\addeux{$c$}) \varphi (\adtroisun{$b$}{$a$}) \\
& & + \varphi (\adtroisun{$b$}{$a$}) \varphi (\addeux{$c$}) + \mbox{Perm}(\varphi (\addeux{$a$})  \varphi (\addeux{$b$}) \varphi (\addeux{$c$})) .
\end{eqnarray*}
\item Finally, in edges degree $4$, with the tree $\adcinqsix{$a$}{$b$}{$c$}{$d$}$,
\begin{eqnarray*}
\Phi (\adcinqsix{$a$}{$b$}{$c$}{$d$}) & = & \varphi (\adcinqsix{$a$}{$b$}{$c$}{$d$}) + \varphi (\addeux{$a$}) \varphi (\adquatreun{$c$}{$d$}{$b$}) + \varphi (\addeux{$b$}) \varphi (\adquatrequatre{$a$}{$c$}{$d$}) + \varphi (\adquatrequatre{$a$}{$c$}{$d$}) \varphi (\addeux{$b$}) + \varphi (\addeux{$c$}) \varphi (\adquatredeux{$b$}{$a$}{$d$}) \\
& & + \varphi (\adquatredeux{$b$}{$a$}{$d$}) \varphi (\addeux{$c$}) + \varphi (\addeux{$d$}) \varphi (\adquatredeux{$b$}{$a$}{$c$}) + \varphi (\adquatredeux{$b$}{$a$}{$c$}) \varphi (\addeux{$d$}) + \varphi(\adtroisun{$b$}{$a$}) \varphi(\adtroisun{$d$}{$c$}) \\
& & + \varphi(\adtroisun{$d$}{$c$}) \varphi(\adtroisun{$b$}{$a$}) + \varphi (\adtroisdeux{$a$}{$d$}) \varphi(\adtroisun{$b$}{$c$}) + \varphi (\adtroisdeux{$a$}{$c$}) \varphi(\adtroisun{$b$}{$d$}) + \mbox{Perm}(\varphi(\adtroisun{$b$}{$a$}) \varphi (\addeux{$c$}) \varphi (\addeux{$d$})) \\
& & + \mbox{Perm}(\varphi(\adtroisun{$d$}{$c$}) \varphi (\addeux{$a$}) \varphi (\addeux{$b$})) + \mbox{Perm}(\varphi(\adtroisdeux{$a$}{$d$}) \varphi (\addeux{$b$}) \varphi (\addeux{$c$})) \\
& & + \mbox{Perm}(\varphi(\adtroisdeux{$a$}{$c$}) \varphi (\addeux{$b$}) \varphi (\addeux{$d$})) + \varphi (\addeux{$a$}) \varphi(\adtroisun{$c$}{$b$}) \varphi (\addeux{$d$}) + \varphi (\addeux{$a$}) \varphi (\addeux{$d$}) \varphi(\adtroisun{$c$}{$b$}) \\
& & + \varphi (\addeux{$d$}) \varphi (\addeux{$a$}) \varphi(\adtroisun{$c$}{$b$}) + \varphi (\addeux{$a$}) \varphi(\adtroisun{$d$}{$b$}) \varphi (\addeux{$c$}) + \varphi (\addeux{$a$}) \varphi (\addeux{$c$}) \varphi(\adtroisun{$d$}{$b$}) \\
& & + \varphi (\addeux{$c$}) \varphi (\addeux{$a$}) \varphi(\adtroisun{$d$}{$b$}) + \mbox{Perm}(\varphi (\addeux{$a$}) \varphi (\addeux{$b$}) \varphi (\addeux{$c$}) \varphi (\addeux{$d$}))
\end{eqnarray*}
\end{itemize}
}

\subsection{From $ \cc^{\mathcal{D}}_{CK} $ to $ \Csh^{\mathcal{D}} $}

Let $ \varphi : \mathbb{K} \left( \mathbb{T}^{\mathcal{D}}_{\cc_{CK}} \right) \rightarrow \mathbb{K} \left( \mathcal{D} \right) $ be a $ \mathbb{K} $-linear map. We suppose that $ \mathcal{D} $ is equipped with an associative and commutative product $ \left[ \cdot , \cdot \right] : (a,b) \in \D^{2} \mapsto \left[ a b\right] \in \D $.

\begin{theo} \label{theoccsh}
There exists a unique Hopf algebra morphism $ \Phi : \cc^{\mathcal{D}}_{CK} \rightarrow \Csh^{\mathcal{D}} $ such that the following diagram
\begin{eqnarray} \label{diagccsh}
\xymatrix{
\mathbb{K} \left( \mathbb{T}^{\mathcal{D}}_{\cc_{CK}} \right)  \ar[r]^{\varphi} \ar@{^{(}->}[d]^{i} & \mathbb{K} \left( \mathcal{D} \right) \\
\cc^{\mathcal{D}}_{CK} \ar[r]^{\Phi} & \Csh^{\mathcal{D}} \ar@{->>}[u]^{\pi}
}
\end{eqnarray}
is commutative.
\end{theo}

\begin{proof}
This is the same proof as for theorem \ref{theoremhsh}.
\end{proof}
\\

We give a combinatorial description of the morphism $\Phi$ defined in theorem \ref{theoccsh}. For this, we give the following definition:

\begin{defi}
Let $F$ be a nonempty rooted forest of $\cc_{CK}$. A generalized and contracted partition of $F$ is a $l$-uplet $(f_{1}, \hdots , f_{l})$ such that:
\begin{enumerate}
\item for all $1 \leq i \leq l$, $f_{i} = (\boldsymbol{e}^{i}_{1}, \hdots , \boldsymbol{e}^{i}_{k_{i}})$ is a $k_{i}$-uplet of subsets of $E(F)$,
\item $(\boldsymbol{e}^{1}_{1}, \hdots ,\boldsymbol{e}^{1}_{k_{1}}, \boldsymbol{e}^{2}_{1}, \hdots , \boldsymbol{e}^{l}_{k_{l}}) \in \mathcal{P}(F)$,
\item if $Part_{\boldsymbol{e}^{i}_{p}}(F)$ and $Part_{\boldsymbol{e}^{i}_{q}}(F)$ are two disconnected components of $F$ and if the shortest path in $F$ between $Part_{\boldsymbol{e}^{i}_{p}}(F)$ and $Part_{\boldsymbol{e}^{i}_{q}}(F)$ contains an edge $\in \boldsymbol{e}^{j}_{r}$, then $j>i$.
\end{enumerate}
We shall denote by $\mathcal{P}_{c}(F)$ the set of generalized and contracted partitions of $F$.
\end{defi}

\begin{prop} \label{propcombccsh}
Let $F$ be a nonempty forest $\in \cc_{CK}^{\D}$. Then
\begin{eqnarray} \label{formulecombccsh}
\begin{array}{rcl}
\Phi (F) & = & \displaystyle\sum_{\substack{(f_{1}, \hdots, f_{l}) \in \mathcal{P}_{c}(F)\\f_{i} = (\boldsymbol{e}^{i}_{1}, \hdots , \boldsymbol{e}^{i}_{k_{i}})}} \left( \left[ \varphi(Cont_{\overline{\boldsymbol{e}^{1}_{1}}}(F)) \hdots \varphi(Cont_{\overline{\boldsymbol{e}^{1}_{k_{1}}}}(F)) \right]^{(k_{1})} \hdots \right. \\
& & \hspace{5cm} \hdots \left. \left[ \varphi(Cont_{\overline{\boldsymbol{e}^{l}_{1}}}(F)) \hdots \varphi(Cont_{\overline{\boldsymbol{e}^{l}_{k_{l}}}}(F)) \right]^{(k_{l})} \right).
\end{array}
\end{eqnarray}
\end{prop}

\begin{proof}
It suffices to resume the proof of proposition \ref{propcombcsh}. Note that, if $T$ is a rooted tree and $\boldsymbol{e} \models E(T)$, $Cont_{\boldsymbol{e}}(T)$ is a tree and $Part_{\boldsymbol{e}}(T)$ is a forest. So there is possibly contractions for the product $[\cdot , \cdot ]$ to the left of $(\Phi \otimes \Phi) \circ \tdelta_{\mathbf{H}_{CK}^{\mathcal{D}}} (T)$. Remark that
\begin{itemize}
\item the trees of $Part_{\boldsymbol{e}}(T)$ are disconnected components of $T$ and they appear to the left of $(\Phi \otimes \Phi) \circ \tdelta_{\mathbf{H}_{CK}^{\mathcal{D}}} (T)$,
\item the edges of $\overline{\boldsymbol{e}}$ between two disconnected components of $Part_{\boldsymbol{e}}(T)$ in $T$ are edges of $Cont_{\boldsymbol{e}}(T)$ and thus they appear to the right of $(\Phi \otimes \Phi) \circ \tdelta_{\mathbf{H}_{CK}^{\mathcal{D}}} (T)$.
\end{itemize}
We deduce formula (\ref{formulecombccsh}).
\end{proof}
\\

{\bf Remark.} {In the expression of $\Phi(F)$ (formula (\ref{formulecombccsh})), we find the terms of (\ref{formulecombcsh}) and other terms with contractions for the product $\left[ \cdot , \cdot \right]$. Taking $\left[ \cdot , \cdot \right] = 0$, we obtain (\ref{formulecombcsh}}) again.
\\

{\bf Examples.} {From the examples at the end of section \ref{csh}, we give the other terms with contractions for the product $\left[ \cdot , \cdot \right]$.
\begin{itemize}
\item There are no terms with contractions for the following trees: $ \addeux{$a$}, \adtroisun{$b$}{$a$}, \adtroisdeux{$a$}{$b$}, \adquatreun{$c$}{$b$}{$a$}, \adquatrequatre{$a$}{$c$}{$b$} $.
\item For the tree $\adquatredeux{$b$}{$a$}{$c$}$,
\begin{eqnarray*}
\Phi (\adquatredeux{$b$}{$a$}{$c$}) =  \hdots + \left[ \varphi (\addeux{$b$}) \varphi (\addeux{$c$}) \right] \varphi (\addeux{$a$})
\end{eqnarray*}
\item For the tree $ \adcinqsix{$a$}{$b$}{$c$}{$d$} $,
\begin{eqnarray*}
\Phi (\adcinqsix{$a$}{$b$}{$c$}{$d$}) & = & \hdots + \left[ \varphi (\addeux{$b$}) \varphi (\addeux{$c$}) \right] \varphi (\adtroisdeux{$a$}{$d$}) + \left[ \varphi (\addeux{$b$}) \varphi (\addeux{$d$}) \right] \varphi (\adtroisdeux{$a$}{$c$}) + \left[ \varphi (\addeux{$b$}) \varphi (\adtroisun{$d$}{$c$}) \right] \varphi (\addeux{$a$}) \\
& & + \varphi (\addeux{$c$}) \left[ \varphi (\addeux{$b$}) \varphi (\addeux{$d$}) \right] \varphi (\addeux{$a$}) + \left[ \varphi (\addeux{$b$}) \varphi (\addeux{$d$}) \right] \varphi (\addeux{$a$}) \varphi (\addeux{$c$}) \\
& & + \left[ \varphi (\addeux{$b$}) \varphi (\addeux{$d$}) \right] \varphi (\addeux{$c$}) \varphi (\addeux{$a$}) + \left[ \varphi (\addeux{$b$}) \varphi (\addeux{$c$}) \right] \varphi (\addeux{$a$}) \varphi (\addeux{$d$}) \\
& & + \left[ \varphi (\addeux{$b$}) \varphi (\addeux{$c$}) \right] \varphi (\addeux{$d$}) \varphi (\addeux{$a$}) + \varphi (\addeux{$d$}) \left[ \varphi (\addeux{$b$}) \varphi (\addeux{$c$}) \right] \varphi (\addeux{$a$}) .
\end{eqnarray*}
\end{itemize}
}

\nocite{*}

\bibliographystyle{amsalpha} 
\bibliography{ref}

\providecommand{\bysame}{\leavevmode\hbox to3em{\hrulefill}\thinspace}
\providecommand{\MR}{\relax\ifhmode\unskip\space\fi MR }
\providecommand{\MRhref}[2]{%
  \href{http://www.ams.org/mathscinet-getitem?mr=#1}{#2}
}
\providecommand{\href}[2]{#2}
\begin{thebibliography}{CEFM11}

\bibitem[CEFM11]{Calaque11}
D.~Calaque, K.~Ebrahimi-Fard, and D.~Manchon, \emph{Two interacting {H}opf
  algebras of trees}, Advances in Appl. Math. \textbf{47} (2011), 282--308,
  arXiv:0806.2238.

\bibitem[CK98]{Connes98}
A.~Connes and D.~Kreimer, \emph{Hopf algebras, {R}enormalization and
  {N}oncommutative geometry}, Comm. Math. Phys. \textbf{199} (1998), 203--242,
  arXiv:hep-th/9808042.

\bibitem[CK00]{Connes00}
\bysame, \emph{Renormalization in quantum field theory and the
  {R}iemann-{H}ilbert problem {I}. the {H}opf algebra of graphs and the main
  theorem}, Comm. Math. Phys. \textbf{210} (2000), 249--273,
  arXiv:hep-th/9912092.

\bibitem[CK01]{Connes01}
\bysame, \emph{Renormalization in quantum field theory and the
  {R}iemann-{H}ilbert problem {II}. $ \beta $-function, diffeomorphisms and the
  renormalization group}, Comm. Math. Phys. \textbf{216} (2001), 215--241,
  arXiv:hep-th/0003188.

\bibitem[DHT02]{Duchamp00}
G.~Duchamp, F.~Hivert, and J.-Y. Thibon, \emph{Noncommutative symmetric
  functions {VI}: Free quasi-symmetric functions and related algebras}, Int. J.
  Algebra Comput. \textbf{12} (2002), no.~671, arXiv:math/0105065.

\bibitem[EV04]{Ecalle04}
J.~Ecalle and B.~Vallet, \emph{The arborification-coarborification transform:
  analytic, combinatorial, and algebraic aspects}, Ann. Fac. Sc. Toulouse
  \textbf{13} (2004), no.~4, 575--657.

\bibitem[Foi02a]{FoissyI02}
L.~Foissy, \emph{Les algèbres de {H}opf des arbres enracin{\'e}s
  d{\'e}cor{\'e}s, {I}}, Bull. Sci. Math. \textbf{126} (2002), no.~3, 193--239,
  arXiv:math/0105212.

\bibitem[Foi02b]{FoissyII02}
\bysame, \emph{Les algèbres de {H}opf des arbres enracin{\'e}s d{\'e}cor{\'e}s,
  {II}}, Bull. Sci. Math. \textbf{126} (2002), no.~4, 249--288,
  arXiv:math/0105212.

\bibitem[Foi12]{Foissy11}
\bysame, \emph{Ordered forests and parking functions}, Int. Math. Res. Notices
  \textbf{2011} (2012), arXiv:1007.1547.

\bibitem[Foi13]{Foissy13}
\bysame, \emph{The {H}opf algebra of {F}liess operators and its dual pre-{L}ie
  algebra}, arXiv:1304.1726.

\bibitem[FU10]{Foissy10}
L.~Foissy and J.~Unterberger, \emph{Ordered forests, permutations and iterated
  integrals}, preprint (2010), arXiv:1004.5208.

\bibitem[GL90]{Grossman90}
R.L. Grossman and R.G. Larson, \emph{Hopf-algebraic structure of combinatorial
  objects and differential operators}, Israel J. Math. \textbf{72} (1990),
  no.~1--2, 109--117, arXiv:0711.3877.

\bibitem[Hof00]{Hoffman00}
M.E. Hoffman, \emph{Quasi-shuffle products}, Journal of Algebraic Combinatorics
  \textbf{11} (2000), 49--68, arXiv:math/9907173.

\bibitem[Hol03]{Holtkamp03}
R.~Holtkamp, \emph{Comparaison of {H}opf algebras on trees}, Arch. Math.
  (Basel) \textbf{80} (2003), no.~4, 368--383.

\bibitem[LR06]{Loday05}
J.L. Loday and M.~Ronco, \emph{On the structure of cofree hopf algebra}, J.
  reine angew. Math. \textbf{592} (2006), 123--155, arXiv:math/0405330.

\bibitem[LR10]{Loday10}
\bysame, \emph{Combinatorial hopf algebras}, Amer. Math. Soc., Quanta of maths
  (2010), no.~11, 347--383, arXiv:0810.0435.

\bibitem[LV12]{LodayV}
J.L. Loday and B.~Vallette, \emph{Algebraic {O}perads}, Springer, 2012.

\bibitem[Moe01]{Moerdijk01}
I.~Moerdijk, \emph{On the {C}onnes-{K}reimer construction of {H}opf algebras},
  Contemp. Math. \textbf{271} (2001), 311--321, arXiv:math-ph/9907010.

\bibitem[MR95]{Malvenuto05}
C.~Malvenuto and C.~Reutenauer, \emph{Duality between quasi-symmetric functions
  and the {S}olomon descent algebra}, J. Algebra \textbf{177} (1995), no.~3,
  967--982.

\bibitem[MS11]{Manchon08}
D.~Manchon and A.~Saïdi, \emph{Lois pré-{L}ie en interaction}, Communications
  in Algebra \textbf{39} (2011), 3662--3680, arXiv:0811.2153.

\bibitem[NT06]{Novelli06}
J.-C. Novelli and J.-Y. Thibon, \emph{Polynomial realizations of some
  trialgebras}, arXiv:math/0605061.

\bibitem[Reu93]{Reutenauer93}
C.~Reutenauer, \emph{Free {L}ie algebras}, The Clarendon Press, 1993.

\bibitem[Slo]{Sloane}
N.J.A. Sloane, \emph{On-line encyclopedia of integer sequences},
  http://www.research.att.com/ $\tilde{\:}$njas/ sequences/ Seis.html.

\bibitem[Sta02]{Stanley97}
R.P. Stanley, \emph{Enumerative combinatorics. {V}ol. 1}, Cambridge University
  Press, 2002.

\end{thebibliography}

\end{document}